%% file: article.tex
\documentclass[a4paper]{amsart}

\RequirePackage{pantone}
\RequirePackage{amssymb}
\RequirePackage{amsthm}
\RequirePackage{mathtools}
\RequirePackage{mathrsfs} % Curly maths symbols
\usepackage{aliascnt}
\RequirePackage{wasysym} % Create new type of integral symbol
\RequirePackage{manfnt} % Allow \dbend for Bourbaki ``danger''
\RequirePackage[maxbibnames=99]{biblatex}
\RequirePackage[normalem]{ulem} % Allows underline
\RequirePackage{float} % allows pictures to stay where they are
\RequirePackage{enumerate} % Allows styling of enumeration
\RequirePackage{enumitem}
\RequirePackage{multirow} % Allow multiple rows
\RequirePackage{xparse} % Allows optional argument definitions
\RequirePackage{xargs} 
\RequirePackage{caption} % removes colon after caption if caption is empty
\RequirePackage{array} % Allows changing of array font colors
\RequirePackage{ragged2e} % Allows \RaggedRight which allows Hyphenation
\RequirePackage{hyperref}
\usepackage{rotating}

\colorlet{ablue}{Blue072}

\RequirePackage[a4paper, left=25mm, right=25mm, top=25mm, bottom=25mm]{geometry}

\RequirePackage{todonotes} % Adds todos

\RequirePackage{tikz}
\RequirePackage{tkz-graph}
\RequirePackage{tkz-euclide}
\usetikzlibrary{decorations, decorations.markings, arrows.meta}

\newtheoremstyle{thmstyle}{}{}{\itshape}{}{\bfseries}{ }{5pt}{}
\newtheoremstyle{exstyle}{}{}{}{}{\bfseries}{ }{5pt}{}
\newtheoremstyle{defstyle}{}{}{}{}{\bfseries}{ }{5pt}{}
\newtheoremstyle{remstyle}{}{}{}{}{\bfseries}{ }{5pt}{}
\newtheoremstyle{proofstyle}{}{}{}{}{\bfseries}{ }{5pt}{}

\theoremstyle{thmstyle}
\newtheorem{thm}{Theorem}[section]
\newtheorem{theorem}[thm]{Theorem}
\def\thmautorefname~#1\null{Theorem~#1\null}

\newaliascnt{lemma}{thm}
\newtheorem{lemma}[lemma]{Lemma}
\aliascntresetthe{lemma}
\def\lemmaautorefname~#1\null{Lemma~#1\null}

\newaliascnt{proposition}{thm}

\aliascntresetthe{proposition}
\def\propositionautorefname~#1\null{Proposition~#1\null}

\newaliascnt{corollary}{thm}

\aliascntresetthe{corollary}
\def\corollaryautorefname~#1\null{\corollaryname~#1\null}

\def\equationautorefname~#1\null{\text{(#1)\null}}
\def\figureautorefname~#1\null{Figure~#1\null}
\def\itemautorefname~#1\null{(#1)\null}
\def\sectionautorefname~#1\null{Section~#1\null}

\def\subsubsectionautorefname~#1\null{Subsubsection~#1\null}

\theoremstyle{exstyle}
\newaliascnt{example}{thm}
\newtheorem{example}[example]{Example}
\newtheorem{examples}[example]{Examples}
\aliascntresetthe{example}
\def\exampleautorefname~#1\null{Examples~#1\null}

\theoremstyle{defstyle}
\newaliascnt{definition}{thm}

\aliascntresetthe{definition}
\def\definitionautorefname~#1\null{Definition~#1\null}

\theoremstyle{remstyle}
\newaliascnt{remark}{thm}
\newtheorem{remark}[remark]{Remark}

\aliascntresetthe{remark}
\def\remarkautorefname~#1\null{Remark~#1\null}

\newtheorem*{openprob}{Open Problems}

\theoremstyle{proofstyle}

\renewcommand{\b}[1]{\mathbf{#1}} % bold letters
\providecommand\defn[1]{{\emph{#1}}}
\newcommand\mcA{\mathcal{A}}
\newcommand\msR{\mathscr{R}}
\newcommand\msB{\mathscr{B}}
\newcommand{\R}{\mathbb{R}} % real numbers
\newcommand{\N}{\mathbb{N}} % naturals
\newcommand{\Z}{\mathbb{Z}} % integers
\newcommand{\field}{\mathbb{K}}

\newcommand{\st}{\;\mid\;}
\newcommand{\bs}{\backslash}
\newcommand{\ie}{\textit{i.e.},~} % id est
\newcommand{\set}[2]{\left\{#1\st#2\right\}} % set notation
\newcommand{\gen}[1]{\left<#1\right>}
\newcommand{\genset}[2]{\gen{#1\st#2}}
\newcommand{\iso}{\cong} % isomorphism
\DeclarePairedDelimiter\abs{\lvert}{\rvert}%
\newcommand{\order}[1]{\abs{#1}}

\newcommand{\roots}{\Phi}
\newcommand{\proots}{\roots^+}

\DeclareMathOperator{\wt}{wt}
\DeclareMathOperator{\sgn}{sgn}
\DeclareMathOperator{\Shi}{Shi}
\newcommand{\shiarr}{\Shi_{\Phi}}
\newcommand{\weylarr}{\mcA_{\Phi}}
\DeclareMathOperator{\nar}{N}
\DeclareMathOperator{\poin}{Poin}

\newcommand{\bns}[4]{\binom{#1}{#2} - \binom{#3}{#4}}
\newcommand{\nabovediag}[4]{%
    \binom{#3 + #4 - #1 - #2}{#4 - #2} - \binom{#3 + #4 - #1 - #2}{#4 - #1 + 1}
}
\newcommand{\weakabove}[2]{%
    \gamma\left(\! #1 \xrightarrow{y\geq x} #2 \!\right)
}
\newcommand{\Fdiag}[2]{%
    \gamma\left(\! #1 \xrightarrow{y\geq x} #2^{D\Sigma} \!\right)
}
\newcommand{\weakaboveNar}[2]{%
	\gamma'\left( #1 \xrightarrow{y\geq x} #2 \right)
}

\makeatletter

\newcommand\Autoref[1]{\@first@ref#1,@}
\def\@throw@dot#1.#2@{#1}% discard everything after the dot
\def\@set@refname#1{%    % set \@refname to autoefname+s using \getrefbykeydefault
    \edef\@tmp{\getrefbykeydefault{#1}{anchor}{}}%
    \xdef\@tmp{\expandafter\@throw@dot\@tmp.@}%
    \ltx@IfUndefined{\@tmp autorefnameplural}%
         {\def\@refname{\@nameuse{\@tmp autorefname}s}}%
         {\def\@refname{\@nameuse{\@tmp autorefnameplural}}}%
}
\def\@first@ref#1,#2{%
  \ifx#2@\autoref{#1}\let\@nextref\@gobble% only one ref, revert to normal \autoref
  \else%
    \@set@refname{#1}%  set \@refname to autoref name
    \@refname~\ref{#1}% add autoefname and first reference
    \let\@nextref\@next@ref% push processing to \@next@ref
  \fi%
  \@nextref#2%
}
\def\@next@ref#1,#2{%
   \ifx#2@ and~\ref{#1}\let\@nextref\@gobble% at end: print and+\ref and stop
   \else, \ref{#1}% print  ,+\ref and continue
   \fi%
   \@nextref#2%
}
\makeatother

\title{Enumerating regions of Shi arrangements per Weyl cone}
\author[A. Dermenjian]{Aram Dermenjian}
\address[A. Dermenjian]{University of Manchester}
\email[A. Dermenjian]{aram.dermenjian.math@gmail.com}
\thanks{The first author was funded by Heilbronn Institute for Mathematical Research and the Dame Kathleen Ollerenshaw Trust}
\author[E. Tzanaki]{Eleni Tzanaki}
\address[E. Tzanaki]{University of Crete, Greece}
\email[E. Tzanaki]{etzanaki@uoc.gr}
\thanks{The second author was partially supported by the Hellenic Foundation for Research and Innovation (H.F.R.I.). 
	Project Number: HFRI-FM20-0453}

    \subjclass[2020]{20F55,52C35,05C20,05C38,05C30,05C22,14N10}
    \keywords{Weyl groups, Coxeter groups, Hyperplane arrangements, Directed graphs, Shi arrangements, Root poset}
\begin{document}

\maketitle

\begin{abstract}
    Given a Shi arrangement $\mcA_\Phi$, it is well-known that the total number of
    regions is counted by the parking number of type $\Phi$ and the total number
    of regions in the dominant cone is given by the Catalan number of type
    $\Phi$. In the case of the latter, in \cite{Shi_Number}, Shi gave a bijection
    between antichains in the root poset of $\Phi$ and the regions in the
    dominant cone. This result was later extended by Armstrong, Reiner and Rhoades
    in \cite{ArmstrongReinerRhoades_Parking} where they gave a bijection
    between the number of regions contained in an arbitrary Weyl cone $C_w$ in
    $\mcA_\Phi$ and certain subposets of the root poset. In this article we
    expand on these results by giving a determinental formula for the precise
    number of regions in $C_w$ using paths in certain digraphs related to Shi
    diagrams.
\end{abstract}

\section{Introduction}

Let $\Phi$ be a finite crystallographic root system with positive roots $\Phi^+$ spanning an $n$-dimensional Euclidean vector space $V$ 
equiped with an inner product $\gen{\,,}$.
The \defn{Weyl} (or \defn{Coxeter}) \defn{arrangement} associated to $\Phi$, denoted by $\mcA_{\Phi}$,  is the collection of the hyperplanes 
$\gen{\alpha,x}=0$ for all $\alpha \in \Phi^+$. 
The regions of $\mcA_{\Phi}$ are cones which are in bijection with the elements of the reflection group $W$ 
associated to the root system $\Phi$. 
The \defn{dominant cone} $C_e$ is the  intersection of the 
positive halfspaces $\gen{\alpha,x}>0$ for all $\alpha \in \Phi^+$. 
It is not difficult to see that each cone in $\mcA_{\Phi}$ can be written as $wC_e$ (or $C_w$ for short) for a unique $w\in W$. 
The \defn{Shi arrangement} associated to $\Phi$ is a deformation of the Weyl arrangement consisting of  the hyperplanes $\gen{\alpha,x} \in \left\{ 0,1 \right\}$ for $\alpha \in \Phi^+$.  

There are several  uniform enumerative formulas concerning Weyl and Shi arrangements.
The  most  well known are the  product formulas enumerating the regions in the dominant cone and enumerating all regions of $\shiarr$. More precisely, the number of 
 regions of the Shi arrangement $\shiarr$  is
 given by the \defn{parking number of type $\Phi$}
$$  \prod_{i=1}^{\ell} (e_i+h+1)$$
and the number of regions in the dominant cone of $\shiarr$ is given by the \defn{Catalan number of type $\Phi$}
$$ \prod_{i=1}^{\ell} \frac{e_i+h+1}{e_i+1},$$
where $e_1,\ldots,e_n$ are the exponents and $h$ is the Coxeter number of the Weyl group $W$. 
The above nice product formulas  combined with the fact that 
 the set of regions of the Shi arrangement is partitioned by the cones in $\mcA_{\Phi}$
 naturally leads us to the question
 of enumerating the  regions within each  Weyl cone $C_w$.
 This article answers this question by giving a determinental formula for enumerating the regions of $C_w$.
 
Towards our answer we exploit the well known bijection between  regions in the dominant cone of a Shi arrangement $\shiarr$ and  antichains in the root poset $\Phi^+$ given by Shi in \cite{Shi_Number} and its  generalization  which relates regions in an arbitrary cone $C_w$ of $\shiarr$ to the antichains of a certain subposet
$\Phi^+_w$   of the root poset $\Phi^+$ given by Armstrong, Reiner and Rhoades in \cite{ArmstrongReinerRhoades_Parking}.
In this setup our objective becomes the enumeration of antichains in each 
 subposet $\Phi^+_w$. 
 Next, we make use of Shi's diagrams $\Lambda_X$, which are essentially a way
   to arrange the positive roots  of $\Phi^+$ 
 in certain arrays of boxes (possibly overlapping) so that the partial order in $\Phi^+$ is nicely visualized. 
 We show that  one can construct an acyclic directed graph $\Gamma_X$ 
 with source $I$ and sink $F$,  
 whose (directed) paths from $I$ to $F$ are in bijection with antichains in the root poset $\Phi^+$. 
 When restricted to the subposet $\Phi^+_w$, the antichains  are in bijection with  paths from $I$ to $F$ which {\em do not} contain 
 certain subpaths depending on $w$. 
 Therefore, the general situation is the following: we have an acyclic 
 directed
 graph $\Gamma$,  a collection $\Pi=\{\pi_1,\ldots,\pi_n\}$ of  subpaths of $\Gamma$ and 
 we want to enumerate all paths from the source $I$ to  the since $F$ of $\Gamma$ which do no 
 contain any of the subpaths in $\Pi$. 
 It is easily understood that this is a problem of inclusion-exclusion and it turns out 
 that, under suitable conditions on the paths of the set $\Pi$, the answer can be expressed in the form of a determinant.

Our first main theorem, which is proven in  \autoref{sec:non-overlapping-paths}, is  a more general statement which applies to any  acyclic directed 
graph $\Gamma$ and any collection $\{\pi_1,\ldots,\pi_n\}$ of \defn{non-overlapping} subpaths of $\Gamma$ (see 
\autoref{sec:non-overlapping-paths} for undefined terms).  
 In the following statement, we  denote by  $\gamma(A\to B)$  the number of paths from $A$ to $B$ in the graph $\Gamma$.

\begin{theorem}
	\label{thm:path-nonoverlapping_intro}
	Let $I$ and $F$ be two arbitrary vertices in an acyclic digraph $\Gamma$. 
	Let $\Pi = \left\{ \pi_1,  \ldots, \pi_n \right\}$ be a collection of non-overlapping paths
	and 	$I_i, F_i$ be the initial and final point of each  subpath $\pi_i$. 
	Then, 
	the number of paths in $\Gamma$ from $I$ to $F$ which do not contain any of the subpaths in $\Pi$ is 
	equal to the determinant 
%	$\det(M_{\Pi})$ of the following matrix 
	\[
 \begin{vmatrix}
		1 & \gamma{(F_2 \to I_1)} & \cdots & \gamma{(F_{n} \to I_1)} & \gamma{(I \to I_1)}\\
		\gamma{(F_1 \to I_2)} & 1 & \cdots & \gamma{(F_n\to I_2)} & \gamma{(I \to I_2)}\\
		\vdots & \vdots & \ddots & \vdots & \vdots\\
		\gamma{(F_1 \to I_{n})} & \gamma(F_2 \to I_n) & \cdots & 1 & \gamma{(I \to I_n)}\\
		\gamma{(F_1 \to F)} & \gamma{(F_2 \to F)} & \cdots & \gamma{(F_{n} \to F)} & \gamma{(I \to F)}\\
	\end{vmatrix}. 
	\]
\end{theorem}
\noindent
We also present a slightly more general version of the above theorem, where we allow edges to have weights (see \autoref{thm:nonoverlapping}).
\medskip

%\autoref{sec:Root posets as digraphs}

Our next main result is the application of \autoref{thm:path-nonoverlapping_intro} for the  enumeration of regions in each Weyl cone $C_w$. 
More precisely, for each $ X \in \{ A_n,B_n, C_n, D_n, E_6, F_4, G_2\} $ we describe the construction of  the corresponding  digraph  $\Gamma_X$ (see \autoref{sec:root_posets_as_digraphs}) where the digraphs of type $E_7$ and $E_8$ are left as open problems.
As we mentioned earlier, the  digraph graph $\Gamma_X$ has the property that the paths from $I$ to $F$  biject to  antichains in the root poset $\Phi^+$. 
Moreover, each  positive root $\alpha \in \Phi^+$  corresponds 
uniquely to a length two subpath of $\Gamma_X$, represented by a {\em corner} in  $\Gamma_X$.
When  we are restricted to the subposet $\Phi^+_w$, 
 we need to count antichains which do not contain a certain subset $N(w^{-1})$ of positive roots
 where $N(w^{-1})$ is the {\em inversion set} of $w^{-1}$.
 Thus, the antichains  in $\Phi^+_w$  are in bijection with  paths from $I$ to $F$ in $\Gamma_X$ which {\em do not} contain certain corners. 
 Since   any subset of corners is a set of non-overlapping subpaths,  our  \autoref{thm:path-nonoverlapping_intro} is applicable. 

 The graphs for $X=A_n$ and $B_n$ are quite natural to construct from the corresponding root poset $\Phi^+$ (see \autoref{sssec:types_A_and_B})  and, as we show in \Autoref{ssec:type-A-counting,ssec:type-B-counting}, it is possible to give a nice precise
formula for the entries $\gamma(A\to B)$ of the determinant in the enumeration. 
In type $D$ however, due to the complexity of the root poset, the graph  $\Gamma_{D_n}$ is much more complicated.  In  \autoref{ssec:type-D-counting} we describe its construction and, as we did for the other two cases, we give a formula for  the entries  $\gamma(A\to B)$ of the corresponding determinant (see \autoref{lem:type_D_counting}). 
Unfortunately, the formula is complicated and has to be split in several distinct cases.

We conclude our paper with \autoref{sec:narayana_numbers}, where we present an appropriate modification of our main theorem, which counts the number of regions in each cone $C_w$ according to 
their number of separating walls (or equivalently the number of antichains in each subposet $\Phi_w^+$
according to their cardinality) .
More precisely, we give  a determinental formula whose entries  are generating polynomials
counting paths  on their number of corners. 
When restricted to the dominant cone, the expansion of the determinant returns the Narayana polynomial, whereas in an arbitrary cone $C_w$ we get  the Poincar\'e polynomial of the cone $C_w$ in the sense of \cite{DBStump}.

\section{Shi arrangements}
In this section we give background and set up notation on Shi arrangements.
For further and more in-depth details, the interested reader is directed to \cite{Humphreys}.

\subsection{Weyl Arrangements}
Let $\R^n$ be an $n$-dimensional (real) Euclidean vector space
equipped with an inner product $\gen{\,,}$.
A \defn{hyperplane}  is a   codimension $1$ affine subspace of $\R^n$. 
For $\alpha \in \R^n$ and $k\in \R$ we use the notation 
$H_{\alpha,k} = \set{v \in \R^n}{\gen{\alpha, v} = k}$ and often abbreviate  $H_{\alpha,0}$ to $H_{\alpha}$.
A  \defn{hyperplane arrangement}, or \defn{arrangement} for short, is a  finite collection $\mcA$ of hyperplanes  in $\R^n$. 
A \defn{subarrangement} of an arrangement $\mcA$ is a subset of $\mcA$.
If all the hyperplanes in $\mcA$ are linear (\ie they pass through the origin) then we say that $\mcA$  is \defn{central}.

To each  hyperplane   $H_{\alpha, 0}$, we associate a reflection $s_{\alpha}$ which fixes pointwise   $H_{\alpha, 0}$  and sends $\alpha$ to $-\alpha$.
Similarly, to any hyperplane $H_{\alpha,k}$ we associate a reflection $s_{\alpha, k}$ fixing $H_{\alpha,k}$ pointwise.
A \defn{root system $\Phi$} in $\R^n$ is a finite collection of nonzero vectors  called \defn{roots},
satisfying the  following conditions:
\begin{enumerate}
    \item $\Phi \cap \R\alpha = \left\{ \alpha, -\alpha \right\}$ for every $\alpha \in \Phi$, and
    \item $s_\alpha\left( \Phi \right) = \Phi$ for every $\alpha \in \Phi$.
    \item $\gen{\alpha, \beta} \in \Z$ for all $\alpha, \beta \in \Phi$.
\end{enumerate}
The set  $\Phi$  can be decomposed into the subsets of  \defn{positive} $\Phi^+$ and \defn{negative roots} $\Phi^-$ respectively.  Given such a decomposition, the set  $\Delta$ of \defn{simple roots} is  the smallest subset of $\Phi^+$ such that every positive root is a positive linear combination of elements in $\Delta$. 
Let  $W$ be  the group generated by the reflections  $s_{\alpha}$ for $ \alpha \in\Phi$. 
We say that $W$ is the \defn{Weyl group} associated to the root system $\Phi$.

The \defn{Weyl arrangement} associated to a root system $\Phi$ is the
 central hyperplane arrangement whose hyperplanes are  normal to the  roots in $\Phi$, \ie $\weylarr = \set{H_{\alpha}}{\alpha \in \Phi^+}$.
In this case, the reflections $s_\alpha$ for $\alpha\in \Phi^+$ form the reflection group $W$ 
associated to the root system $\Phi$.
It is well-known that finite irreducible Weyl groups are classified into a finite number of types: $A_n$, $B_n$, $C_n$, $D_n$,
 being the four infinite families and $E_6$, $E_7$, $E_8$, $F_4$ and $G_2$ being the exceptional types.
We say that a Weyl arrangement is \defn{of type $X$} if its associated Weyl group is a type $X$ Weyl group where $X$ is one of the $9$ types of the classification of Weyl groups.

Given a root system $\Phi$ of a Weyl group, with simple roots $\Delta$ and roots $\Phi$, there is a nice partial order on the set of positive roots which we describe next.
Given two roots $\alpha, \beta \in \proots$ then we say $\alpha < \beta$ if and only if $\beta - \alpha \in \N\Delta$.
This gives us what is known as the \defn{root poset} of $\Phi^+$ which we denote by $(\Phi^+, \leq)$.
We will use this poset extensively throughout this article.

\subsection{Inversions and regions}
Given a Weyl arrangement $\weylarr = \set{H_{\alpha, 0}}{\alpha \in \Phi^+}$, we let $W$ be its associated Weyl group and $S = \set{s_{\alpha}}{\alpha \in \Delta}$ be the set of \defn{simple reflections} and let $T = \bigcup_{w \in W} wSw^{-1}$, the conjugates of $S$, be the set of \defn{reflections}.
The elements of $W$ can be represented as a word over the alphabet $S$, in other words, $w = s_1 \ldots s_m$ for $s_i \in S$.
The \defn{length $\ell(w)$} of an element $w \in W$ is the length of the shortest such representation.
For each element, we can also associate a set of positive roots.
The \defn{(left) inversion set of $w$} is given by
\[
    N(w) = \Phi^+ \cap w\left( \Phi^- \right) = \set{\alpha \in \Phi^+}{\ell(t_{\alpha}w) < \ell(w),\,t_\alpha \in T}.
\]
It is well known that $\ell(w) = \order{N(w)}$.
Inversion sets have a natural description in terms of the Weyl arrangement as well.

Given a Weyl arrangement $\weylarr$, the \defn{regions} of the arrangement are the connected components of $\R^n \bs \weylarr$.
Without loss of generality, we may assume that there is a unique region called the \defn{base region}, which we denote by $B$,  which is the intersection of the positive half-spaces of all hyperplanes.
Fixing $B$, we define the \defn{separation set $S(R)$} for a region to be the set of hyperplanes which separate $B$ and $R$.
Then for each $S(R)$ there is a unique $w \in W$ such that $N(w) = \set{\alpha}{H_{\alpha,0} \in S(R)}$.
The converse of this is true as well, giving us a bijection between separation sets and inversions.
In other words, to each region of $\weylarr$ we can associate a unique $w \in W$ where the identity is mapped to $B$.

\subsection{Type \texorpdfstring{$A$}{A} - The symmetric group \texorpdfstring{$S_{n+1}$}{Sn}}
One of the best known examples of a Weyl group is the symmetric group.
Recall that the \defn{symmetric group $S_n$} is the group of permutations of $\left[ n \right] = \left\{ 1, 2, \ldots, n \right\}$.
We encode a permutation $\sigma \in S_n$ in one of two ways.
The \defn{one-line notation of $\sigma$} is given by a sequence of numbers $\sigma(1) \sigma(2) \ldots \sigma(n)$ where the number $1$ is sent to $\sigma(1)$, $2$ is sent to $\sigma(2)$, etc.
The \defn{cycle notation of $\sigma$} is given by ordered sets of numbers such as $(ij\ldots k)$ where $i$ is sent to $j$, and $k$ is sent to $i$.

\begin{example}
    As an example let $\sigma \in S_4$ be the permutation which sends $1$ to $2$, $2$ to $4$, $3$ to $3$ and $4$ to $1$.
    Then the one-line notation $\sigma$ is given by $2431$ and the cycle notation of $\sigma$ is given by $(124)(3)$.
\end{example}

We say that a cycle $(a_1, \ldots, a_k)$ has \defn{length $k$}.
A cycle of length $2$ is known as a \defn{transposition}.
It is well-known that every permutation can be written as a product of transpositions.
The \defn{sign $\sgn(\sigma)$} of a permutation $\sigma = t_1 \ldots t_n$ is given by the formula $\sgn(\sigma) = (-1)^n$ where the $t_i$ are transpositions.

The symmetric group $S_n$ is the Weyl group of type $A_{n-1}$ as it is generated by $n-1$ transpositions.
In particular, it is given by the presentation $S_n = \genset{(i, i+1)}{i \in \left[ n-1 \right]}$.
These adjacent transpositions $(i, i+1)$ are the set of simple reflections.
The set of all transpositions $(i, j)$ is the set of reflections.
The hyperplane arrangement associated to $S_n$ is the arrangement whose hyperplanes are defined by the equations $x_i = x_j$ for $i, j \in [n]$ and where $i \neq j$.
To each reflection $(i, j)$ we associate the (positive) root $e_i - e_j$ (where $i < j$).

\begin{example}
    \label{ex:A2}
    We take a moment now to give an example of all definitions used in the type $A_2$ Weyl group (which is the symmetric group $S_3$).
    Our two simple reflections are given by the set $S = \left\{ (1,2), (2,3) \right\}$ with a third non-simple reflection given by the transposition $(1,3)$.
    The hyperplanes for $A_2$ are then defined by the simple roots $\alpha_1 = e_1 - e_2$ and $\alpha_2 = e_2 - e_3$.
    The third (positive) root is then given by $\alpha_1 + \alpha_2 = e_1 - e_3$.
    Taking the negatives of all of these roots gives us the root system for $A_2$.
    The root poset of $A_2$ is then depicted using the following Hasse diagram.
    
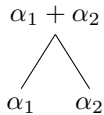
\begin{center}
        \begin{tikzpicture}[scale=0.3]
            \node (a1) at (0,0) {$\alpha_1$};
            \node (a2) at (3,0) {$\alpha_2$};
            \node (a3) at (1.5,4) {$\alpha_1 + \alpha_2$};
            \draw (a1.north) -- (a3.south);
            \draw (a2.north) -- (a3.south);
        \end{tikzpicture}
    \end{center}

    Although these hyperplanes live in $\R^3$, it is well-known that we can project down into $\R^2$ to get the following hyperplane arrangement (known as the Weyl arrangement of type $A_2$):
    
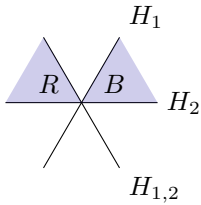
\begin{center}
        \begin{tikzpicture}
            \draw (-1,0) -- (1,0);
            \node[right] at (1, 0) {$H_2$};
            \draw (-0.5, -0.866025403784439) -- (0.5, 0.866025403784439);
            \node[above right] at (0.5, 0.866025403784439) {$H_1$};
            \draw (0.5, -0.866025403784439) -- (-0.5, 0.866025403784439);
            \node[below right] at (0.5, -0.866025403784439) {$H_{1,2}$};
            \fill[ablue, opacity=0.2] (0,0) -- (1, 0)  -- (0.5, 0.866025403784439);
            \node at (0.433012701892219,0.25) {$B$};
            \fill[ablue, opacity=0.2] (0,0) -- (-1, 0)  -- (-0.5, 0.866025403784439);
            \node at (-0.433012701892219, 0.25) {$R$};
        \end{tikzpicture}
    \end{center}

    In the figure, we've shaded two regions.
    In the top right we set (and fixed) an arbitrary region $B$ as the base region and we labelled the region in the top left by an $R$.
    The separation set of $R$ is given by $S(R) = \left\{ H_1, H_{1,2} \right\}$ where $H_1$ is the hyperplane associated to $\alpha_1$ and $H_{1,2}$ is the hyperplane associated to $\alpha_1 + \alpha_2$.
    The element in $A_2$ associated to $R$ is represented by the word $s_1s_2$ (denoting that we first reflect $B$ over $H_2$ and then over $H_1$).
    Calculating the inversion set for $s_1s_2$ we see that $N(s_1s_2) = \left\{ \alpha_1, \alpha_{1} + \alpha_2 \right\}$ as we would expect.
\end{example}

\subsection{Shi Arrangements}
The \defn{Shi arrangement $\shiarr$} corresponding to the root system $\Phi$
consists  of the hyperplanes
\[
    \shiarr = \set{H_{\alpha, k}}{\alpha \in \proots,\,k \in \left\{ 0,1 \right\}}
\]
The Shi arrangement $\shiarr$ consists of the Weyl arrangement $\weylarr$ together with a
positive unit translate of each hyperplane in $\weylarr$. 

Let $\shiarr$ be the Shi arrangement of type $X$ and let $\weylarr$ be the subarrangement of $\shiarr$ associated to the Weyl arrangement of type $X$.
As before, we look at the connected components of $\R^n \bs \shiarr$ which we call the \defn{regions of $\shiarr$} and denote it by $\msR$.
Let $C_e$ be the cone of the intersection of the positive half-spaces of all hyperplanes in the Weyl subarrangement of $\shiarr$, \ie $C_e = \cap_{H \in \weylarr} H^+$.
We call $C_e$ the \defn{dominant cone} of $\shiarr$ and it is associated to the identity of the Weyl group.
Reflecting this cone over the hyperplanes in $\weylarr$, we get a unique cone $C_w$ associated to each $w \in W$ which we call the \defn{Weyl cone} in the Shi arrangment $\shiarr$ associated to $w$.

It is known that the number of regions in the dominant cone are in bijection with the number of antichains in the root poset
where an \defn{antichain} in a poset is a collection of elements which are pairwise incomparable.

\begin{theorem}[{\cite[Theorem 1.4]{Shi_Number}}]
    \label{thm:bij-dom-anti}
    There is a bijection between the number of regions in the dominant cone and the number of antichains in the root poset.
\end{theorem}
This bijection, originally noted by Shi in \cite{Shi_Number}, is given in the following way.
Given an antichain $A$ in $(\Phi^+, \leq)$, let ${I_A = \set{\alpha \in \Phi^+}{\alpha \geq \beta \text{ for some } \beta \in A}}$ be the upper ideal of $A$.
Then $A$ is mapped to the region $R_A$ in the dominant cone where for an arbitrary point $x$ in the relative interior of $R$ we have $\gen{\alpha, x} > 1$ for all $\alpha \in I_A$ and $0 < \gen{\alpha, x} <1$ otherwise.
In other words
\[
    R_A=\set{x}{\gen{\alpha, x} > 1 \text{ if } \alpha \in I_A \text{ and } 0 < \gen{\alpha, x} <1 \text{ otherwise}}
\]
The reverse is obvious.

%\aram{According to the article by Armstrong, Reiner, Rhoades, they were the first people to note that? If you look at the first paragraph in the ``Parking spaces'' paper, it goes a little into how things happened if that helps?}
%\eleni{In "On a refinement the generalized Catalan numbers for Weyl  Groups", Athanasiadis refers to Shi. Look just below of Corol. 3.14 of Ath's paper}
%\aram{I'm an idiot. Yes, the bijection above was first done by Shi. (For some reason my brain thought the theorem below.) I've edited the text to show this.}

In \cite{ArmstrongReinerRhoades_Parking}, Armstrong, Reiner and Rhoades refine the number of regions in the dominant cone.
They give a bijection between the number of regions in an arbitrary Weyl cone and the number of antichains in a certain subposet of the root poset, as stated in the following theorem.
\begin{theorem}[{\cite[Proposition 10.3]{ArmstrongReinerRhoades_Parking}}]
    \label{thm:num-in-cone}
    The number of regions in a Weyl cone $C_w$ is equal to the number of antichains in the subposet of the root poset restricted to $\Phi^+ \bs N(w^{-1})$.
\end{theorem}

This article gives formulas to directly calculate the number of antichains in the subposet of the root poset restricted to $\Phi^+ \bs N(w^{-1})$ using digraphs associated to the root poset.
For this we will need some theory on non-overlapping paths, which we describe in the following section.

\section{Non-overlapping paths and determinants}
\label{sec:determinant}
\label{sec:non-overlapping-paths}
In this section we describe certain families of non-overlapping paths and use the principal of inclusion-exclusion in order to enumerate them. 
These will be used to generate a formula for the enumeration of the number of regions in a Weyl cone.
For more background, the interested reader is directed to the chapter on lattice paths by Krattenthaler \cite{Krattenthaler_LatticePathEnumeration} or to the book by Stanley \cite{Stanley}.

\subsection{Directed graphs}
A \defn{directed graph} (or \defn{digraph} for short) is a graph $\Gamma = (V, E)$ with a set of vertices $V$ and a set of directed edges $E$.
A \defn{path} $\pi = (v_1, e_1, v_2, \ldots, v_{n-1}, e_{n-1}, v_n)$ in $\Gamma$ is an alternating sequence of vertices and edges such that $e_i$ is a directed edge in $E$ from $v_i$ to $v_{i-1}$.
For a given path $\pi$ let $I(\pi) = v_1$ and $F(\pi) = v_n$ be the \defn{initial} and \defn{final} vertices of $\pi$.
If $v_1 = v_n$ then $\pi$ is called a \defn{cycle}.
If $\Gamma$ has no cycles  then we say that $\Gamma$ is \defn{acyclic}.
We assume that all our digraphs are acyclic.

Let $\pi_1$ and $\pi_2$ be two paths.
Given a path $\pi_1 = (v_1, e_1, v_2, \ldots, v_{n-1}, e_{n-1},  v_n)$ we say that $\pi_1$ is a \defn{subpath} of ${\pi_2 = (u_1, f_1, u_2 \ldots, u_{m-1}, f_{m-1}, u_m)}$
if $\pi_1 = (f_i,u_{i+1},\ldots,u_{j},f_j)$ for some $1\leq i<j \leq m$. 
In other words, if the sequence of $\pi_1$ is a subsequence of $\pi_2$.
We say that $\pi_2$ \defn{overlaps} $\pi_1$ if either $\pi_1$ is a subpath of $\pi_2$ or there exists some $i \in \left[ n-1 \right]$ such that for all $j \in \left[ n-i \right]$, then  $e_{i+j-1} = f_{j}$, \ie if the final $i$ edges in $\pi_1$ coincide with the first $i$ edges of $\pi_2$.
Let $\Pi$ be a collection of paths.
Then we say that $\Pi$ is a \defn{non-overlapping collection of paths} if there does not exist any $\pi, \pi' \in \Pi$ such that $\pi$ overlaps $\pi'$.
In other words, we say that $\Pi$ is a non-overlapping collection of paths if no path in $\Pi$ overlaps some other path in $\Pi$.

\begin{example}
    The following are four collections of paths.
    Each collection contains two paths: a dotted red path on the bottom (which we call $\pi_1$), and a dashed blue path on the top (which we call $\pi_2$).
    
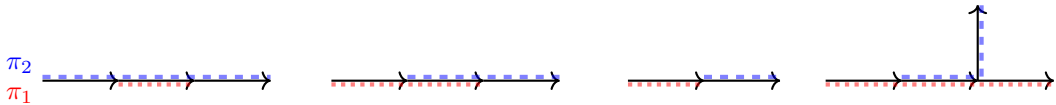
\begin{center}
        \begin{tikzpicture}
            [thick,decoration={markings,mark=at position 1.0 with {\arrow{>}}}]    

            \begin{scope}[shift={(0,0)}]
                \draw[postaction={decorate}] (0,0) -- (1,0);
                \draw[postaction={decorate}] (1,0) -- (2,0);
                \draw[postaction={decorate}] (2,0) -- (3,0);
                \draw[blue, opacity=0.5, dashed, ultra thick] (0,0.05) -- (3,0.05);
                \draw[red, opacity=0.5, dotted, ultra thick] (1,-0.05) -- (2,-0.05);
                \node at (-0.3,- 0.2){ \color{red} $\pi_1$};
                \node at (-0.3, 0.2){ \color{blue} $\pi_2$};
            \end{scope}
            \begin{scope}[shift={(3.8,0)}]
                \draw[postaction={decorate}] (0,0) -- (1,0);
                \draw[postaction={decorate}] (1,0) -- (2,0);
                \draw[postaction={decorate}] (2,0) -- (3,0);
                \draw[red, opacity=0.5, dotted, ultra thick] (0,-0.05) -- (2,-0.05);
                \draw[blue, opacity=0.5, dashed, ultra thick] (1,0.05) -- (3,0.05);
            \end{scope}
            \begin{scope}[shift={(7.7,0)}]
                \draw[postaction={decorate}] (0,0) -- (1,0);
                \draw[postaction={decorate}] (1,0) -- (2,0);
                \draw[red, opacity=0.5, dotted, ultra thick] (0,-0.05) -- (1,-0.05);
                \draw[blue, opacity=0.5, dashed, ultra thick] (1,0.05) -- (2,0.05);
            \end{scope}
            \begin{scope}[shift={(10.3,0)}]
                \draw[postaction={decorate}] (0,0) -- (1,0);
                \draw[postaction={decorate}] (1,0) -- (2,0);
                \draw[postaction={decorate}] (2,0) -- (3,0);
                \draw[postaction={decorate}] (2,0) -- (2,1);
                \draw[red, opacity=0.5, dotted, ultra thick] (0,-0.05) -- (3,-0.05);
                \draw[blue, opacity=0.5, dashed, ultra thick] (1,0.05) -- (2.05,0.05) -- (2.05,1);
            \end{scope}
        \end{tikzpicture}
    \end{center}

    The first graph is an example of overlapping as $\pi_1$ is a subpath of $\pi_2$ implying that $\pi_2$ overlaps $\pi_1$.
    The second graph is another example of overlapping as $\pi_2$ overlaps $\pi_1$ since the first edge of $\pi_2$ and the final edge of $\pi_1$ coincide.
%    \eleni{Second example: The edges are not clear in the figure}
%    \aram{Ah ok. I've tried to put in a fix. I don't want to just use colours just in case people are colour blind. So we want one that is dashed and one that is dotted to help colour-blind people notice the difference too. Is that a little better? We should also use the same arrow-heads that we end up using for all other graphs.}
%    \eleni{ I agree!}
    The final two graphs are non-examples of overlapping (\ie they are non-overlapping) since the final \emph{edges} in $\pi_1$ are not the initial edges in $\pi_2$ and neither path is a subpath of the other.
\end{example}

To each edge $e \in E$ we associate some weight $\wt(e) \in \field$ where $\field$ is a field.
The \defn{weight of a path $\pi$} is then the multiplication of the weights of its edges:
${\wt(\pi) = \prod_{e \in \pi} \wt(e)}$.
Furthermore, the \defn{weight of a set of paths $\Pi$} is the sum of the weights of the paths: ${\wt(\Pi) = \sum_{\pi \in \Pi} \wt(\pi)}$.

\subsection{Non-overlapping paths}
\label{ssec:non-overlapping_paths}
Let $\Gamma$ be an acyclic digraph and
$\Pi = \left\{ \pi_1, \ldots, \pi_n \right\}$ a collection 
of non-overlapping paths in $\Gamma$.
We will use the principle of inclusion-exclusion in order to give a determinental formula for the number of paths 
between two vertices $v_1,v_2$ in $\Gamma$ which do not have any path in $\Pi$ as
a subpath.

We first set some notation.
Let $\Pi = \left\{ \pi_1, \ldots, \pi_n \right\}$ be a collection of paths.
For a given path $\pi_i$ we let $I_i = I(\pi_i)$, $F_i = F(\pi_i)$ be shorthand for the initial and final point of $\pi_i$.
For two vertices $v_1$ and $v_2$ we let $\gamma(v_1 \to v_2) = \wt\left(\set{\pi}{I(\pi) = v_1,\,F(\pi) = v_2}\right)$.
If no such path exists then $\gamma(v_1 \to v_2)$ is equal to $0$ and if $v_1 = v_2$ then $\gamma(v_1 \to v_1) = 1$.
Note that if all the weights are equal to $1$ then $\gamma(v_1 \to v_2)$ is just the number of paths from $v_1$ to $v_2$.

We will also sometimes restrict our paths.
Let
\[
    \gamma(v_1 \to v_2 \st \text{some property}) = \wt\left( \set{\pi}{I(\pi) = v_1,\,F(\pi) = v_2,\,\pi \text{ satisfies some property}} \right).
\]
Our two main examples of this are
\begin{align*}    
    \gamma(v_1 \to v_2 \st  & \text{has no subpath in }  \Pi)  \\ & = \wt\left( \set{\pi}{I(\pi) = v_1,\,F(\pi)=v_2,\,\pi' \text{ is not a subpath of } \pi \text{ for all } \pi'\in \Pi} \right), \text{ and}\\
    \gamma(v_1 \to v_2 \st & \text{has subpath in }   \Pi) \\ &= \wt\left( \set{\pi}{I(\pi) = v_1,\,F(\pi)=v_2,\,\pi' \text{ is a subpath of } \pi \text{ for some } \pi' \in \Pi}\right).
\end{align*}
%\eleni{Shoud we better write "has subpath"?}
%\aram{Yeah! much better}
For ease of notation, we let
\[
    \gamma(v_1 \xrightarrow{\pi} v_2) = \gamma(v_1 \to v_2 \st \text{has subpath in } \left\{\pi\right\}),
\]
which is essentially the  (weighted) number of paths from $v_1$ to $v_2$ which contain the subpath $\pi$.
Additionally, we note that due to the multiplicative nature of weights, we have
\[
    \gamma(v_1 \to v_2 \to v_3) = \gamma(v_1 \to v_2) \cdot \gamma(v_2 \to v_3).
\]

We next prove the main theorem for this section  where all weights are set to $1$.
This allows us to write a (slightly) more easy to read proof.
The proof for arbitrary weights is nearly identical (up to needing to keep track of the weights) which we discuss after.
We also note that by setting the weights equal to $1$, the following theorem counts precisely the number of paths.

\begin{theorem}
    \label{thm:path-nonoverlapping}
    Let $I$ and $F$ be two arbitrary vertices in an acyclic digraph $\Gamma$ where all weights are equal to $1$.
    Let ${\Pi = \left\{ \pi_1, \ldots, \pi_n \right\}}$ be a collection of non-overlapping paths.
 Then 
the number of paths in $\Gamma$ from $I$ to $F$ which do not contain any of the subpaths in $\Pi$ is 
equal to the determinant 
$\det(M_{\Pi})$ of the following matrix 
    \[
	M_{\Pi} = \begin{pmatrix}
            1 & \gamma{(F_2 \to I_1)} & \cdots & \gamma{(F_{n} \to I_1)} & \gamma{(I \to I_1)}\\
            \gamma{(F_1 \to I_2)} & 1 & \cdots & \gamma{(F_n\to I_2)} & \gamma{(I \to I_2)}\\
            \vdots & \vdots & \ddots & \vdots & \vdots\\
            \gamma{(F_1 \to I_{n})} & \gamma(F_2 \to I_n) & \cdots & 1 & \gamma{(I \to I_n)}\\
            \gamma{(F_1 \to F)} & \gamma{(F_2 \to F)} & \cdots & \gamma{(F_{n} \to F)} & \gamma{(I \to F)}\\
        \end{pmatrix}. 
    \]
       In other words, 
    $$\gamma\left( I \to F \st \text{has no subpath in }\Pi \right) = \det(M_{\Pi}). $$
\end{theorem}
%\eleni{we should point out where in the proof we use the non-overlapping condition}
%\aram{Added a paragraph to explain. I guess we forgot to include this paragraph in the proof.}

\begin{proof} 
	First recall that the determinant of any  $(n+1)\times(n+1)$ matrix $A = (a_{i,j})$ is equal to	
	\begin{align}
    	\det(A) = \sum_{\sigma \in S_{n+1}} \left(\sgn(\sigma) \prod_{i \in [n+1]}a_{i, \sigma(i)}\right) \label{dete1}
	\end{align}
	where $S_{n+1}$ is the symmetric group on $n+1$ elements.
	
	Let us write the elements of the $(n+1)\times(n+1)$ matrix $M_{\Pi}$ explicitly:
\begin{equation}
 a_{ij} = \begin{cases} 
 	1 & \text{ for } 1\leq i = j \leq n \\
 	\gamma{(F_j\to I_i)} & \text{ for } 1\leq i \neq j \leq n \\
	\gamma{(F_j \to F)} & \text{ for } 1\leq j \leq n \text{ and } i =n+1\\ 
	\gamma{(I \to I_i)} & \text{ for } 1\leq i \leq n \text{ and } j =n+1\\ 
    \gamma{(I\to F)}  & \text{ for } i =j=n+1.
\end{cases}
\label{elements_aij}
\end{equation}
		If we write each $\sigma \in S_{n+1}$ as a product of cycles,  the determinant in \autoref{dete1} can   be  expressed as 
	\begin{equation}
		\sum_{
			\substack{\sigma = c_1 c_2\cdots c_{\ell}  \\ \sigma \in S_{n+1}}} \left(\sgn(\sigma) 
		\prod_{i\in c_1} a_{i,c_1(i)} \prod_{i \in c_2} a_{i,c_2(i)} \cdots 
		\prod_{i \in c_{\ell}}a_{i, c_{\ell}(i)}
    \right).
		\label{dete2}
	\end{equation}

	We next claim that if $c_j$ is a 
 non-trivial cycle in $S_{n+1}$ not containing $n+1$, the corresponding product in \autoref{dete2} vanishes. 
 Indeed, 
let $c=(i_1 i_2\ldots i_k)$ be a non-trivial cycle (i.e. $k\geq 2$) in $S_{n+1}$ 
with $i_j \neq n+1$ for all $j$. 
  Then
 \begin{align}
 	 \prod_{i \in c} a_{i,c(i)}  & = a_{i_1,i_2} a_{i_2,i_3} \cdots a_{i_{k-1},i_k}  a_{i_{k},i_1} \notag \\
 	  & = 
 	  \gamma{(F_{i_2}\to I_{i_1})}  \gamma{(F_{i_3}\to I_{i_2})} \cdots  
 	   \gamma{(F_{i_k}\to I_{i_{k-1}})} 
 	    \gamma{(F_{i_1}\to I_{i_{k}})} \label{2} \\
 	   	 &=  \gamma{(I_{i_2} \xrightarrow{\pi_{i_2}} F_{i_2}\to I_{i_1})}  \gamma{( I_{i_3} \xrightarrow{\pi_{i_3}} F_{i_3}\to I_{i_2})} \cdots  
 	   	  \gamma{( I_{i_k} \xrightarrow{\pi_{i_k}} F_{i_k}\to I_{i_{k-1}})}
 	   	  \gamma{( I_{i_1} \xrightarrow{\pi_{i_1}} F_{i_1}\to I_{i_{k}})} \label{3}  \\
 	   	& =   \gamma(  I_{i_1} \xrightarrow{\pi_{i_1}} F_{i_1}\to I_{i_{k}} 
 	   	\xrightarrow{\pi_{i_k}} F_{i_k}\to I_{i_{k-1}} \to \cdots \to I_{i_2} \xrightarrow{\pi_{i_2}} F_{i_2}\to I_{i_1} )
\label{4}
        \\
        & = 0,
  \end{align} 
  where,  
  \begin{itemize}[leftmargin=*]
%  	\item  \autoref{1} to \autoref{2} we just used  	
    \item  to go from \autoref{2} to \autoref{3} we used the fact that 
    the number $\gamma{(F_{i_j}\to I_{i_{j-1}})} $
    of paths from $F_{i_j}$ to $I_{i_{j-1}}$     
    remains the same if we prepend the path $\pi_{i_j}$ before $F_{i_j}$, and
  	\item 	the expression in \autoref{4} is the number of paths  which pass through $\pi_{i_1}, \pi_{i_k} \ldots,\pi_{i_2}$,  from $I_{i_1}$ to itself.  Since our digraph is acyclic there
  	do not  exist non trivial paths from a node to itself.  Since  $k\geq 2$ the paths in 
  	\autoref{4} are non trivial,  therefore the above number is 0. 
  \end{itemize}
\medskip
  
%	\end{remark}
Thus, the only terms that survive the expression in \autoref{dete2} are those $\sigma \in S_{n+1}$
whose only possibly non trivial cycle is the one containing $n+1$.
Therefore the sum of the determinant \autoref{dete2} runs over all ${\sigma = (n+1\, i_1 i_2 \cdots i_{k}) \in S_{n+1}}$
written in cycle notation (with the trivial cycles omitted), 
with ${\{i_1,\ldots,i_k\} \subseteq \{1,\ldots,n\}}$. 
The trivial cycles of $\sigma$  will contribute terms of type $a_{ii}=\gamma(I_i\to I_i)$, which are equal to 1 (since the graph is acyclic, the only cycle from a node to itself is the trivial one).
Thus, we rewrite \autoref{dete2} as follows
\begin{align}
 &\sum_{k=0}^{n} 
 \sum_{  \substack{ I \subseteq \{1,\ldots,n\} \\ |I|=k}}
 \sum_{ \substack{ c=(n+1 i_1 i_2 \cdots i_k)    \\  \{i_1,\ldots,i_k\}=I }  }  \sgn{(c)} \prod_{i \in c } a_{i,c(i)} \notag \\
 = \gamma(I\to F)  +	&\sum_{k=1}^{n} 
 \sum_{  \substack{ I \subseteq \{1,\ldots,n\} \\ |I|=k}}
 \sum_{ \substack{ c=(n+1 i_1 i_2 \cdots i_k)    \\  \{i_1,\ldots,i_k\}=I }  }  \sgn{(c)} \prod_{i \in c } a_{i,c(i)} 	\label{dete3}
\end{align}

Let us now focus on an arbitrary single term of the above sum. 
If $c = ({n+1}\, i_1 i_2 \cdots i_k)$ then 
we have 
\begin{align}
%	\sum_{c = ({n+1}\, i_1 i_2 \cdots i_k)}  & \sgn{(c)} \prod_{i \in c } a_{i,c(i)}\\
 \sgn{(c)} \prod_{i \in c } a_{i,c(i)}
	        &  = (-1^k)
	 a_{n+1,i_1}\, a_{i_1,i_{2}} \cdots a_{i_{k-1},i_k}\, a_{i_k,n+1}  \label{dete4} \\
	 &  = (-1)^k 
	 \gamma(F_{i_1} \to F_{}) 
	    \gamma(F_{i_2} \to I_{i_1})  \cdots
%	  \gamma(F_{i_{k-1}} \to I_{i_{k-2}})
	  \gamma(F_{i_k} \to I_{i_{k-1}})
	     \gamma(I_{} \to I_{i_k}) \notag \\
	  & = (-1)^k
	      \gamma(I_{i_1} \xrightarrow{\pi_{i_1}}  F_{i_1}\to F)
	      \gamma(I_{i_2} \xrightarrow{\pi_{i_1}}  F_{i_1}\to I_{i_{1}})
	  \cdots
	  \gamma(I_{i_k} \xrightarrow{\pi_{i_k}}  F_{i_k}\to I_{i_{k-1}})
	  \gamma(I\to I_{i_k})
	  \\
	  & = (-1)^k \gamma( 
          I\to I_{i_k} \xrightarrow{\pi_{i_k}} F_{i_k} \to   I_{i_{k-1}} \xrightarrow{\pi_{i_{k-1}}} F_{i_{k-1}}\to
	  \cdots \to  I_{i_1} \xrightarrow{\pi_{i_1}} F_{i_1} \to  F),
      \label{eq:9}
\end{align}
where the equalities use the same identities as those  in \autoref{2}, \autoref{3} and \autoref{4}.

Forgetting the sign for a moment, we claim the final equality \autoref{eq:9} is the number of all paths from $I$ to $F$ which pass through the subpaths $\pi_{i_k},\pi_{i_{k-1}},\ldots,\pi_{i_1}$ in this precise order since the paths are non-overlapping.
Suppose contrarily that there is some path $\pi$ from $I$ to $F$ which passes through all subpaths $\pi_{i_k}, \ldots, \pi_{i_1}$, but is not counted by \autoref{eq:9}.
Then there is some $F_{i_j}$ and some $I_{i_{j-1}}$ which both lie on $\pi$ such that $I_{i_{j-1}}$ lies before $F_{i_j}$ (forcing \autoref{eq:9} to be equal to $0$).
By construction, since $\pi_{i_j}$ and $\pi_{i_{j-1}}$ are subpaths of $\pi$, then $F_{i_{j-1}}$ is also on the path $\pi$ and it appears either before or after $F_{i_j}$ in $\pi$.
If $F_{i_{j-1}}$ comes before $F_{i_j}$ then $\pi_{i_{j-1}}$ is a subpath of $\pi_{i_j}$ and if $F_{i_{j-1}}$ comes after $F_{i_j}$ then $\pi_{i_{j-1}}$ overlaps $\pi_{i_j}$.
In both cases, we have a contradiction since $\Pi$ is a collection non-overlapping paths.

Therefore, remembering the sign, since the paths in $\Pi$ are pairwise non-overlapping, the quantity in \autoref{eq:9} is precisely $(-1)^k$ times the number of paths from $I$ to $F$ which contain the subpaths 
$\pi_{i_k},\pi_{i_{k-1}}\ldots,\pi_{i_1}$  in this precise order giving us:
\begin{align*}
  \sgn{(c)} \prod_{i \in c } a_{i,c(i)} = (-1)^k 
	 \gamma \Bigl(I\to F \st   \substack{\displaystyle \pi \text{ contains all } \pi_{i_k},\pi_{i_{k-1}},\ldots,\pi_{i_1} \text{ as subpaths}  \\ \displaystyle \text{ in this precise order}}  \Bigr),
\end{align*}
and hence, fixing a subset  $I \subseteq\{1,\ldots,n\}$ of size $k$
and summing over all possible 
cycles $c=(n+1 \,i_1 i_2 \cdots i_k)$ with 
$ \{ i_1,\ldots,i_k\} = I $, 
we obtain 
\begin{align*}
	\sum_{ \substack{ c=(n+1 i_1 i_2 \cdots i_k)    \\  \{i_1,\ldots,i_k\}=I }  }  \sgn{(c)} \prod_{i \in c } a_{i,c(i)} = (-1)^k 
	\gamma\Bigl(I\to F \st   \substack{\displaystyle \pi \text{ contains all } \pi_{i_1},\pi_{i_{2}},\ldots,\pi_{i_k} \text{ as subpaths}  \\ \displaystyle \text{ in whatever order}}  \Bigr)
\end{align*}
In view of the above, it is straightforward to see that 
\autoref{dete3} is an inlusion-exclusion:
\begin{align*}
	 & \gamma(I \to F) +\sum_{ \emptyset \subset I \subseteq \{1,\ldots,n\}} (-1)^{|I|} \gamma(I\to F\st\text{ contains all } 
	\pi_i \text{ with } i \in I) \\
	  & = 
	\gamma(I\to F\st \text{ has no subpath in }  \{\pi_1,\ldots, \pi_n\} ). 
\end{align*}
Therefore $\gamma\left( I \to F \st \text{ has no subpath in }\Pi \right) = \det(M_{\Pi})$.
\end{proof}

In the general case where we allow arbitrary weights, we have the following.
\begin{theorem}
    \label{thm:nonoverlapping}
    Let $I$ and $F$ be two arbitrary vertices in an acyclic digraph $\Gamma$
    and  $\Pi=$ $\left\{ \pi_1, \ldots, \pi_n \right\}$ be a collection of non-overlapping paths.
    Let
    \begin{align*}
        M_{\Pi} &=  \begin{pmatrix}
            \gamma{(I_1 \to I_1)} & \gamma{(I_2 \xrightarrow{\pi_2} F_2 \to I_1)} & \cdots & \gamma{(I_{n} \xrightarrow{\pi_n} F_{n} \to I_1)} & \gamma{(I \to I_1)}\\
            \gamma{(I_1 \xrightarrow{\pi_1} F_1 \to I_2)} & \gamma{(I_2 \to I_2)} & \cdots & \gamma{(I_{n} \xrightarrow{\pi_n} F_{n} \to I_2)} & \gamma{(I \to I_2)}\\
            \vdots & \vdots & \ddots & \vdots & \vdots\\
            \gamma{(I_1 \xrightarrow{\pi_1} F_1 \to I_{n})} & \cdots & \cdots & \gamma{(I_{n} \to I_{n})} & \gamma{(I \to I_n)}\\
            \gamma{(I_1 \xrightarrow{\pi_1} F_1 \to F)} & \gamma{(I_2 \xrightarrow{\pi_2} F_2 \to F)} & \cdots & \gamma{(I_{n} \xrightarrow{\pi_n} F_{n} \to F)} & \gamma{(I \to F)}\\
        \end{pmatrix}\\ \\
        &= \begin{pmatrix}
            1 & \wt(\pi_2)\cdot\gamma{(F_2 \to I_1)} & \cdots & \wt(\pi_n)\cdot\gamma{(F_{n} \to I_1)} & \gamma{(I \to I_1)}\\
            \wt(\pi_1)\cdot\gamma{(F_1 \to I_2)} & 1 & \cdots & \wt(\pi_n)\cdot\gamma{(F_n\to I_2)} & \gamma{(I \to I_2)}\\
            \vdots & \vdots & \ddots & \vdots & \vdots\\
            \wt(\pi_1)\cdot\gamma{(F_1 \to I_{n})} & \wt(\pi_2)\cdot \gamma(F_2 \to I_n) & \cdots & 1 & \gamma{(I \to I_n)}\\
            \wt(\pi_1)\cdot\gamma{(F_1 \to F)} & \wt(\pi_2)\cdot\gamma{(F_2 \to F)} & \cdots & \wt(\pi_n)\cdot\gamma{(F_{n} \to F)} & \gamma{(I \to F)}\\
        \end{pmatrix}
    \end{align*}
    Then $\gamma\left( I \to F \st \text{has no subpath in }\Pi \right) = \det(M_{\Pi})$.
\end{theorem}
\begin{proof}
    The proof for this is identical to the proof for \autoref{thm:path-nonoverlapping} with the additional fact that we must keep track of the weights.
    In particular, the elements of the $(n+1)\times(n+1)$ matrix are given by:
\begin{align*}
 a_{ij} = \begin{cases} 
    \gamma{(I_1 \to I_1)} & \text{ for } 1\leq i = j \leq n \\
    \gamma{(I_j \xrightarrow{\pi_j} F_j\to I_i)} & \text{ for } 1\leq i \neq j \leq n \\
	\gamma{(F_j \to F)} & \text{ for } 1\leq j \leq n \text{ and } i =n+1\\ 
	\gamma{(I \to I_i)} & \text{ for } 1\leq i \leq n \text{ and } j =n+1\\ 
	\gamma{(I\to F)}  & \text{ for } i =j=n+1
\end{cases}
\end{align*}
Combining this with the facts that $\gamma(I_i \xrightarrow{\pi_i} F_i) = \wt(\pi_i)$, $\gamma(I_i \to I_i) = 1$ and that weights multiply, gives us the desired results.
\end{proof}

\begin{remark}
    Although these matrices are large as the number of paths increases, these matrices are sparse meaning that they contain a large number of $0$s.
    This comes from the fact that for any distinct $i$ and $j$ then, since our digraph is acyclic, either $\gamma\left( F_i \to I_j \right) = 0$ or $\gamma\left( F_j \to I_i \right) = 0$ (or both).
    In other words, calculating the matrix becomes fairly fast as many entries are zero.
    For an example of this the reader is invited to look at the matrix in \autoref{ex:A5} in which $15$ out of $36$ entries of the matrix are $0$.
\end{remark}

\begin{examples}
    The condition that our collection be non-overlapping is necessary.
    The following two examples show why the definition of overlapping is as it is.
    This can be seen to be true in trivial cases (paths which are lines), but we give a slightly larger example to show the nuances.
    \begin{itemize}
        \item Suppose we have the following digraph where $\pi_1$ is red/dotted/shorter and $\pi_2$ is blue/dashed/longer.
            
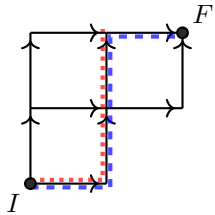
\begin{center}
                \begin{tikzpicture}
                    [thick,decoration={markings,mark=at position 0.9 with {\arrow{>}}}] 

                    \draw[postaction={decorate}] (0,0) -- (1,0);
                    \draw[postaction={decorate}] (1,0) -- (1,1);
                    \draw[postaction={decorate}] (1,1) -- (1,2);
                    \draw[postaction={decorate}] (1,1) -- (2,1);
                    \draw[postaction={decorate}] (2,1) -- (2,2);
                    \draw[postaction={decorate}] (0,0) -- (0,1);
                    \draw[postaction={decorate}] (0,1) -- (1,1);
                    \draw[postaction={decorate}] (0,1) -- (0,2);
                    \draw[postaction={decorate}] (0,2) -- (1,2);
                    \draw[postaction={decorate}] (1,2) -- (2,2);
                    \draw[ultra thick, dotted,opacity=0.7, red] (0, 0.05) -- (0.95,0.05) -- (0.95,1.05) -- (0.95,2.05);
                    \draw[ultra thick, dashed,opacity=0.7, blue] (0,-0.05) -- (1.05,-0.05) -- (1.05,0.95) -- (1.05,1.95) -- (1.95,1.95);
                    \draw[fill=Black] (0, 0) circle[radius=2pt, fill=Black] node[below left] {$I$};
                    \draw[fill=Black] (2, 2) circle[radius=2pt, fill=Black] node[above right] {$F$};
                \end{tikzpicture}    
            \end{center}
            Then we have the following matrix:
            \[
                \det\left( \begin{matrix}
                    1 &  0 & 1\\
                    0 & 1 & 1\\
                    1 & 1 & 5
            \end{matrix}\right) = 5 - 1 - 1 = 3, 
            \]
            even though the number of paths from $I$ to $F$ not containing $\pi_1$ nor $\pi_2$ as subpaths is equal to $4$.
        \item Suppose we have the following digraph where $\pi_1$ is red/dotted and $\pi_2$ is blue/dashed.
            
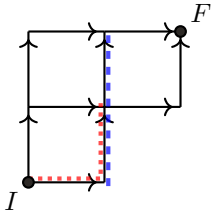
\begin{center}
                \begin{tikzpicture}
                    [thick,decoration={markings,mark=at position 0.9 with {\arrow{>}}}]
                    \draw[postaction={decorate}] (0,0) -- (1,0);
                    \draw[postaction={decorate}] (1,0) -- (1,1);
                    \draw[postaction={decorate}] (1,1) -- (1,2);
                    \draw[postaction={decorate}] (1,1) -- (2,1);
                    \draw[postaction={decorate}] (2,1) -- (2,2);
                    \draw[postaction={decorate}] (0,0) -- (0,1);
                    \draw[postaction={decorate}] (0,1) -- (1,1);
                    \draw[postaction={decorate}] (0,1) -- (0,2);
                    \draw[postaction={decorate}] (0,2) -- (1,2);
                    \draw[postaction={decorate}] (1,2) -- (2,2);
                    \draw[ultra thick, dotted, red, opacity=0.7] (0,0.05) -- (0.95,0.05) -- (0.95,1.05);
                    \draw[ultra thick, dashed, blue,opacity=0.7] (1.05,-0.05) -- (1.05,0.95) -- (1.05,1.95);
                    \draw[fill=Black] (0, 0) circle[radius=2pt, fill=Black] node[below left] {$I$};
                    \draw[fill=Black] (2, 2) circle[radius=2pt, fill=Black] node[above right] {$F$};
                \end{tikzpicture}    
            \end{center}
            Then we have the following matrix
            \[
                \det\left( \begin{matrix}
                        1 & 0 & 1\\
                        0 & 1 & 1\\
                        2 & 1 & 5
                \end{matrix}\right) = 5 - 1 - 2 = 2, 
            \]
            even though the number of paths from $I$ to $F$ not containing $\pi_1$ nor $\pi_2$ as subpaths is equal to $3$.
    \end{itemize}
\end{examples}

\section{Root posets as digraphs}
\label{sec:root_posets_as_digraphs}
We will use non-overlapping paths in the setting of Weyl groups in order to enumerate the number of regions per Weyl cone.
In order to use non-overlapping paths, we must associate a digraph to each Weyl group.
In particular, we associate a digraph to each root poset $(\Phi^+, \leq)$ for a Weyl group $W$ where the paths in this digraph will be associated to antichains in the root poset.
These paths are then used for the enumeration of each region.
We do this in a case by case basis.

For ease of notation,
% we set the following notation for the rest of this section. 
if $\alpha_1,\ldots,\alpha_n$ are the simple roots in $\Phi$, we write 
\begin{align}
	\alpha_{ij} &= \sum_{k = i}^j \alpha_k, \text{~and}\nonumber\\
%	\alpha_{ij,mn} &= \alpha_{ij} + \alpha_{mn}\label{eq:shorthand}.
	\alpha_{ij,\ell m} &= \alpha_{ij} + \alpha_{\ell m}\label{eq:shorthand}.
	\end{align}

We sometimes simplify $\alpha_{ii} $ to $\alpha_i$ and 	$\alpha_{ij,\ell \ell}$  to $\alpha_{ij,\ell}$.

\subsection{Shi's Diagrams}
\label{sssec:shis-diagrams}
In \cite{Shi_Number}, Shi describes a way to associate a diagram to a root poset (in a type by type manner for types $A$, $B$ and $D$) such that certain subdiagrams are associated to antichains in the root poset.
We describe these diagrams and subdiagrams next as they will be the starting point for our digraphs.

A \defn{diagram} $\Lambda$ is an array of boxes divided into rows and columns (potentially overlapping).
To each Weyl type $X$, we associate a particular diagram $\Lambda_X$ where the boxes are labelled by roots.

\subsubsection{Type $A$}
In type $A$ we let $\Lambda_{A_n}$ be the staircase Young diagram of size $n$ where the first row (on the bottom) has one box, the second row has two boxes, etc.
To the $i$th box in the $j$th row, we associate the root $\alpha_{ij}$.
The simple roots occupy the main diagonal boxes of the diagram and every other root 
is the sum of the simple roots lying to the south and to the east.
An example of $\Lambda_{A_3}$ can be found on the left of \autoref{fig:shi-diagram-ABD}.
\begin{figure}
    \begin{center}
        \begin{tikzpicture}
            \begin{scope}[shift={(0,0)}]
                \draw (0,0) -- (1,0);
                \draw (0,0) -- (0,3);
                \draw (1,0) -- (1,3);
                \draw (0,1) -- (2,1);
                \draw (2,1) -- (2,3);
                \draw (0,2) -- (3,2);
                \draw (3,2) -- (3,3);
                \draw (0,3) -- (3,3);
                \node at (0.5, 0.5) {$\alpha_{11}$};
                \node at (1.5, 1.5) {$\alpha_{22}$};
                \node at (2.5, 2.5) {$\alpha_{33}$};
                \node at (0.5, 1.5) {$\alpha_{12}$};
                \node at (0.5, 2.5) {$\alpha_{13}$};
                \node at (1.5, 2.5) {$\alpha_{23}$};
            \end{scope}
            \begin{scope}[shift={(4,0)}, xscale=1]
                \draw (0,0) -- (1,0);
                \draw (0,0) -- (0,5);
                \draw (1,0) -- (1,5);
                \draw (0,1) -- (2,1);
                \draw (2,1) -- (2,4);
                \draw (0,2) -- (3,2);
                \draw (3,2) -- (3,3);
                \draw (0,3) -- (3,3);
                \draw (0,4) -- (2,4);
                \draw (0,5) -- (1,5);
                \node at (0.5, 0.5) {$\alpha_{11}$};
                \node at (0.5, 1.5) {$\alpha_{12}$};
                \node at (0.5, 2.5) {$\alpha_{13}$};
                \node at (0.5, 3.5) {$\alpha_{13,33}$};
                \node at (0.5, 4.5) {$\alpha_{13,23}$};
                \node at (1.5, 1.5) {$\alpha_{22}$};
                \node at (1.5, 2.5) {$\alpha_{23}$};
                \node at (1.5, 3.5) {$\alpha_{23,33}$};
                \node at (2.5, 2.5) {$\alpha_{33}$};
            \end{scope}
            \begin{scope}[shift={(8,0)}, xscale=1.4]
            \draw[thick] (0,0) -- (0,7);
            \draw[thick] (1,0) -- (1,7);
            \draw[thick] (2,1) -- (2,6);
            \draw[thick] (3,2) -- (3,5);
            \draw[thick] (4,3) -- (4,4);
            \draw[thick] (0,0) -- (1,0);
            \draw[thick] (0,1) -- (2,1);
            \draw[thick] (0,2) -- (3,2);
            \draw[thick] (0,3) -- (4,3);
            \draw[thick] (0,4) -- (4,4);
            \draw[thick] (0,5) -- (3,5);
            \draw[thick] (0,6) -- (2,6);
            \draw[thick] (0,7) -- (1,7);

            \draw[thick] (0,4) -- (1, 3);
            \draw[thick] (1,4) -- (2, 3);
            \draw[thick] (2,4) -- (3, 3);
            \draw[thick] (3,4) -- (4, 3);

            \node at (0.5, 0.5) {\small$\alpha_{11}$};
            \node at (0.5, 1.5) {\small$\alpha_{12}$};
            \node at (1.5, 1.5) {\small$\alpha_{22}$};
            \node at (0.5, 2.5) {\small$\alpha_{13}$};
            \node at (1.5, 2.5) {\small$\alpha_{23}$};
            \node at (2.5, 2.5) {\small$\alpha_{33}$};

            \node at (0.38, 3.22) {\small$\alpha_{13,55}$};
            \node at (0.7, 3.75) {\small$\alpha_{14}$};
            \node at (1.38, 3.22) {\small$\alpha_{23,55}$};
            \node at (1.67, 3.75) {\small$\alpha_{24}$};
            \node at (2.38, 3.22) {\small$\alpha_{33,55}$};
            \node at (2.7, 3.75) {\small$\alpha_{34}$};
            \node at (3.38, 3.22) {\small$\alpha_{55}$};
            \node at (3.75, 3.75) {\small$\alpha_{44}$};

            \node at (0.5, 4.5) {\small$\alpha_{15}$};
            \node at (1.5, 4.5) {\small$\alpha_{25}$};
            \node at (2.5, 4.5) {\small$\alpha_{35}$};
            \node at (0.5, 5.5) {\small$\alpha_{15,33}$};
            \node at (1.5, 5.5) {\small$\alpha_{25,33}$};
            \node at (0.5, 6.5) {\small$\alpha_{15,23}$};
        \end{scope}
        \end{tikzpicture}
        \caption{The diagrams $\Lambda_{A_3}$ (left), $\Lambda_{B_3}$ (middle) and $\Lambda_{D_5}$ (right) as constructed by Shi.}
        \label{fig:shi-diagram-ABD}
    \end{center}
\end{figure}
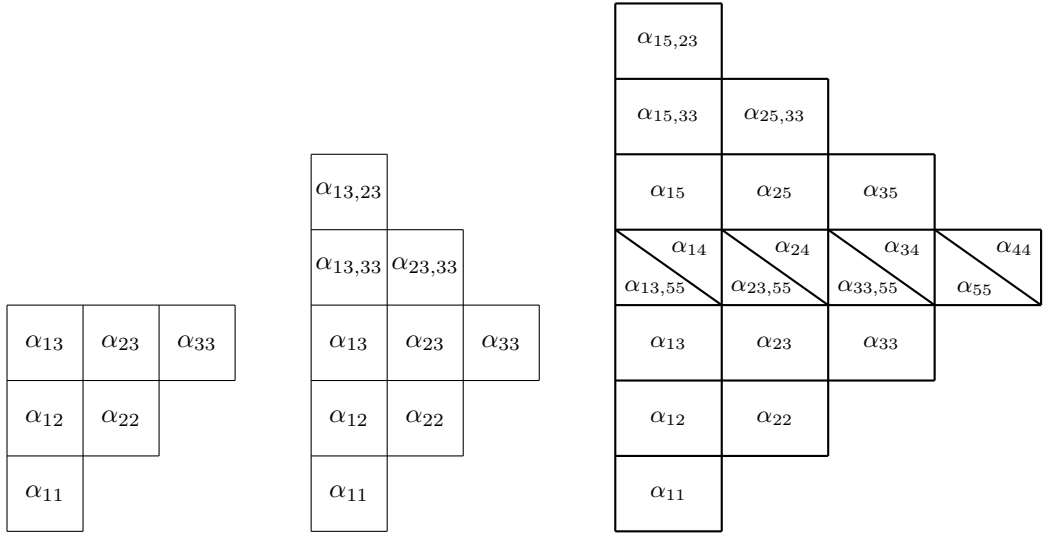

\subsubsection{Type $B$ (and $C$)}
Note that since $B_n \iso C_n$, it suffices to look at $B_n$ Weyl groups.
In type $B$ we let $\Lambda_{B_n}$ be the staircase skew Young diagram of size $n$ where the first row (on the bottom) has one box, the second row has two boxes, etc. up until the $n$th row and then decreasing back down to one box on the top row.
We label the boxes in the following way:
\begin{itemize}
    \item If $j \in \left\{ 1, \ldots, n \right\}$ then the $i$th box in the $j$th row is denoted by $\alpha_{ij}$.
    \item If $j \in \left\{ n+1, \ldots, 2n-1 \right\}$ then the $i$th box in the $j$th is denoted by $\alpha_{in,kn}$ where $k = 2n - j + 1$.
\end{itemize}
An example of $\Lambda_{B_3}$ can be found in the middle of \autoref{fig:shi-diagram-ABD}.

\subsubsection{Type $D$}
In type $D$ we let $\Lambda_{D_n}$ be the staircase skew Young diagram of size $n-1$ as defined for type $B$.
In addition, we duplicate the middle row so that two boxes are overlapping one another.
We label the boxes in the following way.
\begin{itemize}
    \item If $j \in \left\{ 1, \ldots, n-1 \right\}$ then the $i$th box in the $j$th row is denoted by $\alpha_{ij}$.
    \item If $j \in \left\{ n+2, \ldots, 2n-1 \right\}$ then the $ith$ box in the $j$th is denoted by $\alpha_{in,k\ell}$ where $k = 2n - j + 2$ and $\ell = n-2$.
    \item If $j = n+1$ then the $i$th box is denoted by $\alpha_{in}$.
    \item If $j = n$ then there are two boxes:
        \begin{itemize}
            \item The $i$th upper box (or top right corner in the figures) is denoted by $\alpha_{ik}$ where $k = n-1$.
            \item The $i$th lower box (or bottom left corner in the figures) is denoted by $\alpha_{ik,nn}$ where $k = n-2$ if $i \neq n-1$ else it's denoted by $\alpha_{nn}$.
        \end{itemize}
\end{itemize}

An example of the type $D_5$ diagram can be found on the right of \autoref{fig:shi-diagram-ABD}.

\subsubsection{Subdiagrams}
A \defn{subdiagram} $\lambda$ of a diagram $\Lambda$ is a subset of $\Lambda$ such that if a box is in $\lambda$ then every box to the north and to the west of the box is also in $\lambda$.
Shi showed that these subdiagrams are precisely the subsets of $\Lambda$ associated to antichains in the root poset (using upper ideals).
\begin{theorem}[{\cite{Shi_Number}}]
    \label{thm:boxes-iso-anti}
    Let $\Lambda_{X}$ be the diagram associated to a Weyl group $W$ of type ${X \in \left\{ A_n, B_n, C_n, D_n \right\}}$.
    Let $\Phi$ be an associated root poset of $W$.
    Then there is a bijection between subdiagrams of $\Lambda_{X}$ and antichains in $\Phi$.
\end{theorem}

The diagrams $\Lambda_X$ will be used to construct our digraphs.
Before that, we set some notation for boxes and corners of boxes which we will use throughout the rest of this article.

\subsection{Boxes and corners}
\label{ssec:boxes_and_corners}

Let $\lambda$ be a box in some diagram $\Lambda$.
Note that the box has four edges and four vertices.
Suppose that $\alpha \in \Phi^+$ is the root associated to the box $\lambda$.
Then the vertices are labelled in the following way:
\begin{align*}
    \text{top left vertex}&:~v^{tl}&\text{top right vertex}&:~v^{tr} \\
    \text{bottom left vertex}&:~v^{bl}&\text{bottom right vertex}&:~v^{br}\\
\end{align*}

%\et{
%\begin{align*}
%\text{top left vertex}&:~v^{tl}&\text{top right vertex}&:~v^{tr} \\
%\text{bottom left vertex}&:~ \bl{v}&\text{bottom right vertex}&:~\tr{v}\\
%\end{align*}
%\eleni{I'm not sure I like it :-) }
%
%}

We give an orientation to the edges of every box such that the bottom left vertex $v^{bl}$ is the unique source.
Additionally, we will occasionally remove the top edge of a box, replacing the edge with a dashed edge to denote the edge is \emph{not} a part of the digraph.
After orientation, we let \defn{the corner associated to a box} be the (length $2$) subpath $\pi$ which goes from $v^{bl}$ to $v^{br}$ to $v^{tr}$.
In the notation of \autoref{sec:non-overlapping-paths}, then $v^{bl} = I(\pi)$ and $v^{tr} = F(\pi)$ in the corner associated to the box.
Here are the three possible orientations we will be working with where the corners are thickened.
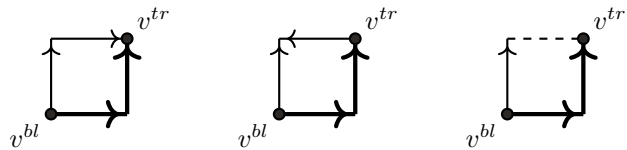
\begin{center}
    \begin{tikzpicture}[thick,decoration={markings,mark=at position 0.9 with {\arrow{>}}}]    

        \begin{scope}[shift={(0,0)}]
            \draw[postaction={decorate}] (0,0) -- (0,1);
            \draw[ultra thick, postaction={decorate}] (0,0) -- (1,0);
            \draw[postaction={decorate}] (0,1) -- (1,1);
            \draw[ultra thick, postaction={decorate}] (1,0) -- (1,1);
            \draw[fill=Black] (0, 0) circle[radius=2pt, fill=Black] node[below left] {$v^{bl}$};
            \draw[fill=Black] (1, 1) circle[radius=2pt, fill=Black] node[above right] {$v^{tr}$};
        \end{scope}
        \begin{scope}[shift={(3,0)}]
            \draw[postaction={decorate}] (0,0) -- (0,1);
            \draw[ultra thick, postaction={decorate}] (0,0) -- (1,0);
            \draw[postaction={decorate}] (1,1) -- (0,1);
            \draw[ultra thick,postaction={decorate}] (1,0) -- (1,1);
            \draw[fill=Black] (0, 0) circle[radius=2pt, fill=Black] node[below left] {$v^{bl}$};
            \draw[fill=Black] (1, 1) circle[radius=2pt, fill=Black] node[above right] {$v^{tr}$};
        \end{scope}
        \begin{scope}[shift={(6,0)}]
            \draw[postaction={decorate}] (0,0) -- (0,1);
            \draw[ultra thick, postaction={decorate}] (0,0) -- (1,0);
            \draw[thick, dashed] (0,1) -- (1,1);
            \draw[ultra thick, postaction={decorate}] (1,0) -- (1,1);
            \draw[fill=Black] (0, 0) circle[radius=2pt, fill=Black] node[below left] {$v^{bl}$};
            \draw[fill=Black] (1, 1) circle[radius=2pt, fill=Black] node[above right] {$v^{tr}$};
        \end{scope}
    \end{tikzpicture}
\end{center}

\subsection{Types \texorpdfstring{$A$}{A} and \texorpdfstring{$B$}{B}}
\label{sssec:types_A_and_B}
For types $A$ and $B$, we use the diagrams as constructed in \cite{Shi_Number} (see \autoref{sssec:shis-diagrams}) as the underlying graph for our digraphs.
Examples of the digraphs constructed below for type $A_7$ and $B_4$ are found in \autoref{fig:gamma_AB} to help follow along with the constructions.
\begin{figure}
    \begin{center}
        \input{figures/A7B4}
        \caption{The digraph $\Gamma_{A_7}$ (left) and $\Gamma_{B_4}$ (right) associated to the Shi arrangements of type $A$ and $B$.}
        \label{fig:gamma_AB}
    \end{center}
\end{figure}
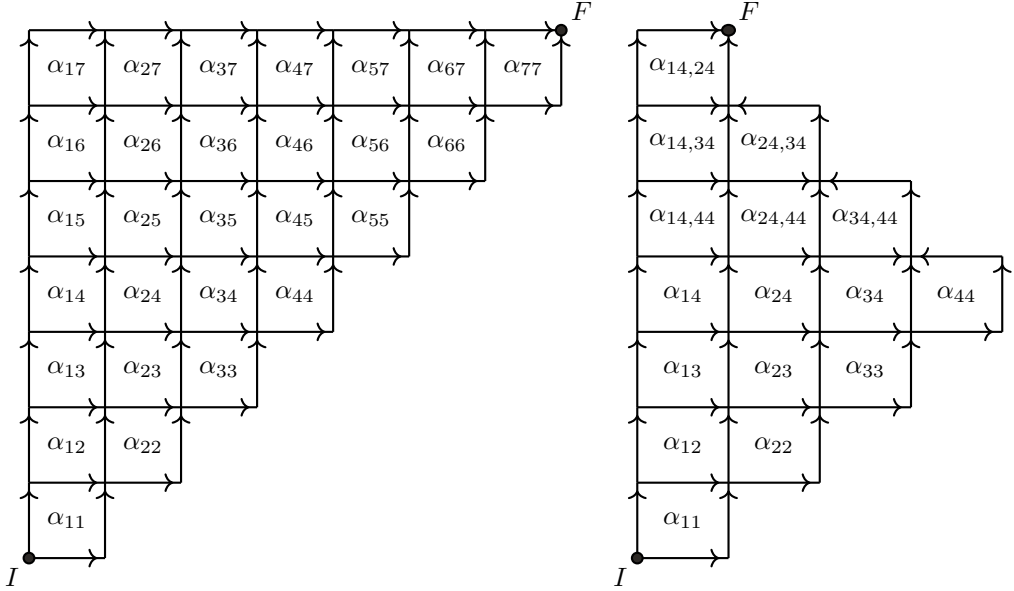

\paragraph{\textbf{Type \texorpdfstring{$A$}{A}:}}
For a type $A_n$ Weyl group, we let $\Lambda_{A_n}$ be the staircase Young diagram of size $n$ constructed in \autoref{sssec:shis-diagrams}.
It remains to give an orientation to the edges in $\Lambda_{A_n}$ to turn it into a digraph.
Let $\Gamma_{A_n}$ be the digraph whose underlying graph is $\Lambda_{A_n}$ where the edges are oriented in the following way:
\begin{itemize}
    \item all vertical edges are oriented towards the north and
    \item all horizontal edges are oriented towards the east.
\end{itemize}

\paragraph{\textbf{Type \texorpdfstring{$B$}{B}:}}
For a type $B_n$ Weyl group, we let $\Lambda_{B_n}$ be the staircase skew Young diagram of rank $n$ constructed in \autoref{sssec:shis-diagrams}.
It remains to give an orientation to the edges in $\Lambda_{B_n}$ to turn it into a digraph.
Let $\Gamma_{B_n}$ be the digraph whose underlying graph is $\Lambda_{B_n}$ where the edges are oriented in the following way:
\begin{itemize}
    \item all vertical edges are oriented towards the north,
    \item all horizontal edges at the top of the diagram (except the furthest north one) are oriented towards the west, and
    \item all other (horizontal) edges are oriented towards the east.
\end{itemize}

\paragraph{\textbf{The bijection:}}
We let $I$ denote the unique source (the vertex whose in-degree is $0$) and let $F$ denote the unique sink (the vertex whose out-degree is $0$).
The paths from $I$ to $F$ in the digraphs for types $A$ and $B$ are precisely the antichains in their respective root poset.
This proof is in essence the same proof used in \cite{Shi_Number}, but using the paths in the digraph instead of the boxes in the graph.
\begin{theorem}
    \label{thm:AB-bij}
    Let $\Gamma_X$ be the digraph constructed above where $X$ is of type $A_n$ or $B_n$.
    There is a bijection between paths in $\Gamma_X$ from $I$ to $F$ and subdiagrams of $\Lambda_X$.
    Moreover, there is a bijection between paths in $\Gamma_X$ and antichains in $(\Phi^+, \leq)$ where each antichain is equal to the set of roots associated to the corners of a path.
\end{theorem}
\begin{proof}
    Let $\pi$ be a path in $\Gamma_X$ from $I$ to $F$.
    Note that once we reach an edge on top, there is no choice but to take the unique path directly to $F$.
    As the path $\pi$ can only go north and east until we reach a top edge, it will split the digraph into two sides.
    The set of roots on the north/west side of $\pi$ then define a subdiagram of $\Lambda_X$ as desired.
    The reverse map is clear and follows the same method.

    By \autoref{thm:boxes-iso-anti}, this implies there is a bijection between paths in $\Gamma_X$ and antichains in $(\Phi^+, \leq)$.
    Since our paths can only go north and east, then a root is added to the antichain precisely when the corner associated to that root is a subpath as desired.
\end{proof}

\subsection{Type \texorpdfstring{$D$}{D}}
We would like to define a type $D$ digraph in the same way we did for types $A$ and $B$ using Shi's diagrams.
But, as we will describe next, we run into a critical issue and must change tactics.
Recall the type $D$ diagram $\Lambda_{D_n}$ detailed in \autoref{sssec:shis-diagrams} (with an example in \autoref{fig:shi-diagram-ABD}).
We would like to orient $\Lambda_{D_n}$ in such a way that paths are in bijection with subdiagrams.
Since this diagram has overlapping boxes, we first strengthen our understanding of which boxes are present in a subdiagram by considering the following two examples which use $\Lambda_{D_5}$ in \autoref{fig:shi-diagram-ABD}.
Recall that a subdiagram associated to an antichain is the subset of $\Lambda$ which contains all boxes weakly to the north or to the west of all the boxes associated to roots in the antichain.
\begin{figure}
    \begin{center}
        \input{figures/D5_antichains}
        \caption{On the left is the subdiagram associated to the antichain $\left\{ \alpha_{23} \right\}$ and on the right is the subdiagram associated to the antichain $\left\{ \alpha_{13, 55} \right\}$.}
        \label{fig:D5antichains}
    \end{center}
\end{figure}
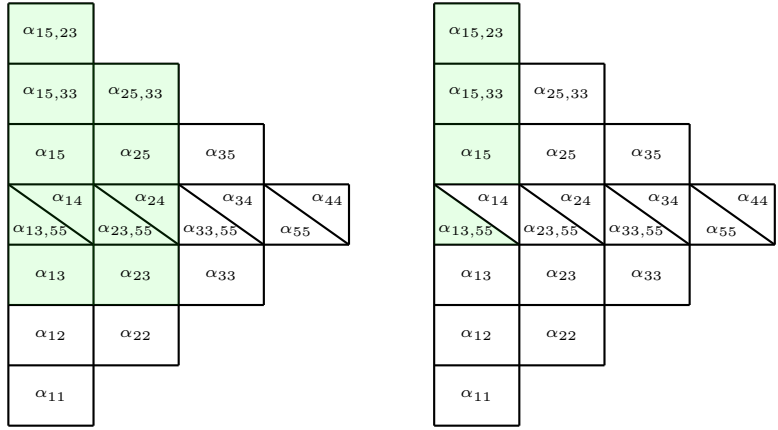
For example, in $\Lambda_{D_5}$ the antichain $\left\{ \alpha_{23} \right\}$ is associated to the subdiagram (see the left hand side of \autoref{fig:D5antichains}) whose boxes are labelled with the following roots:
\[
    \left\{ \alpha_{23},\, \alpha_{13},\, \alpha_{24},\, \alpha_{14},\, \alpha_{23,55},\, \alpha_{13,55},\, \alpha_{25},\, \alpha_{15},\, \alpha_{25,33},\,\alpha_{15,33},\, \alpha_{15,23} \right\}.
\]
On the other hand, the antichain $\left\{\alpha_{13,55}\right\}$ is associated to the subdiagram (see the right hand side of \autoref{fig:D5antichains}) whose boxes are labelled with the following roots:
\[
    \left\{ \alpha_{13,55},\, \alpha_{15},\, \alpha_{15,33},\, \alpha_{15,23} \right\}
\]
In particular, note that $\alpha_{14}$ is \emph{not} contained in the subdiagram of the antichain $\left\{ \alpha_{13,55} \right\}$ 
since $\alpha_{14}$ is above $\alpha_{13,55}$ (in the sense of overlapping) and \emph{not} to the north/west of it.

To understand why converting the diagram to a digraph by just orienting edges won't work, consider the subdiagrams associated to the following two antichains: $\left\{ \alpha_{24},\, \alpha_{55} \right\}$ and $\left\{ \alpha_{23,55},\, \alpha_{44} \right\}$ (see the left and middle figures in \autoref{fig:examplesD}).
It can be verified that no matter how we orient the edges in $\Lambda_{D_n}$, we will not be able to have paths which cut our diagram into two parts giving the associated subdiagrams.
Therefore, we must alter the diagram in some way.

Looking at the diagram for $D_5$, we notice that the main issue comes from these overlapping boxes in the middle.
To get around this, we construct a new diagram based off the original one.
We first split the diagram into four parts:
\begin{itemize}
    \item The bottom part, denoted by $\Gamma^1_{D_n}$, which contains the boxes to the south of the middle row.
    \item The middle upper part, denoted by $\Gamma^2_{D_n}$, which contains (multiple instances of) the boxes in the upper section in the middle row.
    \item The middle lower part, denoted by $\Gamma^3_{D_n}$, which contains (multiple instances of) the boxes in the lower section in the middle row.
    \item The top part, denoted by $\Gamma^4_{D_n}$, which contains the boxes to the north of the middle row.
\end{itemize}
We next describe the four diagrams in more detail, give orientations to the edges and describe the edges between the diagrams.
The reader is invited to follow along with the example in \autoref{fig:gamma_D} where, as a reminder, the dashed lines in the diagram imply the edge has been removed and is not part of the digraph.

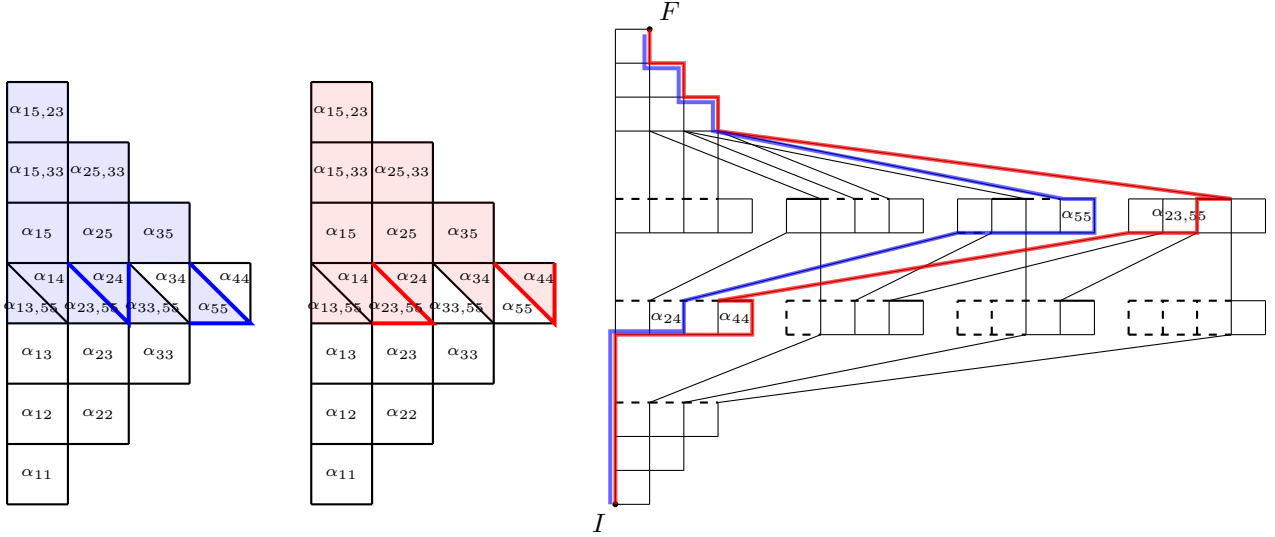
\begin{figure}
	\begin{center}
		\input{figures/D5_example}
		\end{center}
\caption{The antichains $\{\alpha_{24},\alpha_{55}\} $ and $\{\alpha_{23,55},\alpha_{44}\} $ in the root poset $D_5$ are represented by subdiagrams which are shaded in the diagram $\Lambda_{D_5}$. In the graph $\Gamma_{D_5}$ they corresppond to 
distinct paths, each having exactly two corners. 
}
		\label{fig:examplesD}
\end{figure}

\begin{itemize}
    \item[$\Gamma^1_{D_n}$:]
        The bottom part is a staircase shape which looks like a type $A$ diagram.
        Therefore we orient all vertical edges north and all horizontal edges east.
        Finally, we remove the top edges of the diagram.

    \item[$\Gamma^2_{D_n}$:]
        Let $\Gamma_{D_n}^2$ consist of $n-1$ disjoint diagrams, where the $i$th diagram is denoted by $\Gamma_{D_n}^{2,i}$ for $i \in \left[ n-1 \right]$.
        Then $\Gamma_{D_n}^{2,i}$ is a duplicate of the upper part in the middle row of Shi's diagram.
        The first $i-1$ vertical edges are removed and the rest are oriented north.
        The first $i-1$ horizontal edges on the bottom are removed and the rest are oriented east.
        Finally, the final top horizontal edge is oriented west and the rest are removed.
        As an example, we have
        \begin{center}
   	        \input{figures/Gamma2i}
        \end{center}

    \item[$\Gamma^3_{D_n}$:]
        As in the previous case, $\Gamma_{D_n}^3$ consists of $n-1$ disjoint diagrams, where the $i$th diagram is denoted by $\Gamma_{D_n}^{3,i}$ for $i \in \left[ n-1 \right]$.
        Then $\Gamma_{D_n}^{3,i}$ is a duplicate of the lower part in the middle row of Shi's diagram.
        Every vertical edge is oriented to the north, the bottom edge of each box is oriented to the east, and the final horizontal edge on the top row is oriented to the west.
        Additionally, we orient the first $i-1$ horizontal edges on the top row to the east.
        All other edges are removed.
        As an example, we have
    \begin{center}
    	\input{figures/Gamma3i}
    \end{center}

    \item[$\Gamma^4_{D_n}$:]
        The top part is a staircase shape which looks like the top half of a type $B$ diagram.
        Therefore we orient all vertical edges north, all horizontal edges below the top horizontal edges east, the topmost horizontal edge east  and all other top row horizontal edges west as in type $B$.

    \item[Between parts:]
        Finally, we must add directed edges between the different parts in order to make the digraph connected.
        We do this in the following way.
        \begin{itemize}
            \item[$\Gamma_{D_n}^1 \to \Gamma_{D_n}^2$:] There are $n-1$ sink vertices in $\Gamma_{D_n}^1$ (on the top row).
                Ordering these vertices from $1$ to $n-1$ (left to right), then the $i$th vertex has a directed edge to the bottom left vertex of the $i$th box in $\Gamma_{D_n}^{2,i}$.
            \item[$\Gamma_{D_n}^2 \to \Gamma_{D_n}^3$:] There are $n - i$ sink vertices in $\Gamma_{D_n}^{2,i}$ (on the top row).
                Ordering these vertices from $i$ to $n-1$ (left to right), then the $j$th vertex has a directed edge to the bottom left vertex of the $i$th box in $\Gamma_{D_n}^{3,j}$.
            \item[$\Gamma_{D_n}^3 \to \Gamma_{D_n}^4$:] There are $n-i$ sink vertices in $\Gamma_{D_n}^{3,i}$ (on the top row).
                Ordering these vertices from $i$ to $n-1$ (left to right), then the $j$th vertex has a directed edge to the bottom left vertex of the $j$th box in $\Gamma_{D_n}^4$.
        \end{itemize}
\end{itemize}
Then $\Gamma_{D_n}$ is the (connected) digraph obtained from the above process.
As mentioned, an example of $\Gamma_{D_5}$ is given in \autoref{fig:gamma_D} and the reader is invited to relook at the example now.
As a second example, the paths associated to the two antichains $\left\{ \alpha_{24}, \alpha_{55} \right\}$ and $\left\{ \alpha_{23, 55}, \alpha_{44} \right\}$ given earlier are in the rightmost figure in \autoref{fig:examplesD}.

\begin{figure}
    \begin{center}
        \input{figures/D5}
        \caption{The diagram $\Gamma_{D_5}$ for the $D_5$ Shi arrangement.
        }
        \label{fig:gamma_D}
    \end{center}
\end{figure}
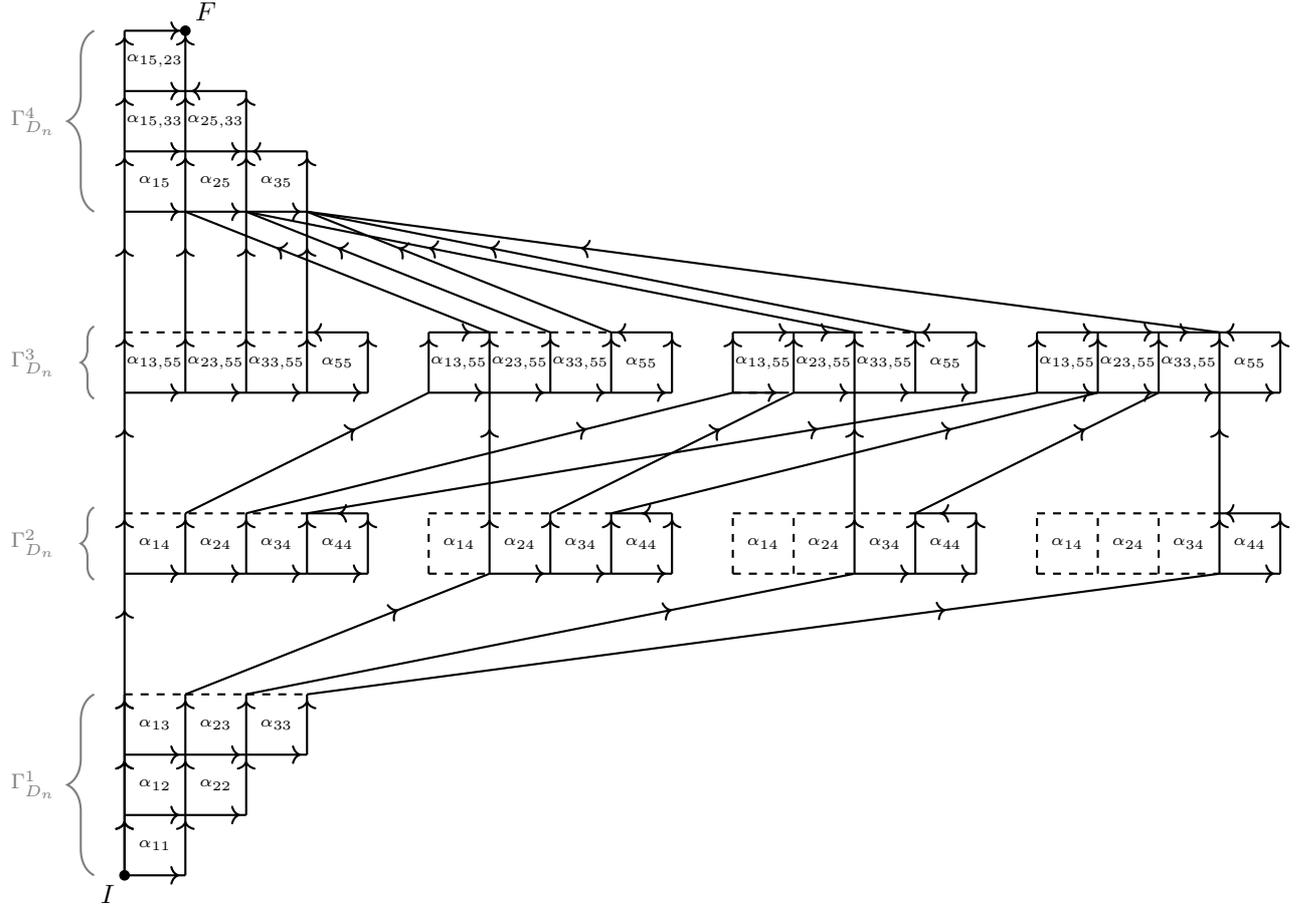

%\newpage
%        \input{figures/D5_three}
%\newpage
%        \eleni{I have changed the position of the arrows in the figure. Do you like it?}
%        \aram{I changed them again as the $\alpha$ were a little hard to read. Is this better? We can also mess around with individual arrows if one arrow is not as visible or something.}
%        \eleni{maybe individual arrows would be nicer... }

As before, we let $I$ denote the unique source (the vertex whose in-degree is $0$) and left $F$ denote the unique sink (the vertex whose out-degree is $0$).
\begin{theorem}
    \label{thm:D-bij}
    Let $\Gamma_{D_n}$ be the digraph obtained from the above algorithm.
    There is a bijection between paths in $\Gamma_{D_n}$ from $I$ to $F$ and antichains in $(\Phi^+, \leq)$.
    Moreover, the antichain is precisely the set of roots whose associated corners are subpaths of the given path.
\end{theorem}
\begin{proof}
    By \autoref{thm:boxes-iso-anti}, there is a bijection between a subdiagram of $\Lambda_{D_n}$ and an antichain in the root poset of $D_n$.
    We give a bijection between paths in $\Gamma_{D_n}$ and subdiagrams of $\Lambda_{D_n}$ which then naturally leads to a bijection with antichains in the root poset of $D_n$.
    
    Let $\pi$ be a path from $I$ to $F$ in $\Gamma_{D_n}$.
    We construct the set $\msB$ of boxes in $\Lambda_{D_n}$ in the following way.
    We break this down by parts for easier readability.
    \begin{enumerate}
        \item In $\Gamma_{D_n}^1$, the path can only go east and north.
            Each time the corner associated to a root $\alpha$ is a subpath of $\pi$, then we add the box associated to $\alpha$ into $\msB$ in addition to adding every box north and to the west of that box (in 	$\Lambda_{D_n}$).
            We eventually will hit the $i$th sink and will be sent to the bottom left vertex of the $i$th box in $\Gamma_{D_n}^{2,i}$.
        \item At this point we can only go north or east ensuring any box in $\Gamma_{D_n}^{2,i}$ already contained in $\msB$ is to the west of the path.
            Following the path, at some point we will be forced to go north.
            If we go east before going north, then we will have a corner associated to a root $\alpha$ as a subpath of $\pi$.
            We then add all boxes to $\msB$ which are to the north and to the west of $\alpha$ in $\Lambda_{D_n}$ (if they are not already in $\msB$).
            We eventually will hit the $j$th sink in $\Gamma_{D_n}^{2, i}$ which will send us to the bottom left vertex of the $i$th box in $\Gamma_{D_n}^{3,j}$.
        \item As before, all boxes in $\Gamma_{D_n}^{3,j}$ already contained in $\msB$ are to the west of the path.
            Following the path, at some point we will be forced to go north.
            If we go east before going north, then we will have a corner associated to a root $\alpha$ as a subpath of $\pi$.
            We then add all boxes to $\msB$ which are to the north and to the west of $\alpha$ in $\Lambda_{D_n}$ (if they are not already in $\msB$).
	        We eventually will hit the $k$th sink where $k \geq i,j$ (since the top row of $\Gamma_{D_n}^{3,j}$ contains $j-1$ east steps on top).
            Therefore, going from $\Gamma_{D_n}^{3,j}$ to $\Gamma_{D_n}^4$ we land in the $k$th box where $k \geq i, j$.
        \item As we land in the $k$th box and as we can only go north and east until we hit the top row, we ensure that all boxes to the north/west of the $k$th box are already contained in $\msB$.
            Following our path, we will again have corners associated to roots as subpaths of $\pi$.
            Adding all boxes associated to these roots and all boxes to the north/west of these boxes into $\msB$ will then give us the subdiagram desired.
    \end{enumerate}
    It is clear that reversing the process above allows you to go from a subdiagram to a path by first tracing the boxes in $\Gamma_{D_{n}}^{1}$ and then tracing the appropriate boxes in the middle row from $\Gamma_{D_n}^2$ and $\Gamma_{D_n}^3$ followed by adding all additional necessary boxes in $\Gamma_{D_n}^4$.
\end{proof}

\subsection{Exceptional types}
It remains to describe the digraphs in the exceptional cases.
Unfortunately, Shi in \cite{Shi_Number} does not describe diagrams associated to any of the exceptional Weyl group.
Luckily, \autoref{thm:num-in-cone} holds for \emph{any} finite Weyl group and therefore in the appendix we present digraphs for the exceptional type Weyl groups $E_6$, $F_4$ and $G_2$ such that the paths are in bijection with the antichains of the associated root poset.
The bijections were checked using sagemath \cite{sagemath}.
As the roots are more complicated, we set the following notation which is a generalisation of our previous root notation:
\[
    \alpha_{i_1j_1,\ldots,i_mj_m} = \sum_{n = 1}^m\alpha_{i_nj_n}.
\]
We leave the following open problems for the exceptional types.
\begin{openprob}
    Although we have a digraph for type $E_6$, digraphs associated to the types $E_7$ and $E_8$ Weyl groups are still unknown.
    As a first open problem we ask what are digraphs for $E_7$ and $E_8$ and is there some algorithmic way to generate them given an arbitrary type $E$ Weyl group?
    This will complete the classification of digraphs for all finite Weyl groups.

    Additionally, the digraphs in types $D$, $E$ and $F$ are ``complicated'' in the sense that there are many edges.
    As a second open problem, we ask whether the provided digraphs have a minimal number of edges.
    In particular, what are (the) minimal digraphs (in the sense of number of edges) that can be used for types $D$, $E$ and $F$?
    This would allow for simplified formulas for faster calculations of the number of regions in $C_w$, which, as can already be seen in the type $D$ formula in \autoref{lem:type_D_counting}, can get very complicated very quickly.
    Note that by ``minimal number of edges'' we allow for vertices with in-degree and out-degree equal to one as these can be trivially removed.
\end{openprob}

\section{Enumerating Regions in a Weyl cone}
\label{sec:Enumerating Regions in a Weyl cone}

From the previous section, we now have digraphs associated to most Weyl group such that the paths in the digraph are in bijection with antichains in the root poset.
We next tackle the question of \emph{how many} paths are there for a root poset restricted to a Weyl cone.
Recall from \autoref{thm:num-in-cone} that there is a bijection between antichains in the subposet $\Phi^+ \bs N(w^{-1})$ and the Weyl cone $C_w$.
Therefore it suffices to make a bijection between certain paths in our digraphs and antichains in the subposet $\Phi^+ \bs N(w^{-1})$.

\begin{theorem}
    \label{thm:nonoverlap-bij-antichains}
    Let $\Gamma_X$ be the digraph for a type $X$ Weyl group $W$ where $X \in \left\{ A_n, B_n, C_n, D_n, E_6, F_4,  G_2 \right\}$.
    Let $(\Phi^+, \leq)$ be the associated root poset for $W$.
    Given an element $w \in W$, then there is a bijection between antichains in the subposet $\Phi^+ \bs N(w^{-1})$ and paths in $\Gamma_X$ which do not contain a corner associated to a root in $N(w^{-1})$ as a subpath.
\end{theorem}
\begin{proof}
    For types $A$, $B$, $C$ and $D$, this is a natural corollary of \autoref{thm:num-in-cone}, \autoref{thm:AB-bij} and \autoref{thm:D-bij}.
    For types $E$, $F$ and $G$, this was verified using sagemath \cite{sagemath}.
\end{proof}

At first sight this theorem might not seem to give us much new information as we are just converting one problem (antichains) to another problem (paths in a digraph).
Luckily, these digraphs are constructed in such a way which makes counting computationally easier.

\subsection{Type \texorpdfstring{$A$}{A}}
\label{ssec:type-A-counting}
In type $A$ we can associate to each vertex in $\Gamma_{A_n}$ a coordinate in the $\Z^2$-lattice.
In particular, we let $I$ be the vertex $(0, 1)$ and $F$ be the vertex $(n, n+1)$.
We make this choice so that every vertex in $\Gamma_{A_n}$ is associated to a vertex weakly above the main ($x = y$) diagonal.
This allows us to use the following theorem:
\begin{theorem}[{\cite[Theorem 10.3.1]{Krattenthaler_LatticePathEnumeration}}]
    \label{thm:count-above-diag}
    Let $\Gamma$ be the infinite digraph of $\Z^2$ with vertical edges pointing north and horizontal edges pointing east.
    Label every vertex of $\Gamma$ by its respective coordinates in $\Z^2$.
    Then 
    \begin{align*}
        \gamma((x_1, y_1) \to (x_2, y_2) \st & \text{weakly above } x = y)\\
         &  = \begin{cases}
    	\nabovediag{x_1}{y_1}{x_2}{y_2} & \text{if } x_1 \leq x_2 \text{~and~} y_1 \leq y_2\\
    	0 & \text{otherwise.}
    \end{cases}
    	    \end{align*}
\qed
   \end{theorem}

We set the following notation for space saving purposes:

\begin{equation}
	    \weakabove{(x_1, y_1)}{(x_2, y_2)} \coloneqq
    \gamma((x_1, y_1) \to (x_2, y_2) \st \text{weakly above } x = y)
\label{form_gamma_above}
\end{equation}

We can associate the digraph $\Gamma_{A_n}$ to a subdigraph of the $\Z^2$-lattice digraph by placing $I = (0,1)$ and letting all other edges line up accordingly.
Recall that for a root $\alpha_i$ we let $v_i^{tr}$ be the vertex in the top right of its associated box and $v_i^{bl}$ be the vertex of the bottom left.
Putting these together gives us the following theorem.
\begin{theorem}
    \label{thm:type-A-count}
    Let $W$ be a type $A$ Weyl group and, for $w \in W$, let $N(w^{-1}) = \left\{ \alpha_{i_1j_1}, \ldots, \alpha_{i_kj_k} \right\}$ be the inversion set of its inverse.
    Then the number of regions in the Weyl cone $C_w$ in the Shi arrangement $\shiarr$ is given by:
    \[
        \order{C_w} = \det\begin{pmatrix}
            1 & \weakabove{v^{tr}_2}{v^{bl}_1} & \cdots & \weakabove{v^{tr}_k}{v^{bl}_1} & \weakabove{I}{v^{bl}_1}\\
            \weakabove{v^{tr}_1}{v^{bl}_2} & 1 & \cdots & \weakabove{v^{tr}_k}{v^{bl}_2} & \weakabove{I}{v^{bl}_2}\\
            \vdots & \vdots & \ddots & \vdots & \vdots\\
            \weakabove{v^{tr}_1}{v^{bl}_k} & \weakabove{v^{tr}_2}{v^{bl}_k} & \cdots & 1 & \weakabove{I}{v^{bl}_k}\\
            \weakabove{v^{tr}_1}{F} & \weakabove{v^{tr}_2}{F} & \cdots & \weakabove{v^{tr}_k}{F} & \weakabove{I}{F}
        \end{pmatrix}, 
          \] 
    where $I = (0,1)$, $F=(n,n+1)$,       
    $v^{tr}_\ell = (i_\ell, j_\ell+1)$ and $v^{bl}_\ell = (i_\ell - 1, j_\ell)$
    and $\gamma$ is the formula in \autoref{form_gamma_above}. 
    
%    $ \bl{v} $
%    
% aaaa    

%\eleni{I'm doing funny things here :-) I'm experimenting with new notation.. we do not necessarily have to do it this way :-) } 
%     \[
%    \order{C_w} = \det\begin{pmatrix}
%    1 & \weakabove{ \tr{v_2}}{\bl{v_1}} & \cdots & \weakabove{ \tr{v_k}}{\bl{v_1}} & \weakabove{I}{\bl{v_1}}\\
%    \weakabove{ \tr{v_1}}{\bl{v_2}} & 1 & \cdots & \weakabove{ \tr{v_k}}{\bl{v_2}} & \weakabove{I}{\bl{v_2}}\\
%    \vdots & \vdots & \ddots & \vdots & \vdots\\
%    \weakabove{ \tr{v_1}}{\bl{v_k}} & \weakabove{ \tr{v_2}}{\bl{v_k}} & \cdots & 1 & \weakabove{I}{\bl{v_k}}\\
%    \weakabove{\tr{v_1}}{F} & \weakabove{\tr{v_2}}{F} & \cdots & \weakabove{\tr{v_k}}{F} & \weakabove{I}{F}\\
%    \end{pmatrix}
%    \]
    
\end{theorem}
\begin{proof}
    By \autoref{thm:num-in-cone} and \autoref{thm:nonoverlap-bij-antichains}, the number of regions in $C_w$ is precisely the number of paths from $I$ to $F$ which don't contain corners associated to the roots in $N(w^{-1})$ as subpaths.
    Letting $\Pi$ be the corners associated to roots in $N(w^{-1})$, then $\Pi$ is a collection of non-overlapping paths.
    By \autoref{thm:path-nonoverlapping} this implies that $\order{C_w}$ is equal to a determinant.
    By our choice of $I = (0,1)$, then we have $F = (n, n+1)$ and for each $\alpha_{i_\ell j_\ell}$
     we have $v^{tr}_\ell = (i_\ell, j_\ell+1)$ and $v^{bl}_\ell = (i_\ell - 1, j_\ell)$.
    In other words, all of our paths are precisely paths weakly above the main diagonal, giving
    \begin{align*}
        \gamma(F_i \to I_j) &= \weakabove{v^{tr}_i}{v^{bl}_j},\\
        \gamma(I \to I_j) &= \weakabove{I}{v^{bl}_j},\\
        \gamma(F_i \to F) &= \weakabove{v^{tr}_i}{F},\text{~and}\\
        \gamma(I \to F) &= \weakabove{I}{F}
    \end{align*}
    as desired.
\end{proof}

\begin{example}
	\label{example1}
    If $n = 2$ then we have the following Shi arrangement of type $A_2$ where the thickened hyperplanes are the hyperplanes of the underlying Weyl arrangement.
    Using the details in \autoref{sssec:types_A_and_B}, the digraph is the digraph $\Gamma_{A_2}$ associated to the root poset.
    
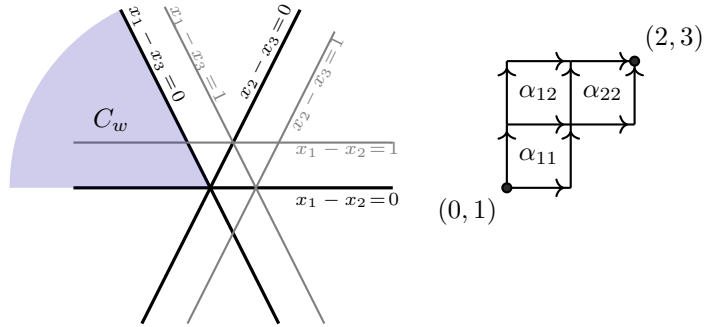
\begin{center}
\begin{tikzpicture}[scale=0.6]
    \begin{scope}[shift={(0,0)}]
		\clip (0,0) circle (4.4cm);
		\draw[domain=-3:4,variable=\x,line width=1.2,black] plot ({\x},0);
		\draw[domain=-3:4,variable=\x,line width=1.2,black] 
		plot({cos(60)*\x},{\x});
		\draw[domain=-3:4,variable=\x,line width=1.2,black]plot({cos(120)*\x},{\x});
		%%%%%
		\draw[domain=-3:4,variable=\x,line width=0.8,gray]plot({cos(120)*\x+1},{\x});
		\draw[domain=-3:4,variable=\x,line width=0.8,gray]plot({cos(60)*\x+1},{\x});
		\draw[domain=-3:4,variable=\x,line width=0.8,gray]plot({\x},{1});
		
		\node at (3,-0.27){\scriptsize\color{black} $x_1-x_2\!=\!0$};
		\node at (1.2,3){\rotatebox{64}{\scriptsize\color{black} $x_2-x_3\!=\!0$}};
		\node at (-1.2,3){\rotatebox{-64}{\scriptsize\color{black} $x_1-x_3=\!0\!$}};
		
		\node at (3,0.85){\scriptsize\color{gray} $x_1-x_2\!=\!1$};	
		\node at (2.35,2.3){\rotatebox{64}{\scriptsize\color{gray} $x_2-x_3\!=\!1$}};
		\node at (-0.3,3){\rotatebox{-64}{\scriptsize\color{gray} $x_1-x_3=\!1\!$}};
			\fill[ablue,opacity=0.2] (0,0)--(117:5.5)--(180:5.5)--cycle;
            \node at (-2.16506350946109, 1.5) {$C_w$};
	\end{scope}
    \begin{scope}[shift={(6.5,0)}, scale=1.4,thick,decoration={markings,mark=at position 0.9 with {\arrow{>}}}] 
        \draw[postaction={decorate}] (0,0) -- (0,1);
        \draw[postaction={decorate}] (0,0) -- (1,0);
        \draw[postaction={decorate}] (0,1) -- (1,1);
        \draw[postaction={decorate}] (0,1) -- (0,2);
        \draw[postaction={decorate}] (0,2) -- (1,2);
        \draw[postaction={decorate}] (1,0) -- (1,1);
        \draw[postaction={decorate}] (1,1) -- (1,2);
        \draw[postaction={decorate}] (1,1) -- (2,1);
        \draw[postaction={decorate}] (1,2) -- (2,2);
        \draw[postaction={decorate}] (2,1) -- (2,2);
        \node at (0.5,0.5) {$\alpha_{11}$};
        \node at (0.5,1.5) {$\alpha_{12}$};
        \node at (1.5,1.5) {$\alpha_{22}$};
        \draw[fill=Black] (0, 0) circle[radius=2pt, fill=Black] node[below left] {$(0,1)$};
        \draw[fill=Black] (2, 2) circle[radius=2pt, fill=Black] node[above right] {$(2,3)$};
	\end{scope}
\end{tikzpicture}
    \end{center}
    Let $w = s_1s_2$.
    Then $N(w^{-1}) = \set{\alpha \in \Phi^+}{\ell(t_\alpha s_2s_1) < \ell(s_2s_1) = 2} = \left\{ \alpha_{22}, \alpha_{12} \right\}$.
    Therefore, the number of regions is equal to the number of paths from $(0,1)$ to $(2,3)$ which do not have the elements of $N(w^{-1})$ as corners.
    This means, we want to avoid the paths: $\pi_{\alpha_{12}}: (0,2) \to (1,2) \to (1,3)$ and $\pi_{\alpha_{22}}: (1,2) \to (2,2) \to (2,3)$.
    By \autoref{thm:type-A-count}
    \begin{align*}
        \order{C_w} &= \det
        {\renewcommand{\arraystretch}{1.2}
        \begin{pmatrix}
        1 & \weakabove{(2,3)}{(0,2)} & \weakabove{(0,1)}{(0,2)}\\
            \weakabove{(1,3)}{(1,2)} & 1 & \weakabove{(0,1)}{(1,2)}\\
            \weakabove{(1,3)}{(2,3)} & \weakabove{(2,3)}{(2,3)} & \weakabove{(0, 1)}{(2, 3)}\\
        \end{pmatrix}
        }
        \\
        &= \det
        {\renewcommand{\arraystretch}{1.2}
        \begin{pmatrix}
            1 & 0 & \nabovediag{0}{1}{0}{2}\\
            0 & 1 & \nabovediag{0}{1}{1}{2}\\
            \nabovediag{1}{3}{2}{3} & \nabovediag{2}{3}{2}{3} & \nabovediag{0}{1}{2}{3}\\
        \end{pmatrix}
        }
        \\
        &= 
        \det
        \begin{pmatrix}
            1 & 0 & 1\\
            0 & 1 & 2\\
            1 & 1 & 5\\
        \end{pmatrix}\\
        &= 2
    \end{align*}
\end{example}

\begin{example}
    \label{ex:A5}
    As a slightly larger example, let $n = 5$ and suppose that $w = s_5 s_2 s_4 s_3 s_1$.
    Then our paths run from $I = (0,1)$ to $F = (5,6)$ allowing only north and east steps, which are weakly above the $y = x$ diagonal.

    Our inversion set is given by: 
    \[
        N(w^{-1}) = \left\{ \alpha_{11},\, \alpha_{33},\, \alpha_{34},\, \alpha_{13},\, \alpha_{35} \right\}
    \]
    The corners to avoid are given by:
    \begin{align*}
        \alpha_{11} = \alpha_1 &:\;(0,1) \to (1,1) \to (1,2)\\
        \alpha_{33} = \alpha_3 &:\; (2,3) \to (3,3) \to (3,4)\\
        \alpha_{34} = \alpha_3 + \alpha_4 &:\; (2,4) \to (3,4) \to (3,5)\\
        \alpha_{13} = \alpha_1 + \alpha_2 + \alpha_3 &:\; (0,3) \to (1,3) \to (1,4)\\
        \alpha_{35} = \alpha_3 + \alpha_4 + \alpha_5 &:\; (2,5) \to (3,5) \to (3,6)\\
    \end{align*}

    Then, the number of regions is given by:
    {\tiny
    \[
        \det{\medmuskip = 1mu%
            \setlength{\arraycolsep}{1pt}
            \renewcommand{\arraystretch}{2.5}
            \begin{pmatrix}
            1 & \weakabove{(3,4)}{(0,1)} & \weakabove{(3,5)}{(0,1)} & \weakabove{(1,4)}{(0,1)} & \weakabove{(3,6)}{(0,1)} & \weakabove{(0,1)}{(0,1)}\\
            \weakabove{(1,2)}{(2,3)} & 1 & \weakabove{(3,5)}{(2,3)} & \weakabove{(1,4)}{(2,3)} & \weakabove{(3,6)}{(2,3)} & \weakabove{(0,1)}{(2,3)}\\
            \weakabove{(1,2)}{(2,4)} & \weakabove{(3,4)}{(2,4)} & 1 & \weakabove{(1,4)}{(2,4)} & \weakabove{(3,6)}{(2,4)} & \weakabove{(0,1)}{(2,4)}\\
            \weakabove{(1,2)}{(0,3)} & \weakabove{(3,4)}{(0,3)} & \weakabove{(3,5)}{(0,3)} & 1 & \weakabove{(3,6)}{(0,3)} & \weakabove{(0,1)}{(0,3)}\\
            \weakabove{(1,2)}{(2,5)} & \weakabove{(3,4)}{(2,5)} & \weakabove{(3,5)}{(2,5)} & \weakabove{(1,4)}{(2,5)} & 1 & \weakabove{(0,0)}{(2,5)}\\
            \weakabove{(1,2)}{(5,6)} & \weakabove{(3,4)}{(5,6)} & \weakabove{(3,5)}{(5,6)} & \weakabove{(1,4)}{(5,6)} & \weakabove{(3,6)}{(5,6)} & \weakabove{(0,1)}{(5,6)}\\
        \end{pmatrix}}\\
    \]
    }
    \begin{align*}
        &= \renewcommand{\arraystretch}{1.4}\det
        \begin{pmatrix}
            1 & 0 & 0 & 0 & 0 & \binom{0}{0} - \binom{0}{2}\\
            \binom{2}{1} - \binom{2}{3} & 1 & 0 & 0 & 0 & \binom{4}{2} - \binom{4}{4}\\
            \binom{3}{2} - \binom{3}{4} & 0 & 1 & \binom{1}{0} - \binom{1}{4} & 0 & \binom{5}{3} - \binom{5}{5}\\
            0 & 0 & 0 & 1 & 0 & \binom{2}{2} - \binom{2}{4}\\
            \binom{4}{3} - \binom{4}{5} & 0 & 0 & \binom{2}{1} - \binom{2}{5} & 1 & \binom{6}{4} - \binom{6}{6}\\
            \binom{8}{4} - \binom{8}{6} & \binom{4}{2} - \binom{4}{4} & \binom{3}{1} - \binom{3}{4} & \binom{6}{2} - \binom{6}{6} & \binom{2}{0} - \binom{2}{4} & \binom{9}{4} - \binom{9}{7}\\
        \end{pmatrix}\\
        &= \renewcommand{\arraystretch}{1}\det
        \begin{pmatrix}
            1 & 0 & 0 & 0 & 0 & 1\\
            2 & 1 & 0 & 0 & 0 & 5\\
            3 & 0 & 1 & 1 & 0 & 9\\
            0 & 0 & 0 & 1 & 0 & 1\\
            4 & 0 & 0 & 2 & 1 & 14\\
            42 & 5 & 3 & 14 & 1 & 132
        \end{pmatrix}&\\
        &= 38
    \end{align*}
\end{example}

There are many $0$ entries in this matrix which makes computations much faster than would be assumed by just looking at \autoref{thm:count-above-diag}.

\subsection{Type \texorpdfstring{$B$}{B}}
\label{ssec:type-B-counting}
\label{type-B-counting}

The type $B$ case is a little more complex, but we can still make use of \autoref{thm:count-above-diag}.
To keep all vertices above the main diagonal we let $I = (0,1)$ as in the type $A$ case.
Furthermore, we note that, in view of the way we have directed the edges in the graph $\Gamma_B$, the 
counting for the
 entries  $\gamma\left( X \to Y \right)$ 
 of the corresponding determinant   is identical with those 
  in type $A$  unless $Y$ is the final point $F$. 

Indeed,  writing down the determinant $M_{\Pi}$ (see also \autoref{elements_aij}) in the special case of the graph $\Gamma_B$, one can see that  
$Y$ is  either an initial point of a {\em corner} of the diagram $\Lambda_{B_n}$  or $Y=F$. 
Since the possible initial points of corners are all the lattice points of the diagram $\Gamma_B$ except those on the lines $x-y=0, 2n, 2n+1$ (see \autoref{fig:typeBcorners}), 
we deduce that when $Y \neq F$  the entry  $\gamma \left( X \to Y \right)$ 
 is precisely the number of lattice paths from $X$ to $Y$  weakly above the $x = y$ diagonal, as in type $A$.
Therefore, the only time we get something different is when $Y = F = (1, 2n)$, which
is the case  only for  the entries of the bottom row of the matrix $M_{\Pi}$.

We handle the point $F$ in the following way.
Let us denote by $\delta$   the line $x-y=2n$ and let 
$F^{D_i}$ for  $i = 0,\ldots,n$ 
 be the 
lattice points of the graph $\Gamma_B$ on $\delta$ from top to bottom, i.e.,  $F^{D_i}$ is  the point $ (i, 2n - i)$ (see Figure \autoref{fig:typeBcorners} right). 
If $X = (a, b)$ is a lattice point of  $\Gamma_B$
on $x-y = 2n$ or $2n+1$, then  there is a unique path from $X$ to $F$. If  $X=(a,b)$ is a point strictly below 
$\delta$ then, in order to compute $\gamma(X \to F)$, 
 we  need to count all paths from $X$ to the  points 
 $F^{D_0}, \ldots, F^{D_n}$ of $\delta $, 
 since each $F^{D_i}$ subsequently leads to $F$ in a unique way. 
 Therefore we have

\begin{equation}
    \label{eq:gamma_Dsum}
    \gamma\left( (a,b) \to F \right) = \Fdiag{(a,b)}{F} = \begin{cases}
        \sum_{i = 1}^{n} \weakabove{(a,b)}{F^{D_i}} & \text{if } b\neq 2n - a + 1\\
    1 & \text{if }b = 2n - a +1
    \end{cases}
\end{equation}
where $D\Sigma$ stands for ``diagonal sum''.

\begin{figure}[h]
\include{figures/diagonal_sum}
\caption{}
\label{fig:typeBcorners}
\end{figure}

Putting these results together, we have the following type $B$ result.
\begin{theorem}
    \label{thm:type-B-count}
    Let $W$ be a type $B$ Weyl group and, for $w \in W$, let $N(w^{-1}) = 
    \left\{ \alpha_{1}, \ldots, \alpha_{k} \right\}$
    be its inversion set  where $\alpha_i$ are roots in $\Phi^+$.
    Then the number of regions in the Weyl cone $C_w$ in the Shi arrangement $\shiarr$ is given by:
    \[
        \order{C_w} = \det\begin{pmatrix}
            1 & \weakabove{v^{tr}_2}{v^{bl}_1} & \cdots & \weakabove{v^{tr}_k}{v^{bl}_1} & \weakabove{I}{v^{bl}_1}\\
            \weakabove{v^{tr}_1}{v^{bl}_2} & 1 & \cdots & \weakabove{v^{tr}_k}{v^{bl}_2} & \weakabove{I}{v^{bl}_2}\\
            \vdots & \vdots & \ddots & \vdots & \vdots\\
            \weakabove{v^{tr}_1}{v^{bl}_k} & \weakabove{v^{tr}_2}{v^{bl}_k} & \cdots & 1 & \weakabove{I}{v^{bl}_k}\\
            \Fdiag{v^{tr}_1}{F} & \Fdiag{v^{tr}_2}{F} & \cdots & \Fdiag{v^{tr}_k}{F} & \Fdiag{I}{F}\\
        \end{pmatrix}
    \]
    where $I = (0,1)$, $F=(1,2n)$,
    $v^{tr}_\ell = (i_\ell, j_\ell+1)$ and $v^{bl}_\ell = (i_\ell - 1, j_\ell)$
    and $\gamma$ is the formula in \autoref{form_gamma_above} and \autoref{eq:gamma_Dsum}. 
\end{theorem}
\begin{proof}
    By \autoref{thm:num-in-cone} and \autoref{thm:nonoverlap-bij-antichains}, the number of regions in $C_w$ is precisely the number of paths from $I$ to $F$ which don't contain corners associated to the roots in $N(w^{-1})$ as subpaths.

    Letting $\Pi$ be the corners associated to roots in $N(w^{-1})$, then $\Pi$ is a collection of non-overlapping paths.
    By \autoref{thm:path-nonoverlapping} this implies $\order{C_w}$ is equal to a determinant.
    By setting $I = (0,1)$ we have $F = (1, 2n)$ and for $m \in [k]$ we have
    \begin{align*}
        \alpha_{i_m j_m} &:\; v^{tr}_m = (i_m, j_m+1)\text{~and~}v^{bl}_m = (i_m - 1, j_m),\\
        \alpha_{i_m j_m, p_m q_m} &:\; v^{tr}_m = (i_m, j_m+1 + (n-p_m))\text{~and~}v^{bl}_m = (i_m - 1, j_m + (n-p_m))
    \end{align*}
    In other words, all of our paths are precisely paths weakly above the main diagonal, \ie 
    \begin{align*}
        \gamma(F_i \to I_j) &= \weakabove{v^{tr}_i}{v^{bl}_j},\\
        \gamma(I \to I_j) &= \weakabove{I}{v^{bl}_j},\\
        \gamma(F_i \to F) &= \Fdiag{v^{tr}_i}{F},\text{~and}\\
        \gamma(I \to F) &= \Fdiag{I}{F}
    \end{align*}
    as desired.
\end{proof}
\begin{remark}
    We can simplify $\Fdiag{X}{F}$ to a smaller summation to help speed up computations if desired whenever $X$ is not on the final south east diagonal $x-y=0$, which is always the case in our matrix.
    We describe this process in this remark and use it in the following example.
    By the equation in \autoref{thm:count-above-diag}, if $X = (a, b)$ then we only have paths from $X$ to $F^{D_i} = (i, 2n-i)$ if $a \leq i$ and $b \leq 2n - i$.
    Then we have the following set of equalities.
    \begin{align*}
        \Fdiag{X}{F} &=  \sum_{i = a}^{\min(n, 2n-b)} \weakabove{(a,b)}{F^{D_i}}\\
        &=  \sum_{i = a}^{\min(n, 2n-b)} \bns{(2n-i) + i - a - b}{2n-i-b}{(2n-i)+i-a-b}{2n-i-a+1}\\
    \end{align*}
    Let $j = 2n-i+i - a - b=2n-a-b$ and $m = \min(n, 2n-b)$.
    Furthermore, as $b \geq a$ then we can let $b = a + c$ for $c \in \N$.
    Then
    \begin{align*}
        \Fdiag{X}{F} &=  \sum_{i = a}^{m} \binom{j}{j + a - i} - \binom{j}{j + b - i + 1}\\
        &=  \sum_{i = a}^{m} \binom{j}{j + a - i} - \binom{j}{j + a + c - i + 1}\\
        &=  \sum_{i = a}^{m} \binom{j}{j + a - i} - \sum_{i=a-c-1}^{m-c-1}\binom{j}{j + a - i}\\
    \end{align*}
    If $a \leq m-c-1$ then all terms (weakly) between $a$ and $m - c-1$ get cancelled.
    In other words, we have the following final equality:
    \begin{align*}
        \Fdiag{(a, b)}{F} &=  \sum_{i = \max(a, m-c)}^{m} \binom{j}{j + a - i} - \sum_{i = a-c-1}^{\min(a, m-c)-1}\binom{j}{j+a-i}\\
%        &= \sum_{i = \max(1, a, \min(n, 2n-b+1) - (b-a))}^{\min(n, 2n-b+1)} \binom{x}{x + a -i} - \sum_{i = \max(1, a)-(b-a)-1}^{\min(n - (b-a)-1, 2n-b-(b-a), \max(1,a)-1)} \binom{x}{x+a-i}
    \end{align*}
    where $j = 2n-a-b$, $m = \min(n, 2n-b)$ and $c = b-a$.
\end{remark}

\begin{example}
    Unlike in type $A$, we'll start with a complex example in type $B$.
    Suppose that the type $B_4$ Weyl group $W$ has presentation such that $(s_1s_2)^3 = (s_2s_3)^3 = (s_3s_4)^4 = e$ and all other simple reflections commute.
    Then $n = 4$ and litting $w = s_2 s_3 s_4 s_1$ we have

    \[
        N(w^{-1}) = \left\{ \alpha_{11}, \alpha_{44}, \alpha_{34,44}, \alpha_{14,44} \right\}
    \]

    The corners to avoid are given by:
    \begin{align*}
        \alpha_{11} = \alpha_1 &:\; (0,1) \to (1,1) \to (1,2)\\
        \alpha_{44} = \alpha_4 &:\; (3,4) \to (4,4) \to (4,5)\\
        \alpha_{34,44} = \alpha_3 + 2\alpha_4 &:\; (2,5) \to (3,5) \to (3,6)\\
        \alpha_{14,44} = \alpha_1 + \alpha_2 + \alpha_3 + 2\alpha_4 &:\; (0,5) \to (1,5) \to (1,6)
    \end{align*}
    
    Then, the number of regions is given by:
    {
    \scriptsize
    \[
        \det{\medmuskip = 1mu%
            \setlength{\arraycolsep}{1pt}
            \renewcommand{\arraystretch}{2.5}
            \begin{pmatrix}
                1 & \weakabove{(4,5)}{(0,1)} & \weakabove{(3,6)}{(0,1)} & \weakabove{(1,6)}{(0,1)} & \weakabove{(0,1)}{(0,1)}\\
            \weakabove{(1,2)}{(3,4)} & 1 & \weakabove{(3,6)}{(3,4)} & \weakabove{(1,6)}{(3,4)} & \weakabove{(0,1)}{(3,4)}\\
            \weakabove{(1,2)}{(2,5)} & \weakabove{(4,5)}{(2,5)} & 1 & \weakabove{(1,6)}{(2,5)} & \weakabove{(0,1)}{(2,5)}\\
            \weakabove{(1,2)}{(0,5)} & \weakabove{(4,5)}{(0,5)} & \weakabove{(3,6)}{(0,5)} & 1 & \weakabove{(0,1)}{(0,5)}\\
            \Fdiag{(1,2)}{(1,8)} & \Fdiag{(4,5)}{(1,8)} & \Fdiag{(3,6)}{(1,8)} & \Fdiag{(1,6)}{(1,8)} & \Fdiag{(0,1)}{(1,8)}\\
        \end{pmatrix}}
    \]
    }

    For the diagonal sums, recall that for $(a, b) \to F^{D\Sigma}$ we have $j = 2n - a - b$, $m = \min(n, 2n - b)$ and $c = b - a$.
    Calculating the diagonal sums, we have
    \begin{align*}
        \Fdiag{(1,2)}{(1,8)} %&=  \sum_{i=1}^4 \weakabove{(1,2)}{(i,8-i+1)}\\
        &= \sum_{i=3}^{4} \binom{5}{5 + 1 - i} - \sum_{i = -1}^{0} \binom{5}{5 + 1 - i}&(j = 5,\, m=4,\,c=1)\\
        &=  \binom{5}{3} + \binom{5}{2} - \binom{5}{7} - \binom{5}{6}\\
%        &= \bns{6}{6}{6}{8} + \bns{6}{5}{6}{7} + \bns{6}{4}{6}{6} + \bns{6}{3}{6}{5}\\
%        &= 1 - 0 + 6 - 0 + 15 - 1 + 20 - 6\\
        &= 10 + 10 - 0 - 0\\
        &= 20\\
        \Fdiag{(4,5)}{(1,8)} &=  1\\
        \Fdiag{(3,6)}{(1,8)} &=  1 \\
        \Fdiag{(1,6)}{(1,8)} %&=  \sum_{i=1}^3 \weakabove{(1,6)}{(i,8-i+1)}\\
%        &= \bns{2}{2}{2}{8} + \bns{2}{1}{2}{7} + \bns{2}{0}{2}{6}\\
%        &= 1-0+2-0+1-0+0\\
        &=\sum_{i=1}^{2} \binom{1}{1 + 1 - i} - \sum_{i = -5}^{-4} \binom{1}{1 + 1 - i} &(j = 1,\,m=2,\,c=5)\\
        &= \binom{1}{1} + \binom{1}{0} - \binom{1}{7} - \binom{1}{6}\\
        &= 1 + 1 + 0 + 0 \\
        &= 2\\
    \end{align*}
    \begin{align*}
        \Fdiag{(0,1)}{(1,8)} % &=  \sum_{i=1}^4 \weakabove{(0,1)}{(i,8-i+1)}\\
%        &= \bns{8}{7}{8}{9} + \bns{8}{6}{8}{8} + \bns{8}{5}{8}{7} + \bns{8}{4}{8}{6}\\
%        &= 8-0+28-1+56-8+70-28\\
        &= \sum_{i=3}^{4} \binom{7}{7+0-i} - \sum_{i=-2}^{-1} \binom{7}{7+0-i}&(j = 7,\,m=4,\,c=1)\\
        &= \binom{7}{4} + \binom{7}{3} - \binom{7}{9} - \binom{7}{8}\\
        &= 35 + 35 - 0 - 0\\
        &= 70\\
    \end{align*}

    Putting this together, we have
    \begin{align*}
        \abs{C_w} &= \det{
            \renewcommand{\arraystretch}{1.2}
        \begin{pmatrix}
            1 & 0 & 0 & 0 & \binom{0}{0} - \binom{0}{2}\\
            \binom{4}{2} - \binom{4}{4} & 1 & 0 & 0 & \binom{6}{3} - \binom{6}{5}\\
            \binom{4}{3} - \binom{4}{5} & 0 & 1 & 0 & \binom{6}{4} - \binom{6}{6}\\
            0 & 0 & 0 & 1 & \binom{4}{4} - \binom{4}{6}\\
            20 & 1 & 1 & 2 & 70\\
        \end{pmatrix}
        }\\
        &= 
        \det
        \begin{pmatrix}
            1 & 0 & 0 & 0 &1\\
            5 & 1 & 0 & 0 & 14\\
            4 & 0 & 1 & 0 & 14\\
            0 & 0 & 0 & 1 &1\\
            20 & 1 & 1 & 2 & 70\\
        \end{pmatrix}&\\
        &= 29
    \end{align*}
\end{example}

\subsection{Type \texorpdfstring{$D$}{D}}
\label{ssec:type-D-counting}

Getting an explicit formula for type $D$ is a little more complex since we can't
associate it to a $\Z^2$-lattice as in the type $A$ and $B$ cases.
This implies we need a different approach for calculating the number of paths.
Recall that $\Gamma_{D_n}$ has four parts where each of the two middle parts have
multiple diagrams.
We associate to each point of $\Gamma_{D_n}$ a coordinate   $(x,y)$ of the $\Z^2$-lattice together with a double index 
 which determines which 
part of the diagram $\Gamma_{D_n}$ we are in. More precisely, the first index 
${u\in\left\{ 1,2,3,4 \right\}}$ determines which of the four parts  $\Gamma_{D_n}^{u}$ we are considering
 while  in the  two middle cases, i.e., when  $u=2$ or $3$,  the second index   $v=1,\ldots,n-1$ 
   determines which multiple copy we are in. 
  For $u=1$ or 4 we set the default value  $v=0$,  since we have no multiple copies to consider. 
   Altogether, the possible pairs for $(u,v)$ are 
$(1,0),(4,0)$ and $(2,v), (3,v)$ with $ v =1,\ldots,n-1$. 
In what follows, we give examples for the points $(x, y)_{(u, v)}$ for $D_4$ in each case.
We break this down by the different parts of $\Gamma_{D_n}$.

\begin{itemize}
    \item[$\Gamma_{D_n}^1$:] For all vertices in $\Gamma_{D_n}^1$ we let $u = 1$ and $v = 0$.
        As $\Gamma_{D_n}^1$ looks like a type $A_{n-2}$ diagram, we let $x$ and $y$ be the
        type $A$ coordinates associated to each point.
        In other words, the vertex at the south west of the diagram has $x = 0$ and $y = 1$, \ie $I = (0, 1)_{(1, 0)}$.
            \begin{center}
                \input{figures/D4_1}
            \end{center}
    \item[$\Gamma_{D_n}^{2,i}$:] For all vertices in $\Gamma_{D_n}^{2,i}$ we let $u = 2$ and $v = i$.
        We assume we are on a $\Z^2$-lattice and place $\Gamma_{D_n}^{2, i}$ so that the bottom left most vertex is at $(i-1, 0)$.
            \begin{center}
                \input{figures/D4_2}
            \end{center}
    \item[$\Gamma_{D_n}^{3,i}$:] For all vertices in $\Gamma_{D_n}^{3,i}$ we let $u = 3$ and $v = i$.
        We assume we are on a $\Z^2$-lattice and place $\Gamma_{D_n}^{3, i}$ so that the bottom left most vertex is at $(0, 0)$.
            \begin{center}
                \input{figures/D4_3}
            \end{center}
    \item[$\Gamma_{D_n}^4$:] For all vertices in $\Gamma_{D_n}^4$ we let $u = 4$ and $v = 0$.
        As $\Gamma_{D_n}^4$ looks like the top half of a $B_{n-2}$ diagram we ``place'' $\Gamma_{D_{n}}^4$ as if it were $B_{n-2}$.
        This implies that the south west corner of $\Gamma_{D_n}^4$ would have $x = 0$ and $y = n-2$.
        In particular, $F = (1, 2n)_{(4, 0)}$.
            \begin{center}
                \input{figures/D4_4}
            \end{center}
\end{itemize}

Additionally, since a particular root might be associated to multiple boxes, when we
refer to a root, we will include \emph{all} boxes when describing corners.
\begin{example}
    Let us look at the coordinates associated to corners for certain roots in type $D_5$ in order to understand how the coordinates work.
    We use \autoref{fig:gamma_D} for $D_5$ as a reference.
    We have the following corners for the following roots:
    \begin{align*}
        \alpha_{22} &:\; (1, 2)_{(1, 0)} \to (2, 2)_{(1, 0)} \to (2, 3)_{(1, 0)}\\
        \alpha_{25} &:\; (1, 3)_{(4, 0)} \to (2, 3)_{(4, 0)} \to (2, 4)_{(4, 0)}\\
        \alpha_{34} &:\; (2, 0)_{(2, 1)} \to (3, 0)_{(2, 1)} \to (3, 1)_{(2, 1)}\\
        & \quad\text{and}\quad (2, 0)_{(2, 2)} \to (3, 0)_{(2, 2)} \to (3, 0)_{(2, 2)}\\
        & \quad\text{and}\quad (2, 0)_{(2, 3)} \to (3, 0)_{(2, 3)} \to (3, 0)_{(2, 3)}\\
    \end{align*}
    Notice that for $\alpha_{34}$, since there is no edge below $\alpha_{34}$ in $\Gamma_{D_n}^{2,4}$ there is no fourth corner associated to it.
\end{example}

As we are counting paths using \autoref{thm:path-nonoverlapping}, it suffices to find
the number of paths that start at either the final point $F_i$ of a corner $F_i$ or $I$; and that finish at either the initial
point $I_i$ of a corner or $F$.
With this in mind, we calculate $\gamma(V_1 \to V_2)$ where we let $V_1 = (x_1, y_1)_{(u_1, v_1)}$ and let $V_2 = (x_2, y_2)_{(u_2, v_2)}$.
We additionally suppose that $V_1$ is either $I$ or $F_i$ for some path $\pi_i$ and $V_2$ is either $F$ or $I_i$ for some path $\pi_i$.
We start with $u = 4$ and decrease from there.

First, notice that if $u_2 < u_1$ then $\gamma(V_1 \to V_2) = 0$.
Therefore, we only need to consider when $u_2 \geq u_1$.
For ease of notation we set the following notation:
\begin{align*}
    \gamma_B( (x_1, y_1) \to (x_2, y_2)) &= \begin{cases}
        \weakabove{(x_1, y_1)}{(x_2,y_2)} & \text{if }(x_2, y_2) \neq (1, 2n)\\
        \Fdiag{(x_1, y_1)}{(x_2,y_2)} & \text{if }(x_2, y_2) = (1, 2n)\\
    \end{cases}
\end{align*}
This is the precisely the type $B$ digraph counting formula from earlier where $(x_2, y_2) = (1, 2n)$ is precisely when $V_2 = F$.

\paragraph{\texorpdfstring{$\b{u_1 = 4}$}{u1=4}:} When both $V_1$ and $V_2$ are in $\Gamma_{D_n}^4$, then it's clear that
\[
    \gamma(V_1 \to V_2) = \gamma_B( (x_1, y_1) \to (x_2, y_2)).
\]

\paragraph{\texorpdfstring{$\b{u_1 = 3}$}{u1 = 3}:}
We suppose $V_1 \in \Gamma_{D_n}^{3,v_i}$.
Since $V_1$ must be a final vertex of a corner then $y_1 = 1$.
As $u_2 \geq u_1 = 3$, then $u_2 \in \left\{ 3, 4 \right\}$. 
If $u_2 = 3$, then $V_2$ must be an initial vertex of a corner implying $y_2 = 0$; forcing $\gamma(V_1 \to V_2) = 0$.
Otherwise, we suppose that $V_2 \in \Gamma_{D_n}^4$.
By construction, there is a unique path $\pi$ from $V_1$ to $\Gamma_{D_n}^4$ depending on the relationship between $x_1$ and $v_1$.
If $x_1 < v_1$ then this unique path $\pi$ is given by $V_1 \to (v_1-1, 1)_{(3, v_1)} \to (v_1 - 1, n-2)_{(4, 0)}$.
If $x_1 = n-1$ then this unique path $\pi$ is given by $V_1 \to (n-2, 1)_{(3, v_1)} \to (n-2, n-2)_{(4, 0)}$.
Finally, in all other cases, the unique path $\pi$ is given by $V_1 \to (x_1, n-2)_{(4, 0)}$.
Then it suffices to count the number of paths from $F(\pi)$ to $V_2$.
Putting this together we have:
\[
    \gamma(V_1 \to V_2) = \begin{cases}
        \gamma_B( (v_1 - 1, n-2) \to (x_2, y_2))
            & \text{if } x_1 < v_1\\
        \gamma_B( (x_1, n-2) \to (x_2, y_2))
            & \text{if } v_1 \leq x_1 < n-1\\
        \gamma_B( (n-2, n-2) \to (x_2, y_2) )
            & \text{if } x_1 = n-1 \\
    \end{cases}
\]

\paragraph{\texorpdfstring{$\b{u_1 = 2}$}{u1=2}:}
We suppose $V_1 \in \Gamma_{D_n}^{2,v_i}$.
As in the previous case, we know $y_1 = 1$.
Similarly, if $u_2 = 2$ then $\gamma(V_1 \to V_2) = 0$.
Therefore we suppose that $u_2 > u_1$ and we break  this into two cases depending on if $u_2 = 3$ or $u_2 = 4$.

If $u_2 = 3$ then $y_2 = 0$ as $V_2$ must be the inital vertex of a corner in $\Gamma_{D_n}^{3}$.
Since $v_1$ tells us that $V_1 \in \Gamma_{D_n}^{2, v_1}$ we know by the construction of $\Gamma_{D_n}$ that 
there is a unique edge leaving $V_1$.
This edge goes to $(v_1-1, 0)_{(3, x)}$ where $x = \min(x_1 +1, n-1)$.
Therefore, there is a path to $V_2$ preciesly when $x = v_2$ (or else they're in different diagrams and no path exists) and when $x_2 \geq v_1-1$ in which case there is precisely one path to get to $V_2$.
In other words:
\[
    \gamma( (x_1, 1)_{(2, v_1)} \to (x_2, 0)_{(3, v_2)}) = \begin{cases}
        1& \text{if } x_2 \geq v_1-1 \text{~and~} v_2 = \min(x_1+1, n-1)\\
        0 & \text{otherwise}
    \end{cases}
\]

Finally, if $u_2 = 4$, then $V_2 \in \Gamma_{D_n}^4$.
As before, there is a unique path from $V_1$ to $(v_1 - 1, 0)_{(3, v)}$ where $v = \min(x_1 + 1, n-1)$.
In other words,
\[
    \gamma( V_1 \to V_2) = \gamma\left( (v_1-1, 0)_{(3, v)} \to V_2 \right).
\]
First, we calculate how to get to the bottom row in $\Gamma_{D_n}^4$ from $(v_1 - 1, 0)_{(3, v)}$.
In $\Gamma_{D_n}^{3, v}$ recall that the first $v-1$ horizontal edges on the top row are pointed east and the final edge is pointed west with all other top row edges removed.
Since $v_1 - 1 < v$, there is at least one path to $(j, 1)_{(3, v)}$ for all $v_1 - 1 \leq j \leq n-2$.
As $v = \min(x_1 + 1, n-1)$, we break this into two parts.
If $v = n-1$, then there are $(n - v_1 + 1)$ paths to $(n-2, 1)_{(3, v)}$.
If $v = x_1 + 1 < n-1$, then there are $(v - v_1 + 1) = (x_1 - v_1 + 2)$ paths to $(j, 1)_{(3, v)}$ whenever $v_1 - 1 \leq j \leq v - 1 = x_1$, there is precisely one path to $(j, 1)_{(3, v)}$ whenever $x_1 = v - 1 < j < n-2$ and there are two paths to $(j, 1)_{(3, v)}$ whenever $j = n-2$.
From $(j, 1)_{(3, v)}$ there is then a unique path to $(j, n-2)_{(4,0)}$.
To finish off the count we not that $\gamma\left( (j, n-2)_{(4,0)} \to V_2 \right)$ is then given by $\gamma_B\left( (j, n-2) \to (x_2, y_2) \right)$.
Putting this altogether, we have
\[
    \gamma(V_1 \to V_2) = \begin{cases}
        (n - v_1+1)\gamma_B( (n-2, n-2) \to (x_2, y_2))
            & \text{if } x_1 \geq n-2 \\
           2 \gamma_B( (n-2, n-2) \to (x_2, y_2))\\
            \qquad + (x_1-v_1 + 2)\gamma_B( (x_1, n-2) \to (x_2, y_2))\\
            \qquad+ \sum\limits_{i = x_1 + 1}^{n-3}\gamma_B( (i, n-2) \to (x_2, y_2))
            & \text{if } x_1 < n-2\\
    \end{cases}
\]

\paragraph{\texorpdfstring{$\b{u_1=1}$}{u1=1}:}
As a final case, we suppose $V_1 \in \Gamma_{D_n}^4$.
Like with the previous cases, we break this down into the four components, with the last component being the most complex.

If $u_2 = 1$, then $V_2$ is the beginning vertex of a corner.
In other words $y_2 \leq n-2$ and since $\Gamma_{D_n}^1$ is a type $A$ digraph we have
\[
    \gamma( V_1 \to V_2) = \weakabove{(x_1, y_1)}{(x_2, y_2)}.
\]

If $u_2 = 2$, then, again, $V_2$ is the beginning vertex of a corner, \ie $y_2 = 0$.
Since the bottom edges of $\Gamma_{D_n}^{2,v_2}$ are directed east, there is a unique path from $(v_2 - 1, n-1)_{(1, 0)}$ to $V_2$.
In other words
\[
    \gamma(V_1 \to V_2) = \gamma(V_1 \to (v_2 - 1, n-1)_{(1, 0)})
\]
But since there are no edges in the top row of $\Gamma_{D_n}^1$, we have:
\[
    \gamma(V_1 \to V_2) = \gamma(V_1 \to (v_2 - 1, n-2)_{(1, 0)}) = \weakabove{(x_1, y_1)}{(v_2 - 1, n-2)}.
\]

If $u_2 = 3$, then as before, $V_2$ is the beginning vertex of a corner, \ie $y_2 = 0$.
Since all the bottom edges of $\Gamma_{D_n}^{3, v_2}$ are directed east, then there is a path to $V_2$ coming from $\Gamma_{D_n}^2$ for each $0 \leq i \leq \min(x_2, v_2 - 1)$.
If $v_2 \neq n-1$ then each of the $i$ has a unique path to $\Gamma_{D_n}^1$ to the point $(i, n-1)_{(1, 0)}$.
If $v_2 = n-1$ then each $i$ has two paths to the point $(n-2, n-1)_{(1, 0)}$ in $\Gamma_{D_n}^1$.
Finally, as in the previous case, since there are no edges in the top row of $\Gamma_{D_n}^1$, we can reduce the $y$ value of all of these points and calculate using the digraph of type $A$.
In other words:
\[
    \gamma(V_1 \to V_2) = \begin{cases}
        \sum\limits_{i = 0}^{\min(x_2, v_2 - 1)} \weakabove{(x_1, y_1)}{(i, n-2)} & \text{if }v_2 \neq n-1\\
        \sum\limits_{i = 0}^{\min(x_2, v_2 - 1)} 2 \weakabove{(x_1, y_1)}{(i, n-2)} & \text{if } v_2 = n-1
    \end{cases}
\]

If $u_2 = 4$ then we are in for a treat as we must traverse all four parts of our digraph.
We start from $V_2$ and work our way down.
From a bottom vertex $(i, n-2)_{(4, 0)}$ in $\Gamma_{D_n}^4$ we have $\gamma_B( (i, n-2) \to (x_2, y_2))$ number of paths where $0 \leq i \leq n-2$.
For each $i$, there are $i + 1$ paths to $(i, n-2)_{(4, 0)}$ from $\Gamma_{D_n}^3$.
In particular, there is a path from $(i, 1)_{(3, j)} \in \Gamma_{D_n}^{3, j}$ to $(i, n-2)_{(4, 0)}$ for each $1 \leq j \leq i + 1$.
For each $j$, there are $\lambda_{i, j, k}$ paths from $(i, 1)_{(3,j)}$ to $(k, 0)_{(3, j)}$ for $0 \leq k \leq j - 1$ where
\[
    \lambda_{i,j,k} = \begin{cases}
        2 & \text{if }i = n-2\\
        1 & \text{if } n-2 > i > j-1\\
        j-k+1 &\text{if }i = j-1
    \end{cases}
\]
Ther is a unique edge from $\Gamma_{D_n}^2$ to $(k, 0)_{(3, j)}$ and this edge originates from $(j-1, 1)_{(2, k+1)}$.
If $j-1 \neq n-2$ then there is exactly one path from $(k, 0)_{(2, k+1)}$ to $(j - 1, 1)_{(2, k+1)}$, else there are two paths.
This is encapsulated in the variable $\mu_j$ where
\[
    \mu_j = \begin{cases}
        1 & \text{if }j < n-1\\
        2 & \text{if }j = n-1
    \end{cases}
\]
Finally, as there is only one path from $(k, n-2)_{(1, 0)}$ to $(k, 0)_{(2, k+1)}$, it suffices to count paths from $V_1$ to $(k, n-2)_{(1, 0)}$.
Putting this together gives:
\begin{equation}
    \label{eq:1-4}
    \gamma(V_1 \to V_2) = \sum_{i=0}^{n-2} \sum_{j = 1}^{i+1} \sum_{k = 0}^{j-1} \mu_j \lambda_{i,j,k} \weakabove{(x_1, y_1)}{(k, n-2)} \cdot \gamma_B( (i, n-2) \to (x_2, y_2))
\end{equation}

Putting this altogether, we have the following lemma.
\begin{lemma}
\label{lem:type_D_counting}
    Let $V_1 = (x_1, y_1)_{(u_1, v_1)}$ and $V_2 = (x_2, y_2)_{(u_2, v_2)}$ be two vertices in $\Gamma_{D_n}$ where $V_1$ is either $I$ or a final point $F_i$ of some path and where $V_2$ is either $F$ or an initial point $I_i$ of some path.
    Then
    \[
        \gamma(V_1 \to V_2) = \begin{cases}
            \weakabove{(x_1, y_1)}{(x_2, y_2)}
                & \text{if }u_1 = 1,\,u_2 = 1\\
            \weakabove{(x_1, y_1)}{(v_2 - 1, n-2)}
                & \text{if }u_1 = 1,\,u_2 = 2\\
            \sum\limits_{i = 0}^{\min(x_2, v_2 - 1)} \weakabove{(x_1, y_1)}{(i, n-2)}
                & \text{if }u_1 = 1,\,u_2=3,\,v_2 \neq n-1\\
            \sum\limits_{i = 0}^{\min(x_2, v_2 - 1)} 2\weakabove{(x_1, y_1)}{(i, n-2)}
                & \text{if }u_1=1,\,u_2=3,\, v_2 = n-1\\
            \gamma_D(V_1 \to V_2)
                & \text{if } u_1=1,\,u_2 = 4\\
            1
                & \text{if } u_{1}=2,\,u_2=3,\,x_2 \geq v_1-1 \text{~and~}\\
                & \qquad\qquad v_2 = \min(x_1+1, n-1)\\
            (n - v_1+1)\gamma_B( (n-2, n-2) \to (x_2, y_2))
                & \text{if } u_1 = 2,\,u_2=4,\,x_1 \geq n-2\\
            (x_1 - v_1 + 2)\gamma_B( (x_1, n-2) \to (x_2, y_2))\\
                \qquad+ 2\gamma_B( (n-2, n-2) \to (x_2, y_2))\\
                \qquad+ \sum\limits_{j = x_1+1}^{n-3}\gamma_B( (j, n-2) \to (x_2, y_2))
                & \text{if } u_1=2,\,u_2=4,\,x_1 < n-2\\
            \gamma_B( (v_1 - 1, n-2) \to (x_2, y_2))
                & \text{if } u_1=3,\,u_2=4,\,x_1 < v_1\\
            \gamma_B( (x_1, n-2) \to (x_2, y_2))
                & \text{if } u_1=3,\,u_2=4,\,v_1 \leq x_1 < n-1\\
            \gamma_B( (n-2, n-2) \to (x_2, y_2) )
                & \text{if } u_1= 3,\,u_2=4,\,x_1 = n-1\\
            \gamma_{B} \left( (x_1, y_1) \to (x_2, y_2) \right)
                & \text{if }u_1 = 4,\,u_2=4\\
            0 & \text{otherwise}
        \end{cases}
    \]
    where $\gamma_D(V_1 \to V_2)$ is equal to equation~\autoref{eq:1-4}.
\end{lemma}
Due to the nature of this graph, we do not expect an easier formula to appear.
Additionally, although this formula might seem complicated and have many cases, algorithmically (using a computer) these are extremely fast to calculate due to the nature of $\gamma_B$ and $\gamma_A$.

\newpage
Putting this all together, we have the following theorem.
\begin{theorem}
    \label{thm:type-D-count}
    Let $W$ be a type $D$ Weyl group and, for $w \in W$, let $N(w^{-1}) = 
    \left\{ \alpha_{1}, \ldots, \alpha_{k} \right\}$
    be its inversion set  where $\alpha_i$ are roots in $\Phi^+$.
    Let $\Pi$ be the set of corners
    \[
        \pi_i = (p-1, q)_{( s, t)} \to (p, q)_{(s, t)} \to (p, q+1)_{(s,t)}
    \]
    associated to the root $\alpha_i \in N(w^{-1})$.
    Then the number of regions in the Weyl cone $C_w$ in the Shi arrangement $\shiarr$ is given by:
    \[
        \order{C_w} = \det\begin{pmatrix}
            1 & \gamma(F_2 \to I_1) & \cdots & \gamma(F_k \to I_1) & \gamma(I \to I_1)\\
            \gamma(F_1 \to I_2) & 1 & \cdots & \gamma(F_k \to I_2) & \gamma(I \to I_2)\\
            \vdots & \vdots & \ddots & \vdots & \vdots\\
            \gamma(F_1 \to I_k) & \gamma(F_2 \to I_k) & \cdots & 1 & \gamma(I \to I_k)\\
            \gamma(F_1 \to F) & \gamma(F_2 \to F) & \cdots & \gamma(F_k \to F) & \gamma(I \to F)\\
        \end{pmatrix}
    \]
    where $I = (0,1)_{(1, 0)}$, $F=(1,2n)_{(4, 0)}$, $I_i = (p-1, q)_{(s, t)}$, $F_i = (p, q+1)_{(s,t)}$
    and $\gamma$ is the formula in \autoref{lem:type_D_counting}.

\end{theorem}
\begin{proof}
    This is a corollary of \autoref{thm:nonoverlapping} and \autoref{lem:type_D_counting}.
\end{proof}

\section{Narayana numbers}
\label{sec:narayana_numbers}
Given a Weyl group of type $X$, recall that the number of regions in the dominant cone of $\shiarr$ of type $X$ is given by the Catalan number of type $X$.
In particular, by \autoref{thm:bij-dom-anti}, the Catalan number of type $X$ is precisely the number of antichains in the root poset of type $X$. 
The \defn{Narayana number $\nar_{X, k}$} is then the number of antichains with cardinality $k$ in the root poset of type $X$ and the \defn{Narayana polynomial $\nar_X(t)$} is the polynomial $\sum \nar_{X, k}t^k$. 

Following the terminology in  \cite{DBStump},
if $P$ is an arbitrary poset,  the 
\defn{Poincar\'e number $\poin_{P,k}$} is  the number of antichains in $P$ with cardinality $k$  and  the \defn{Poincar\'e polynomial $\poin_P(t)$} is the polynomial $\sum \poin_{P, k}t^k$.
In this section we give a determinental formula for the Poincar\'e polynomial for each subposet  $\Phi^+_w$ of $\Phi^+$. 
Recalling that the map of  Theorem \ref{thm:bij-dom-anti}  bijects antichains to {\em separating walls}, 
the Poincar\'e polynomial gives  a refined counting of the Shi regions in the cone $C_w$, according to the number of  separating walls. 
Notice that in the case of the dominant cone, the Poincar\'e polynomial coincides with the Narayana one. 

\subsection{Formula}
Given a digraph $\Gamma_W$ for a Weyl group $W$ recall that the corner associated to a box is a length $2$ path which starts from the bottom left vertex $v^{bl}$, goes to the bottom right vertex $v^{br}$ and terminates at the top right vertex $v^{tr}$.
Given a path $\pi \in \Gamma_W$ from $v_1$ to $v_2$, we let $c(\pi)$ be the number of corners which are subpaths of $\pi$.

Although we would like to let the weight of $\gamma(v_1 \to v_2)$ be equal to $t^{c(\pi)}$ and apply \autoref{thm:nonoverlapping}, we cannot by our definition of weights.
In particular, since $\gamma(v_1 \to v_2 \to v_3) = \gamma(v_1 \to v_2) \cdot \gamma(v_2 \to v_3)$, if we let $v_1 \to v_2 \to v_3$ be some corner, then the left-hand side of this equation gives $\gamma(v_1 \to v_2 \to v_3) = t^1$ (as there is one corner) and the right hand side gives $\gamma(v_1 \to v_2) \cdot \gamma(v_2 \to v_3) = 1 \cdot 1 = 1$ as an edge can never be a corner and we get $t^1 = 1$.
Therefore $t = 1$ and everything breaks down.

It turns out that even though we can't define $\wt(\pi) = t^{c(\pi)}$, we can still count the number of corners in an identical way to \autoref{thm:nonoverlapping}.
For this we show that, due to our particular choice of non-overlapping paths, a version of \autoref{thm:nonoverlapping} where we count the number of corners still holds.
In the following, we let 
\[
    \gamma'(v_1 \to v_2) = \sum_{\pi \in \set{\pi}{I(\pi) = v_1,\,F(\pi) = v_2}} t^{c(\pi)}.
\]
We define $\gamma'(v_1 \to v_2 \st \text{some property})$ in a similar way following \autoref{ssec:non-overlapping_paths}.
%\et{why don't we write $\gamma'(\Pi) = \sum_{\pi \in \Pi} t^{c(\pi)}$ ? } \ad{I guess I had written the above wrong since it's technically not $\Pi$ we should be worried about. I've updated the text accordingly.}
\begin{theorem}
    \label{thm:narayana-poly}
    Let $\Gamma_W$ be the digraph associated to a Weyl group $W$ and let $\Pi = \left\{ \pi_1, \ldots, \pi_n \right\}$ be a collection of corners.
 Consider the $(n+1)\times (n+1)$ matrix 
    \begin{align*}
        M_{\Pi} &=  \begin{pmatrix}
            \gamma'{(I_1 \to I_1)} & \gamma'{(I_2 \xrightarrow{\pi_2} F_2 \to I_1)} & \cdots & \gamma'{(I_{n} \xrightarrow{\pi_n} F_{n} \to I_1)} & \gamma'{(I \to I_1)}\\
            \gamma'{(I_1 \xrightarrow{\pi_1} F_1 \to I_2)} & \gamma'{(I_2 \to I_2)} & \cdots & \gamma'{(I_{n} \xrightarrow{\pi_n} F_{n} \to I_2)} & \gamma'{(I \to I_2)}\\
            \vdots & \vdots & \ddots & \vdots & \vdots\\
            \gamma'{(I_1 \xrightarrow{\pi_1} F_1 \to I_{n})} & \cdots & \cdots & \gamma'{(I_{n} \to I_{n})} & \gamma'{(I \to I_n)}\\
            \gamma'{(I_1 \xrightarrow{\pi_1} F_1 \to F)} & \gamma'{(I_2 \xrightarrow{\pi_2} F_2 \to F)} & \cdots & \gamma'{(I_{n} \xrightarrow{\pi_n} F_{n} \to F)} & \gamma'{(I \to F)}\\
        \end{pmatrix}\\\\
        &= \begin{pmatrix}
            1 & t\cdot\gamma'{(F_2 \to I_1)} & \cdots & t\cdot\gamma'{(F_{n} \to I_1)} & \gamma'{(I \to I_1)}\\
            t\cdot\gamma'{(F_1 \to I_2)} & 1 & \cdots & t\cdot\gamma'{(F_n\to I_2)} & \gamma'{(I \to I_2)}\\
            \vdots & \vdots & \ddots & \vdots & \vdots\\
            t\cdot\gamma'{(F_1 \to I_{n})} & t\cdot \gamma'(F_2 \to I_n) & \cdots & 1 & \gamma'{(I \to I_n)}\\
            t\cdot\gamma'{(F_1 \to F)} & t\cdot\gamma'{(F_2 \to F)} & \cdots & t\cdot\gamma'{(F_{n} \to F)} & \gamma'{(I \to F)}\\
        \end{pmatrix}.
    \end{align*}
\medskip 
    Then $\gamma'\left( I \to F \st \text{ has no subpath in }\Pi \right) = \det(M_{\Pi})$.
\end{theorem}
\begin{proof}
    Just as with \autoref{thm:nonoverlapping}, the proof of this theorem is almost identical to \autoref{thm:path-nonoverlapping} with a few minor changes.
    As in the previous two theorems, for the matrix $M_{\Pi}$, we let  
    \begin{align*}
    a_{ij} = \begin{cases} 
        \gamma'{(I_1 \to I_1)} & \text{ for } 1\leq i = j \leq n \\
        \gamma'{(I_j \xrightarrow{\pi_j} F_j\to I_i)} & \text{ for } 1\leq i \neq j \leq n \\
        \gamma'{(F_j \to F)} & \text{ for } 1\leq j \leq n \text{ and } i =n+1\\ 
	    \gamma'{(I \to I_i)} & \text{ for } 1\leq i \leq n \text{ and } j =n+1\\ 
        \gamma'{(I\to F)}  & \text{ for } i =j=n+1
        \end{cases}
    \end{align*}
    where $a_{ij}$ is the entry in the $i$th row and the $j$th column.

    It can be verified that
    \[
        \gamma'(v_1 \xrightarrow{\pi_1} v_2 \xrightarrow{\pi_2} v_3) = \gamma'(v_1 \xrightarrow{\pi_1} v_2) \cdot \gamma'(v_2 \xrightarrow{\pi_2}v_3)
    \]
    if and only if $\pi_1$ does not end in an east step or if $\pi_2$ does not start with a north step.
    This falls naturally since $\gamma'$ is counting the number of corners and a corner is precisely a length $2$ subpath which is an east step followed by a north step.
    Since every corner $\pi \in \Pi$ starts with an east step and ends in an north step, we have
    \[
        \gamma'(I_{i} \xrightarrow{\pi_i} F_i \to V)
        = \gamma'(I_{i} \xrightarrow{\pi_i} F_i) \cdot \gamma'(F_i \to V) 
        = t \cdot \gamma'(F_i \to V)
    \]
    where either $V = F$ or $V = I_j$ for some $j$.
    
    The rest of the proof stays identical.
\end{proof}

\begin{example} 
        Consider again the cone in \autoref{example1}.
	We count paths which do not contain the corners  $\alpha_{12}$ or $ \alpha_{22} $, according to their number of corners.

	
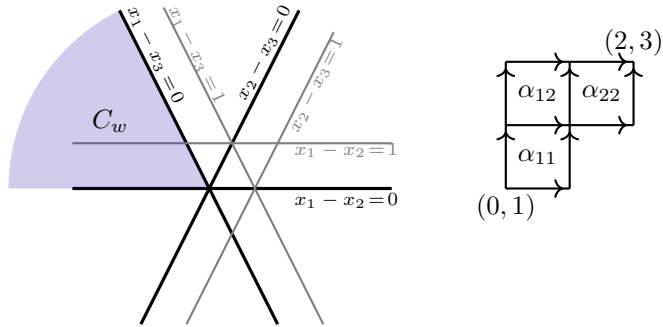
\begin{center}
		\begin{tikzpicture}[scale=0.6]
			\begin{scope}[shift={(0,0)}]
				\clip (0,0) circle (4.4cm);
				\draw[domain=-3:4,variable=\x,line width=1.2,black] plot ({\x},0);
				\draw[domain=-3:4,variable=\x,line width=1.2,black] 
				plot({cos(60)*\x},{\x});
				\draw[domain=-3:4,variable=\x,line width=1.2,black]plot({cos(120)*\x},{\x});
				%%%%%
				\draw[domain=-3:4,variable=\x,line width=0.8,gray]plot({cos(120)*\x+1},{\x});
				\draw[domain=-3:4,variable=\x,line width=0.8,gray]plot({cos(60)*\x+1},{\x});
				\draw[domain=-3:4,variable=\x,line width=0.8,gray]plot({\x},{1});
				
				\node at (3,-0.27){\scriptsize\color{black} $x_1-x_2\!=\!0$};
				\node at (1.2,3){\rotatebox{64}{\scriptsize\color{black} $x_2-x_3\!=\!0$}};
				\node at (-1.2,3){\rotatebox{-64}{\scriptsize\color{black} $x_1-x_3=\!0\!$}};
				
				\node at (3,0.85){\scriptsize\color{gray} $x_1-x_2\!=\!1$};	
				\node at (2.35,2.3){\rotatebox{64}{\scriptsize\color{gray} $x_2-x_3\!=\!1$}};
				\node at (-0.3,3){\rotatebox{-64}{\scriptsize\color{gray} $x_1-x_3=\!1\!$}};
				\fill[ablue,opacity=0.2] (0,0)--(117:5.5)--(180:5.5)--cycle;
				\node at (-2.16506350946109, 1.5) {$C_w$};
			\end{scope}
			\begin{scope}[shift={(6.5,0)}, scale=1.4,thick,decoration={markings,mark=at position 0.9 with {\arrow{>}}}] 
				\draw[postaction={decorate}] (0,0) -- (0,1);
				\draw[postaction={decorate}] (0,0) -- (1,0);
				\draw[postaction={decorate}] (0,1) -- (1,1);
				\draw[postaction={decorate}] (0,1) -- (0,2);
				\draw[postaction={decorate}] (0,2) -- (1,2);
				\draw[postaction={decorate}] (1,0) -- (1,1);
				\draw[postaction={decorate}] (1,1) -- (1,2);
				\draw[postaction={decorate}] (1,1) -- (2,1);
				\draw[postaction={decorate}] (1,2) -- (2,2);
				\draw[postaction={decorate}] (2,1) -- (2,2);
				\node at (0.5,0.5) {$\alpha_{11}$};
				\node at (0.5,1.5) {$\alpha_{12}$};
				\node at (1.5,1.5) {$\alpha_{22}$};
			\node at (0,-0.3){$(0,1)$};
			\node at (2,2.3){$(2,3)$};
			\end{scope}
		\end{tikzpicture}
	\end{center}

	In view of  \autoref{thm:narayana-poly}, we have
	\begin{align*}
		\order{C_w} &= \det
        {
\renewcommand{\arraystretch}{1.6}
		\begin{pmatrix}
			1 & t\,\weakaboveNar{(2,3)}{(0,2)} & \weakaboveNar{(0,1)}{(0,2)}\\
			t\,\weakaboveNar{(1,3)}{(1,2)} & 1 & \weakaboveNar{(0,1)}{(1,2)}\\
			t\,\weakaboveNar{(1,3)}{(2,3)} & t\,\weakaboveNar{(2,3)}{(2,3)} & \weakaboveNar{(0, 1)}{(2, 3)}\\
		\end{pmatrix}
    }
		\\
		&= \det
		\begin{pmatrix}
			1 & 0 & 1\\
			0 & 1 & 1+t \\
		 t  & t  &  1+3t+t^2\\
		\end{pmatrix} \\
        &=  1+t 
	\end{align*}
\end{example}

\begin{remark}
	We remark here that in the type A case the generating polynomials appearing in the  determinant of
	\autoref{thm:narayana-poly} can be computed using \cite[Theorem 10.14.1]{Krattenthaler_LatticePathEnumeration} which states that 
	all lattice paths from $(a,b)$ to $(c,d)$ weakly above $x=y$ with exactly $\ell$ corners 
%	 EN-turns, are 
  are given by 
	$\binom{d-b}{\ell}\binom{c-a}{\ell}-\binom{d-a-1}{\ell-1}\binom{c-b+1}{\ell+1}$. 
	It would be interesting to find such a refined  enumeration for paths in the 
	graphs of type $B$ and $D$. 
\end{remark}

%%%%%%%%%%%%%%%%%%%%%%%%%%%%%%%%%%%%%%%%%%%%%%
% Add Index
%%%%%%%%%%%%%%%%%%%%%%%%%%%%%%%%%%%%%%%%%%%%%%
\printindex

%%%%%%%%%%%%%%%%%%%%%%%%%%%%%%%%%%%%%%%%%%%%%%
% For bibliography with filename biblio.bib
%%%%%%%%%%%%%%%%%%%%%%%%%%%%%%%%%%%%%%%%%%%%%%
\nocite{*}
\printbibliography[title={References}]
\label{sec:biblio}

\begin{sidewaysfigure}
    \vspace*{170mm}
    \begin{center}
        \input{figures/F4}
        \input{figures/G2}
        \caption{On the left is the diagram $\Gamma_{F_4}$ for the $F_4$ Shi arrangement and on the right is the diagram $\Gamma_{G_2}$ for the $G_2$ Shi arrangement.}
        \label{fig:gamma_F}
    \end{center}
\end{sidewaysfigure}
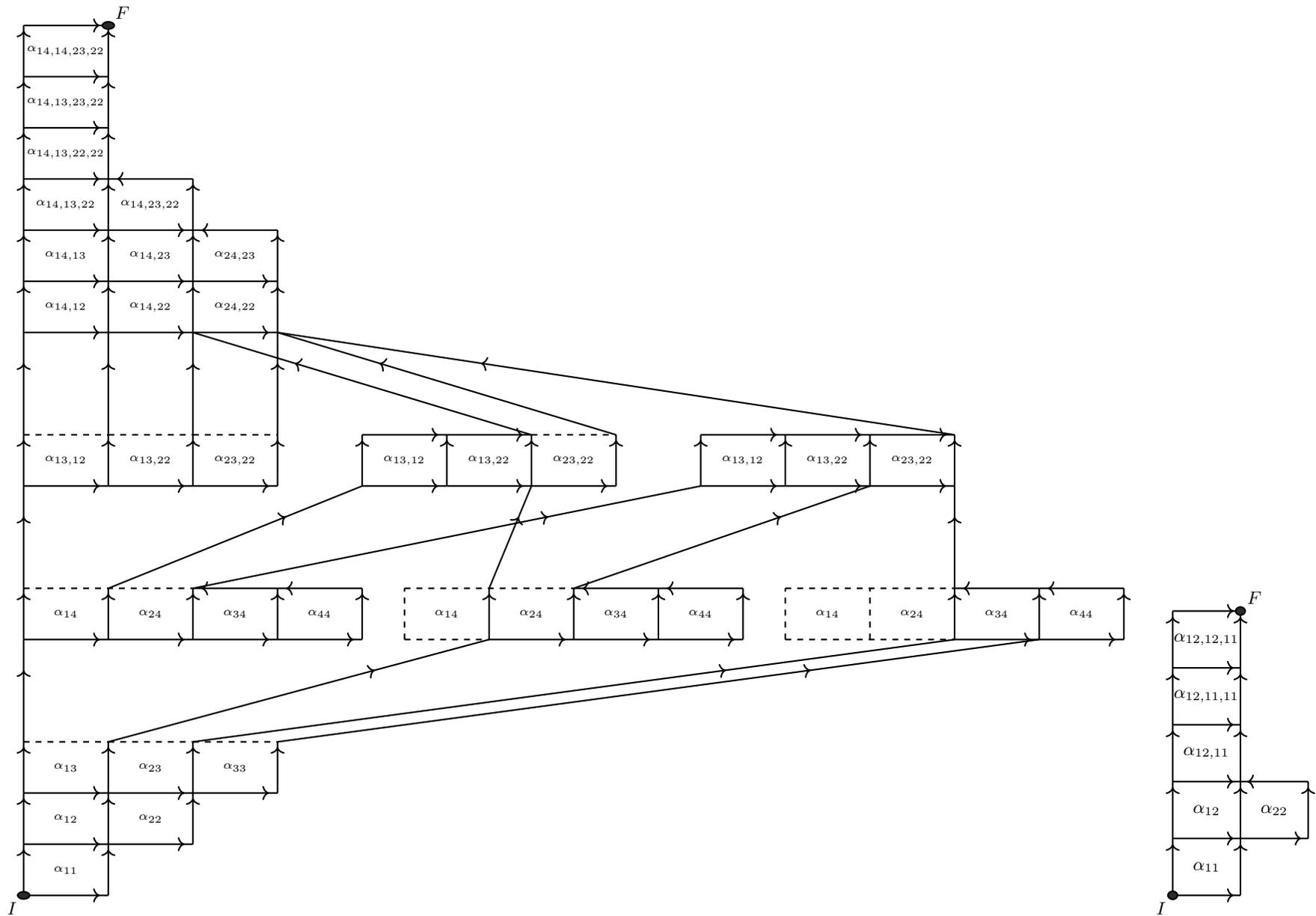

\begin{sidewaysfigure}
    \vspace*{170mm}
    \centering
        \input{figures/E6}
        \caption{The diagram $\Gamma_{E_6}$ for the $E_6$ Shi arrangement. The number represent the different roots found in \autoref{fig:E6_root}.}
        \label{fig:gamma_E6}
\end{sidewaysfigure}
\begin{sidewaysfigure}
    \vspace*{170mm}
    \begin{center}
        \input{figures/E6root}
        \caption{The figure on the left is the root poset labelled by roots while the figure on the right is the root poset labelled by numbers as associated with \autoref{fig:gamma_E6}.}
        \label{fig:E6_root}
    \end{center}
\end{sidewaysfigure}
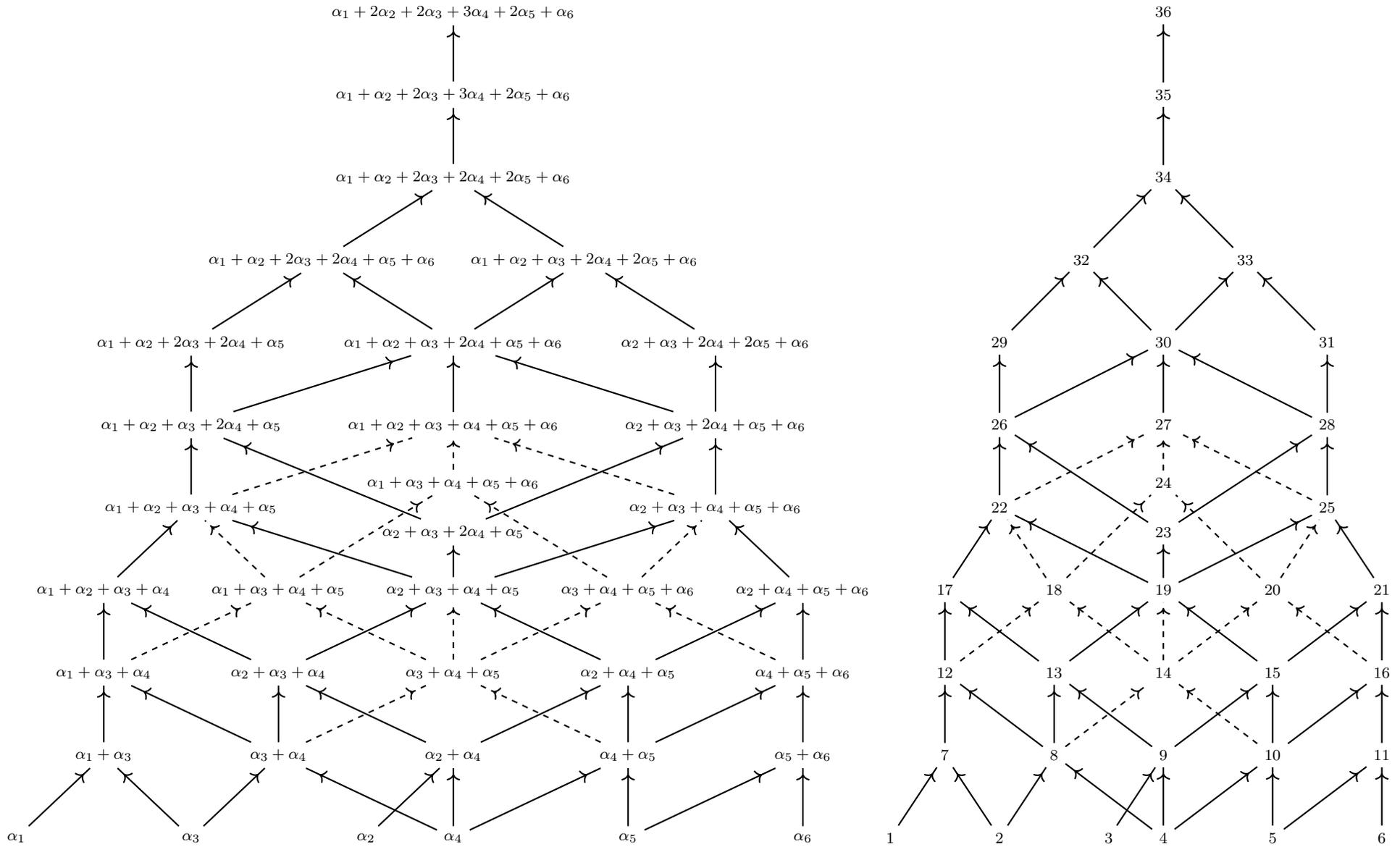

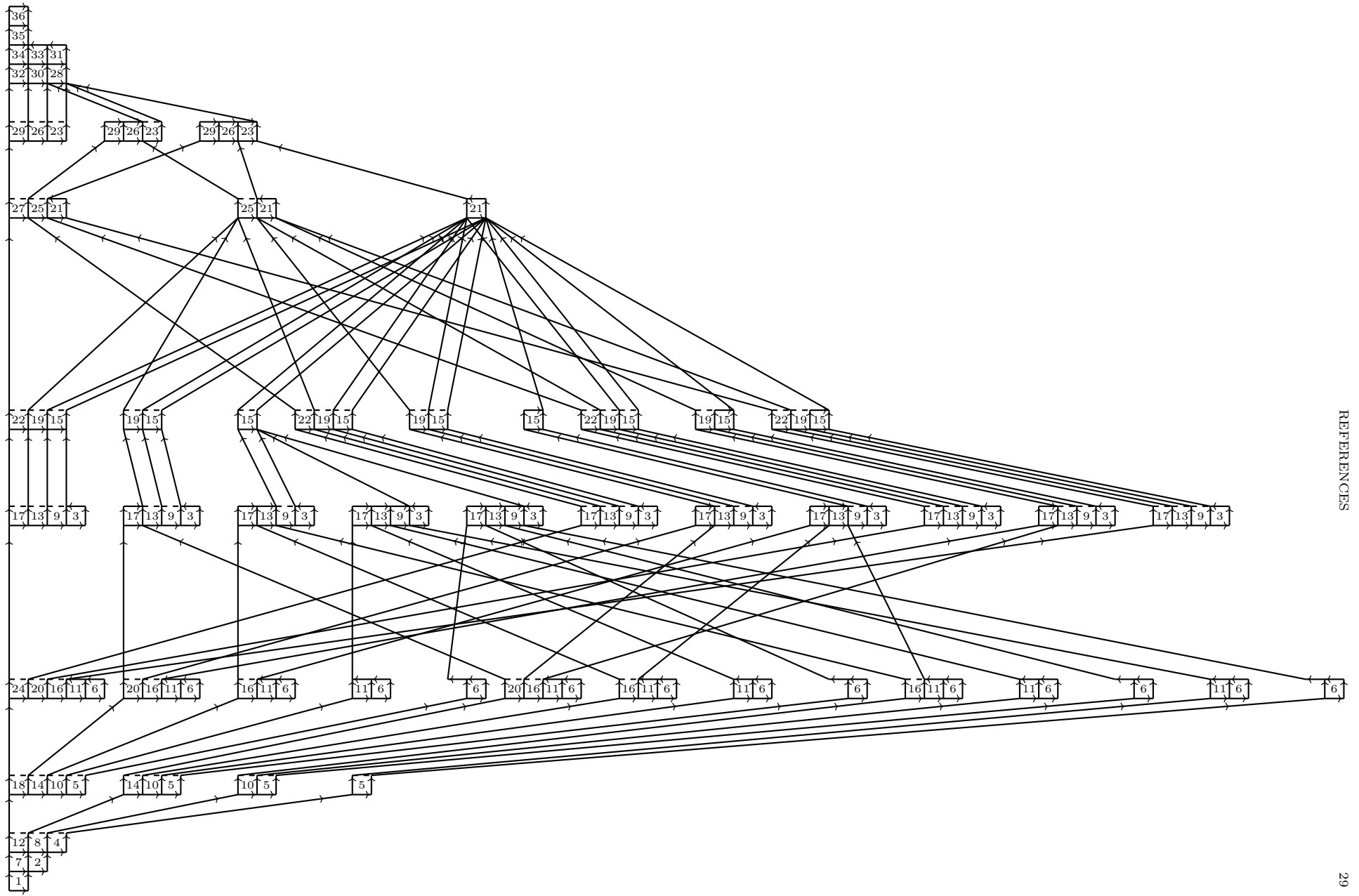
\end{document}

%% file: figures/A7B4.tex
        \begin{tikzpicture}
            [thick,decoration={markings,mark=at position 0.9 with {\arrow{>}}}]    
            \begin{scope}[shift={(0,0)}]
                \draw[postaction={decorate}] (0,0) -- (0,1);
                \draw[postaction={decorate}] (0,1) -- (0,2);
                \draw[postaction={decorate}] (0,2) -- (0,3);
                \draw[postaction={decorate}] (0,3) -- (0,4);
                \draw[postaction={decorate}] (0,4) -- (0,5);
                \draw[postaction={decorate}] (0,5) -- (0,6);
                \draw[postaction={decorate}] (0,6) -- (0,7);
                \draw[postaction={decorate}] (0,0) -- (1,0);
                \draw[postaction={decorate}] (0,1) -- (1,1);
                \draw[postaction={decorate}] (0,2) -- (1,2);
                \draw[postaction={decorate}] (0,3) -- (1,3);
                \draw[postaction={decorate}] (0,4) -- (1,4);
                \draw[postaction={decorate}] (0,5) -- (1,5);
                \draw[postaction={decorate}] (0,6) -- (1,6);
                \draw[postaction={decorate}] (0,7) -- (1,7);
                \draw[postaction={decorate}] (1,0) -- (1,1);
                \draw[postaction={decorate}] (1,1) -- (1,2);
                \draw[postaction={decorate}] (1,2) -- (1,3);
                \draw[postaction={decorate}] (1,3) -- (1,4);
                \draw[postaction={decorate}] (1,4) -- (1,5);
                \draw[postaction={decorate}] (1,5) -- (1,6);
                \draw[postaction={decorate}] (1,6) -- (1,7);
                \draw[postaction={decorate}] (1,1) -- (2,1);
                \draw[postaction={decorate}] (1,2) -- (2,2);
                \draw[postaction={decorate}] (1,3) -- (2,3);
                \draw[postaction={decorate}] (1,4) -- (2,4);
                \draw[postaction={decorate}] (1,5) -- (2,5);
                \draw[postaction={decorate}] (1,6) -- (2,6);
                \draw[postaction={decorate}] (1,7) -- (2,7);
                \draw[postaction={decorate}] (2,1) -- (2,2);
                \draw[postaction={decorate}] (2,2) -- (2,3);
                \draw[postaction={decorate}] (2,3) -- (2,4);
                \draw[postaction={decorate}] (2,4) -- (2,5);
                \draw[postaction={decorate}] (2,5) -- (2,6);
                \draw[postaction={decorate}] (2,6) -- (2,7);
                \draw[postaction={decorate}] (2,2) -- (3,2);
                \draw[postaction={decorate}] (2,3) -- (3,3);
                \draw[postaction={decorate}] (2,4) -- (3,4);
                \draw[postaction={decorate}] (2,5) -- (3,5);
                \draw[postaction={decorate}] (2,6) -- (3,6);
                \draw[postaction={decorate}] (2,7) -- (3,7);
                \draw[postaction={decorate}] (3,2) -- (3,3);
                \draw[postaction={decorate}] (3,3) -- (3,4);
                \draw[postaction={decorate}] (3,4) -- (3,5);
                \draw[postaction={decorate}] (3,5) -- (3,6);
                \draw[postaction={decorate}] (3,6) -- (3,7);
                \draw[postaction={decorate}] (3,3) -- (4,3);
                \draw[postaction={decorate}] (3,4) -- (4,4);
                \draw[postaction={decorate}] (3,5) -- (4,5);
                \draw[postaction={decorate}] (3,6) -- (4,6);
                \draw[postaction={decorate}] (3,7) -- (4,7);
                \draw[postaction={decorate}] (4,3) -- (4,4);
                \draw[postaction={decorate}] (4,4) -- (4,5);
                \draw[postaction={decorate}] (4,5) -- (4,6);
                \draw[postaction={decorate}] (4,6) -- (4,7);
                \draw[postaction={decorate}] (4,4) -- (5,4);
                \draw[postaction={decorate}] (4,5) -- (5,5);
                \draw[postaction={decorate}] (4,6) -- (5,6);
                \draw[postaction={decorate}] (4,7) -- (5,7);
                \draw[postaction={decorate}] (5,4) -- (5,5);
                \draw[postaction={decorate}] (5,5) -- (5,6);
                \draw[postaction={decorate}] (5,6) -- (5,7);
                \draw[postaction={decorate}] (5,5) -- (6,5);
                \draw[postaction={decorate}] (5,6) -- (6,6);
                \draw[postaction={decorate}] (5,7) -- (6,7);
                \draw[postaction={decorate}] (6,5) -- (6,6);
                \draw[postaction={decorate}] (6,6) -- (7,6);
                \draw[postaction={decorate}] (6,7) -- (7,7);
                \draw[postaction={decorate}] (6,6) -- (6,7);
                \draw[postaction={decorate}] (7,6) -- (7,7);

                \draw[fill=Black] (0, 0) circle[radius=2pt, fill=Black] node[below left] {$I$};
                \draw[fill=Black] (7, 7) circle[radius=2pt, fill=Black] node[above right] {$F$};

                \node at (0.5, 0.5) {$\alpha_{11}$};
                \node at (0.5, 1.5) {$\alpha_{12}$};
                \node at (0.5, 2.5) {$\alpha_{13}$};
                \node at (0.5, 3.5) {$\alpha_{14}$};
                \node at (0.5, 4.5) {$\alpha_{15}$};
                \node at (0.5, 5.5) {$\alpha_{16}$};
                \node at (0.5, 6.5) {$\alpha_{17}$};
                \node at (1.5, 1.5) {$\alpha_{22}$};
                \node at (1.5, 2.5) {$\alpha_{23}$};
                \node at (1.5, 3.5) {$\alpha_{24}$};
                \node at (1.5, 4.5) {$\alpha_{25}$};
                \node at (1.5, 5.5) {$\alpha_{26}$};
                \node at (1.5, 6.5) {$\alpha_{27}$};
                \node at (2.5, 2.5) {$\alpha_{33}$};
                \node at (2.5, 3.5) {$\alpha_{34}$};
                \node at (2.5, 4.5) {$\alpha_{35}$};
                \node at (2.5, 5.5) {$\alpha_{36}$};
                \node at (2.5, 6.5) {$\alpha_{37}$};
                \node at (3.5, 3.5) {$\alpha_{44}$};
                \node at (3.5, 4.5) {$\alpha_{45}$};
                \node at (3.5, 5.5) {$\alpha_{46}$};
                \node at (3.5, 6.5) {$\alpha_{47}$};
                \node at (4.5, 4.5) {$\alpha_{55}$};
                \node at (4.5, 5.5) {$\alpha_{56}$};
                \node at (4.5, 6.5) {$\alpha_{57}$};
                \node at (5.5, 5.5) {$\alpha_{66}$};
                \node at (5.5, 6.5) {$\alpha_{67}$};
                \node at (6.5, 6.5) {$\alpha_{77}$};
                
            \end{scope}
            \begin{scope}[shift={(8,0)}, xscale=1.2]
                \draw[postaction={decorate}] (0,0) -- (0,1);
                \draw[postaction={decorate}] (0,1) -- (0,2);
                \draw[postaction={decorate}] (0,2) -- (0,3);
                \draw[postaction={decorate}] (0,3) -- (0,4);
                \draw[postaction={decorate}] (0,4) -- (0,5);
                \draw[postaction={decorate}] (0,5) -- (0,6);
                \draw[postaction={decorate}] (0,6) -- (0,7);
                \draw[postaction={decorate}] (0,0) -- (1,0);
                \draw[postaction={decorate}] (0,1) -- (1,1);
                \draw[postaction={decorate}] (0,2) -- (1,2);
                \draw[postaction={decorate}] (0,3) -- (1,3);
                \draw[postaction={decorate}] (0,4) -- (1,4);
                \draw[postaction={decorate}] (0,5) -- (1,5);
                \draw[postaction={decorate}] (0,6) -- (1,6);
                \draw[postaction={decorate}] (0,7) -- (1,7);
                \draw[postaction={decorate}] (1,0) -- (1,1);
                \draw[postaction={decorate}] (1,1) -- (1,2);
                \draw[postaction={decorate}] (1,2) -- (1,3);
                \draw[postaction={decorate}] (1,3) -- (1,4);
                \draw[postaction={decorate}] (1,4) -- (1,5);
                \draw[postaction={decorate}] (1,5) -- (1,6);
                \draw[postaction={decorate}] (1,6) -- (1,7);
                \draw[postaction={decorate}] (1,1) -- (2,1);
                \draw[postaction={decorate}] (1,2) -- (2,2);
                \draw[postaction={decorate}] (1,3) -- (2,3);
                \draw[postaction={decorate}] (1,4) -- (2,4);
                \draw[postaction={decorate}] (1,5) -- (2,5);
                \draw[postaction={decorate}] (2,6) -- (1,6);
                \draw[postaction={decorate}] (2,1) -- (2,2);
                \draw[postaction={decorate}] (2,2) -- (2,3);
                \draw[postaction={decorate}] (2,3) -- (2,4);
                \draw[postaction={decorate}] (2,4) -- (2,5);
                \draw[postaction={decorate}] (2,5) -- (2,6);
                \draw[postaction={decorate}] (2,2) -- (3,2);
                \draw[postaction={decorate}] (2,3) -- (3,3);
                \draw[postaction={decorate}] (2,4) -- (3,4);
                \draw[postaction={decorate}] (3,5) -- (2,5);
                \draw[postaction={decorate}] (3,2) -- (3,3);
                \draw[postaction={decorate}] (3,3) -- (3,4);
                \draw[postaction={decorate}] (3,4) -- (3,5);
                \draw[postaction={decorate}] (3,3) -- (4,3);
                \draw[postaction={decorate}] (4,4) -- (3,4);
                \draw[postaction={decorate}] (4,3) -- (4,4);
                \node at (0.5, 0.5) {$\alpha_{11}$};
                \node at (0.5, 1.5) {$\alpha_{12}$};
                \node at (0.5, 2.5) {$\alpha_{13}$};
                \node at (0.5, 3.5) {$\alpha_{14}$};
                \node at (1.5, 1.5) {$\alpha_{22}$};
                \node at (1.5, 2.5) {$\alpha_{23}$};
                \node at (1.5, 3.5) {$\alpha_{24}$};
                \node at (2.5, 2.5) {$\alpha_{33}$};
                \node at (2.5, 3.5) {$\alpha_{34}$};
                \node at (3.5, 3.5) {$\alpha_{44}$};
                \node at (0.5, 4.5) {$\alpha_{14,44}$};
                \node at (0.5, 5.5) {$\alpha_{14,34}$};
                \node at (0.5, 6.5) {$\alpha_{14,24}$};
                \node at (1.5, 4.5) {$\alpha_{24,44}$};
                \node at (1.5, 5.5) {$\alpha_{24,34}$};
                \node at (2.5, 4.5) {$\alpha_{34,44}$};
                \draw[fill=Black] (1, 7) circle[radius=2pt, fill=Black] node[above right] {$F$};
            \end{scope}
            \draw[fill=Black] (8, 0) circle[radius=2pt, fill=Black] node[below left] {$I$};
        \end{tikzpicture}

%% file: figures/D5_antichains.tex
\begin{tikzpicture}[xscale=1.4]
   \begin{scope}[shift={(0,0)}, scale=0.8]
	\draw[thick] (0,0) -- (0,7);
	\draw[thick] (1,0) -- (1,7);
	\draw[thick] (2,1) -- (2,6);
	\draw[thick] (3,2) -- (3,5);
	\draw[thick] (4,3) -- (4,4);
	\draw[thick] (0,0) -- (1,0);
	\draw[thick] (0,1) -- (2,1);
	\draw[thick] (0,2) -- (3,2);
	\draw[thick] (0,3) -- (4,3);
	\draw[thick] (0,4) -- (4,4);
	\draw[thick] (0,5) -- (3,5);
	\draw[thick] (0,6) -- (2,6);
	\draw[thick] (0,7) -- (1,7);
	
	\draw[thick] (0,4) -- (1, 3);
	\draw[thick] (1,4) -- (2, 3);
	\draw[thick] (2,4) -- (3, 3);
	\draw[thick] (3,4) -- (4, 3);
    
    \fill[green,opacity=0.1] (0,7) rectangle (1,2);
    \fill[green,opacity=0.1] (1,6) rectangle (2,2);

	\node at (0.5, 0.5) {\tiny$\alpha_{11}$};
	\node at (0.5, 1.5) {\tiny$\alpha_{12}$};
	\node at (1.5, 1.5) {\tiny$\alpha_{22}$};
	\node at (0.5, 2.5) {\tiny$\alpha_{13}$};
	\node at (1.5, 2.5) {\tiny$\alpha_{23}$};
	\node at (2.5, 2.5) {\tiny$\alpha_{33}$};
	
	\node at (0.38, 3.22) {\tiny$\alpha_{13,55}$};
	\node at (0.7, 3.75) {\tiny$\alpha_{14}$};
	\node at (1.38, 3.22) {\tiny$\alpha_{23,55}$};
	\node at (1.67, 3.75) {\tiny$\alpha_{24}$};
	\node at (2.38, 3.22) {\tiny$\alpha_{33,55}$};
	\node at (2.7, 3.75) {\tiny$\alpha_{34}$};
	\node at (3.38, 3.22) {\tiny$\alpha_{55}$};
	\node at (3.75, 3.75) {\tiny$\alpha_{44}$};
	
	\node at (0.5, 4.5) {\tiny$\alpha_{15}$};
	\node at (1.5, 4.5) {\tiny$\alpha_{25}$};
	\node at (2.5, 4.5) {\tiny$\alpha_{35}$};
	\node at (0.5, 5.5) {\tiny$\alpha_{15,33}$};
	\node at (1.5, 5.5) {\tiny$\alpha_{25,33}$};
	\node at (0.5, 6.5) {\tiny$\alpha_{15,23}$};
\end{scope}
\begin{scope}[shift={(4,0)}, scale=0.8]
	\draw[thick] (0,0) -- (0,7);
	\draw[thick] (1,0) -- (1,7);
	\draw[thick] (2,1) -- (2,6);
	\draw[thick] (3,2) -- (3,5);
	\draw[thick] (4,3) -- (4,4);
	\draw[thick] (0,0) -- (1,0);
	\draw[thick] (0,1) -- (2,1);
	\draw[thick] (0,2) -- (3,2);
	\draw[thick] (0,3) -- (4,3);
	\draw[thick] (0,4) -- (4,4);
	\draw[thick] (0,5) -- (3,5);
	\draw[thick] (0,6) -- (2,6);
	\draw[thick] (0,7) -- (1,7);
	
	\draw[thick] (0,4) -- (1, 3);
	\draw[thick] (1,4) -- (2, 3);
	\draw[thick] (2,4) -- (3, 3);
	\draw[thick] (3,4) -- (4, 3);
	
    \fill[green,opacity=0.1] (0,7) rectangle (1,4);
    \fill[green,opacity=0.1] (0,3) -- (0,4) -- (1,3) -- cycle;
	
	\node at (0.5, 0.5) {\tiny$\alpha_{11}$};
	\node at (0.5, 1.5) {\tiny$\alpha_{12}$};
	\node at (1.5, 1.5) {\tiny$\alpha_{22}$};
	\node at (0.5, 2.5) {\tiny$\alpha_{13}$};
	\node at (1.5, 2.5) {\tiny$\alpha_{23}$};
	\node at (2.5, 2.5) {\tiny$\alpha_{33}$};
	
	\node at (0.38, 3.22) {\tiny$\alpha_{13,55}$};
	\node at (0.7, 3.75) {\tiny$\alpha_{14}$};
	\node at (1.38, 3.22) {\tiny$\alpha_{23,55}$};
	\node at (1.67, 3.75) {\tiny$\alpha_{24}$};
	\node at (2.38, 3.22) {\tiny$\alpha_{33,55}$};
	\node at (2.7, 3.75) {\tiny$\alpha_{34}$};
	\node at (3.38, 3.22) {\tiny$\alpha_{55}$};
	\node at (3.75, 3.75) {\tiny$\alpha_{44}$};
	
	\node at (0.5, 4.5) {\tiny$\alpha_{15}$};
	\node at (1.5, 4.5) {\tiny$\alpha_{25}$};
	\node at (2.5, 4.5) {\tiny$\alpha_{35}$};
	\node at (0.5, 5.5) {\tiny$\alpha_{15,33}$};
	\node at (1.5, 5.5) {\tiny$\alpha_{25,33}$};
	\node at (0.5, 6.5) {\tiny$\alpha_{15,23}$};
\end{scope}
\end{tikzpicture}

%% file: figures/D5_example.tex
   \begin{tikzpicture}
   \begin{scope}[shift={(0,0)}, scale=0.8]
	\draw[thick] (0,0) -- (0,7);
	\draw[thick] (1,0) -- (1,7);
	\draw[thick] (2,1) -- (2,6);
	\draw[thick] (3,2) -- (3,5);
	\draw[thick] (4,3) -- (4,4);
	\draw[thick] (0,0) -- (1,0);
	\draw[thick] (0,1) -- (2,1);
	\draw[thick] (0,2) -- (3,2);
	\draw[thick] (0,3) -- (4,3);
	\draw[thick] (0,4) -- (4,4);
	\draw[thick] (0,5) -- (3,5);
	\draw[thick] (0,6) -- (2,6);
	\draw[thick] (0,7) -- (1,7);
	
	\draw[thick] (0,4) -- (1, 3);
	\draw[thick] (1,4) -- (2, 3);
	\draw[thick] (2,4) -- (3, 3);
	\draw[thick] (3,4) -- (4, 3);
    
    \fill[blue,opacity=0.1] (0,7) rectangle (1,3);
    \fill[blue,opacity=0.1] (1,6) rectangle (2,3);
    \fill[blue,opacity=0.1] (2,5) rectangle (3,4);
    \fill[blue,opacity=0.1] (3,3) -- (4,3) -- (3,4) -- cycle;
    \fill[blue,opacity=0.1] (2,3) -- (3,3) -- (2,4) -- cycle;

	\draw[blue,line width=1.5pt] (3,3) -- (4, 3)--(3, 4);
	\draw[blue,line width=1.5pt] (1,4) -- (2, 3)--(2,4);
	\node at (0.5, 0.5) {\tiny$\alpha_{11}$};
	\node at (0.5, 1.5) {\tiny$\alpha_{12}$};
	\node at (1.5, 1.5) {\tiny$\alpha_{22}$};
	\node at (0.5, 2.5) {\tiny$\alpha_{13}$};
	\node at (1.5, 2.5) {\tiny$\alpha_{23}$};
	\node at (2.5, 2.5) {\tiny$\alpha_{33}$};
	
	\node at (0.4, 3.25) {\tiny$\alpha_{13,55}$};
	\node at (0.7, 3.75) {\tiny$\alpha_{14}$};
	\node at (1.4, 3.25) {\tiny$\alpha_{23,55}$};
	\node at (1.67, 3.75) {\tiny$\alpha_{24}$};
	\node at (2.4, 3.25) {\tiny$\alpha_{33,55}$};
	\node at (2.7, 3.75) {\tiny$\alpha_{34}$};
	\node at (3.4, 3.25) {\tiny$\alpha_{55}$};
	\node at (3.75, 3.75) {\tiny$\alpha_{44}$};
	
	\node at (0.5, 4.5) {\tiny$\alpha_{15}$};
	\node at (1.5, 4.5) {\tiny$\alpha_{25}$};
	\node at (2.5, 4.5) {\tiny$\alpha_{35}$};
	\node at (0.5, 5.5) {\tiny$\alpha_{15,33}$};
	\node at (1.5, 5.5) {\tiny$\alpha_{25,33}$};
	\node at (0.5, 6.5) {\tiny$\alpha_{15,23}$};
\end{scope}
   \begin{scope}[shift={(4,0)}, scale=0.8]
	\draw[thick] (0,0) -- (0,7);
	\draw[thick] (1,0) -- (1,7);
	\draw[thick] (2,1) -- (2,6);
	\draw[thick] (3,2) -- (3,5);
	\draw[thick] (4,3) -- (4,4);
	\draw[thick] (0,0) -- (1,0);
	\draw[thick] (0,1) -- (2,1);
	\draw[thick] (0,2) -- (3,2);
	\draw[thick] (0,3) -- (4,3);
	\draw[thick] (0,4) -- (4,4);
	\draw[thick] (0,5) -- (3,5);
	\draw[thick] (0,6) -- (2,6);
	\draw[thick] (0,7) -- (1,7);
	
	\draw[thick] (0,4) -- (1, 3);
	\draw[thick] (1,4) -- (2, 3);
	\draw[thick] (2,4) -- (3, 3);
	\draw[thick] (3,4) -- (4, 3);
	
    \draw[red,line width=1.5pt] (3,4) -- (4, 3)--(4,4);
	\draw[red,line width=1.5pt] (1,3) -- (2, 3)--(1,4);

    \fill[red,opacity=0.1] (0,7) rectangle (1,3);
    \fill[red,opacity=0.1] (1,6) rectangle (2,3);
    \fill[red,opacity=0.1] (2,5) rectangle (3,4);
    \fill[red,opacity=0.1] (4,4) -- (4,3) -- (3,4) -- cycle;
    \fill[red,opacity=0.1] (3,4) -- (3,3) -- (2,4) -- cycle;
	
	\node at (0.5, 0.5) {\tiny$\alpha_{11}$};
	\node at (0.5, 1.5) {\tiny$\alpha_{12}$};
	\node at (1.5, 1.5) {\tiny$\alpha_{22}$};
	\node at (0.5, 2.5) {\tiny$\alpha_{13}$};
	\node at (1.5, 2.5) {\tiny$\alpha_{23}$};
	\node at (2.5, 2.5) {\tiny$\alpha_{33}$};
	
	\node at (0.4, 3.25) {\tiny$\alpha_{13,55}$};
	\node at (0.7, 3.75) {\tiny$\alpha_{14}$};
	\node at (1.4, 3.25) {\tiny$\alpha_{23,55}$};
	\node at (1.67, 3.75) {\tiny$\alpha_{24}$};
	\node at (2.4, 3.25) {\tiny$\alpha_{33,55}$};
	\node at (2.7, 3.75) {\tiny$\alpha_{34}$};
	\node at (3.4, 3.25) {\tiny$\alpha_{55}$};
	\node at (3.75, 3.75) {\tiny$\alpha_{44}$};
	
	\node at (0.5, 4.5) {\tiny$\alpha_{15}$};
	\node at (1.5, 4.5) {\tiny$\alpha_{25}$};
	\node at (2.5, 4.5) {\tiny$\alpha_{35}$};
	\node at (0.5, 5.5) {\tiny$\alpha_{15,33}$};
	\node at (1.5, 5.5) {\tiny$\alpha_{25,33}$};
	\node at (0.5, 6.5) {\tiny$\alpha_{15,23}$};
\end{scope}
%\end{tikzpicture}
%\hspace{-2cm}
% \begin{tikzpicture}[scale=0.4]
   \begin{scope}[shift={(8,0)}, scale=0.45]	
	% \draw[decorate,decoration={markings,mark=at position 0.7 with {\arrow{>}}}]          	
	% \draw[red,decorate,decoration={markings,mark=at position 0.7 with {\arrow{>}}}]  (0,0)--(1,1);          	
%	\begin{scope}[shift={(0,0)},thick,decoration={markings,mark=at position 0.7 with {\arrow{>}}}]    
	\begin{scope}

		% Between 1, 2
		\draw[postaction={decorate}] (0,3) -- (0,5);
		\draw[postaction={decorate}] (1,3) -- (6,5);
		\draw[postaction={decorate}] (2,3) -- (12,5);
		\draw[postaction={decorate}] (3,3) -- (18,5);
		
		% Between 2, 3
		\draw[postaction={decorate}] (0,6) -- (0,8);
		\draw[postaction={decorate}] (1,6) -- (5,8);
		\draw[postaction={decorate}] (2,6) -- (10,8);
		\draw[postaction={decorate}] (3,6) -- (15,8);
		\draw[postaction={decorate}] (6,6) -- (6,8);
		\draw[postaction={decorate}] (7,6) -- (11,8);
		\draw[postaction={decorate}] (8,6) -- (16,8);
		\draw[postaction={decorate}] (12,6) -- (12,8);
		\draw[postaction={decorate}] (13,6) -- (17,8);
		\draw[postaction={decorate}] (18,6) -- (18,8);
		
		% Between
		\draw[postaction={decorate}] (0,9) -- (0,11);
		\draw[postaction={decorate}] (1,9) -- (1,11);
		\draw[postaction={decorate}] (2,9) -- (2,11);
		\draw[postaction={decorate}] (3,9) -- (3,11);
		\draw[postaction={decorate}] (6,9) -- (1,11);
		\draw[postaction={decorate}] (7,9) -- (2,11);
		\draw[postaction={decorate}] (8,9) -- (3,11);
		\draw[postaction={decorate}] (12,9) -- (2,11);
		\draw[postaction={decorate}] (13,9) -- (3,11);
		\draw[postaction={decorate}] (18,9) -- (3,11);
	\end{scope}
	
	\begin{scope}
		%		[shift={(0,0)},thick,decoration={markings,mark=at position 0.9 with {\arrow{>}}}]    
%	--(4,6)--(3,6)--(17,8)--(17,9)--(18,9)--(3,11)--(3,12)--(2,12)--(2,13)--(1,13)--(1,14);	
		% Bottom
		
		\draw[postaction={decorate}] (0,0) -- (0,1);
		\draw[postaction={decorate}] (0,1) -- (0,2);
		\draw[postaction={decorate}] (0,2) -- (0,3);
		%        \end{scope}
	\draw[postaction={decorate}] (0,0) -- (0,1);
	\draw[postaction={decorate}] (0,1) -- (0,2);
	\draw[postaction={decorate}] (0,2) -- (0,3);
	\draw[postaction={decorate}] (0,0) -- (1,0);
	\draw[postaction={decorate}] (0,1) -- (1,1);
	\draw[postaction={decorate}] (0,2) -- (1,2);
	\draw[thick, dashed] (0,3) -- (1,3);
	\draw[postaction={decorate}] (1,0) -- (1,1);
	\draw[postaction={decorate}] (1,1) -- (1,2);
	\draw[postaction={decorate}] (1,2) -- (1,3);
	\draw[postaction={decorate}] (1,1) -- (2,1);
	\draw[postaction={decorate}] (1,2) -- (2,2);
	\draw[thick, dashed] (1,3) -- (2,3);
	\draw[postaction={decorate}] (2,1) -- (2,2);
	\draw[postaction={decorate}] (2,2) -- (2,3);
	\draw[postaction={decorate}] (2,2) -- (3,2);
	\draw[thick, dashed] (2,3) -- (3,3);
	\draw[postaction={decorate}] (3,2) -- (3,3);
	\draw[fill=black] (0, 0) circle[radius=2pt, fill=black] node[below left] {$I$};
%	\node at (0.5, 0.5) {$\alpha_1$};
%	\node at (0.5, 1.5) {$\alpha_{12}$};
%	\node at (0.5, 2.5) {$\alpha_{13}$};
%	\node at (1.5, 1.5) {$\alpha_{2}$};
%	\node at (1.5, 2.5) {$\alpha_{23}$};
%	\node at (2.5, 2.5) {$\alpha_{3}$};
%	
	
	% Mid Bot
	% 1
	\draw[postaction={decorate}] (0,5) -- (1,5);
	\draw[postaction={decorate}] (1,5) -- (2,5);
	\draw[postaction={decorate}] (2,5) -- (3,5);
	\draw[postaction={decorate}] (3,5) -- (4,5);
	\draw[postaction={decorate}] (0,5) -- (0,6);
	\draw[postaction={decorate}] (1,5) -- (1,6);
	\draw[postaction={decorate}] (2,5) -- (2,6);
	\draw[postaction={decorate}] (3,5) -- (3,6);
	\draw[postaction={decorate}] (4,5) -- (4,6);
	\draw[thick, dashed] (0,6) -- (3,6);
	
	\draw[postaction={decorate}] (3.6,6) -- (3.5,6);
	\draw[thick] (4,6) -- (3,6);
	
%	\node at (0.5, 5.5) {$\alpha_{14}$};
	\node at (1.5, 5.5) {{\tiny$\alpha_{24}$}};
%	\node at (2.5, 5.5) {$\alpha_{34}$};
	\node at (3.5, 5.5) {{\tiny$\alpha_{44}$}};
	% 2
	\draw[postaction={decorate}] (6,5) -- (7,5);
	\draw[postaction={decorate}] (7,5) -- (8,5);
	\draw[postaction={decorate}] (8,5) -- (9,5);
	\draw[postaction={decorate}] (6,5) -- (6,6);
	\draw[postaction={decorate}] (7,5) -- (7,6);
	\draw[postaction={decorate}] (8,5) -- (8,6);
	\draw[postaction={decorate}] (9,5) -- (9,6);
	\draw[thick, dashed] (5,5) -- (5,6);
	\draw[thick, dashed] (5,5) -- (6,5);
	\draw[thick, dashed] (5,6) -- (8,6);
	%					\draw[postaction={decorate}] (9,6) -- (8,6);
	\draw[thick] (9,6) -- (8,6);
	\draw[postaction={decorate}] (8.6,6) -- (8.5,6);
%	\node at (5.5, 5.5) {$\alpha_{14}$};
%	\node at (6.5, 5.5) {$\alpha_{24}$};
%	\node at (7.5, 5.5) {$\alpha_{34}$};
%	\node at (8.5, 5.5) {$\alpha_{4}$};
	% 3
	\draw[postaction={decorate}] (12,5) -- (13,5);
	\draw[postaction={decorate}] (13,5) -- (14,5);
	\draw[postaction={decorate}] (12,5) -- (12,6);
	\draw[postaction={decorate}] (13,5) -- (13,6);
	\draw[postaction={decorate}] (14,5) -- (14,6);
	\draw[thick, dashed] (10,5) -- (10,6);
	\draw[thick, dashed] (11,5) -- (11,6);
	\draw[thick, dashed] (10,5) -- (12,5);
	\draw[thick, dashed] (10,6) -- (13,6);
	%					\draw[postaction={decorate}] (14,6) -- (13,6);
	\draw[thick] (14,6) -- (13,6);
	\draw[postaction={decorate}] (13.5,6) -- (13.4,6);
%	\node at (10.5, 5.5) {$\alpha_{14}$};
%	\node at (11.5, 5.5) {$\alpha_{24}$};
%	\node at (12.5, 5.5) {$\alpha_{34}$};
%	\node at (13.5, 5.5) {$\alpha_{4}$};
	% 4
	\draw[postaction={decorate}] (18,5) -- (19,5);
	\draw[postaction={decorate}] (18,5) -- (18,6);
	\draw[postaction={decorate}] (19,5) -- (19,6);
	\draw[thick, dashed] (15,5) -- (15,6);
	\draw[thick, dashed] (16,5) -- (16,6);
	\draw[thick, dashed] (17,5) -- (17,6);
	\draw[thick, dashed] (15,5) -- (18,5);
	\draw[thick, dashed] (15,6) -- (18,6);
	\draw[postaction={decorate}] (19,6) -- (18,6);
%	\node at (15.5, 5.5) {$\alpha_{14}$};
%	\node at (16.5, 5.5) {$\alpha_{24}$};
%	\node at (17.5, 5.5) {$\alpha_{34}$};
%	\node at (18.5, 5.5) {$\alpha_{4}$};

	% Mid Top
	% 1
	\draw[postaction={decorate}] (0,8) -- (1,8);
	\draw[postaction={decorate}] (1,8) -- (2,8);
	\draw[postaction={decorate}] (2,8) -- (3,8);
	\draw[postaction={decorate}] (3,8) -- (4,8);
	\draw[postaction={decorate}] (0,8) -- (0,9);
	\draw[postaction={decorate}] (1,8) -- (1,9);
	\draw[postaction={decorate}] (2,8) -- (2,9);
	\draw[postaction={decorate}] (3,8) -- (3,9);
	\draw[postaction={decorate}] (4,8) -- (4,9);
	\draw[thick, dashed] (0,9) -- (1,9);
	\draw[thick, dashed] (1,9) -- (2,9);
	\draw[thick, dashed] (2,9) -- (3,9);
	\draw[postaction={decorate}] (4,9) -- (3,9);
%	\node at (0.5, 8.5) {$\alpha_{13,5}$};
%	\node at (1.5, 8.5) {$\alpha_{23,5}$};
%	\node at (2.5, 8.5) {$\alpha_{3,5}$};
%	\node at (3.5, 8.5) {$\alpha_{5}$};
	% 2
	\draw[postaction={decorate}] (5,8) -- (6,8);
	\draw[postaction={decorate}] (6,8) -- (7,8);
	\draw[postaction={decorate}] (7,8) -- (8,8);
	\draw[postaction={decorate}] (8,8) -- (9,8);
	\draw[postaction={decorate}] (5,8) -- (5,9);
	\draw[postaction={decorate}] (6,8) -- (6,9);
	\draw[postaction={decorate}] (7,8) -- (7,9);
	\draw[postaction={decorate}] (8,8) -- (8,9);
	\draw[postaction={decorate}] (9,8) -- (9,9);
	%					\draw[postaction={decorate}] (5,9) -- (6,9);
	\draw[postaction={decorate}] (5.6,9) -- (5.7,9);
	\draw[thick] (5,9) -- (6,9);
	\draw[thick, dashed] (6,9) -- (7,9);
	\draw[thick, dashed] (7,9) -- (8,9);
	\draw[postaction={decorate}] (9,9) -- (8,9);
	
%	\node at (5.5, 8.5) {$\alpha_{13,5}$};
%	\node at (6.5, 8.5) {$\alpha_{23,5}$};
%	\node at (7.5, 8.5) {$\alpha_{3,5}$};
%	\node at (8.5, 8.5) {$\alpha_{5}$};
	% 3
	\draw[thick, dashed] (10,8) -- (11,8);
	\draw[postaction={decorate}] (10,8) -- (10.8,8);
	\draw[postaction={decorate}] (11,8) -- (12,8);
	\draw[postaction={decorate}] (12,8) -- (13,8);
	\draw[postaction={decorate}] (13,8) -- (14,8);
	\draw[postaction={decorate}] (10,8) -- (10,9);
	\draw[postaction={decorate}] (11,8) -- (11,9);
	\draw[postaction={decorate}] (12,8) -- (12,9);
	\draw[postaction={decorate}] (13,8) -- (13,9);
	\draw[postaction={decorate}] (14,8) -- (14,9);
	\draw[postaction={decorate}] (10,9) -- (11,9);
	%					\draw[postaction={decorate}] (11,9) -- (12,9);
	\draw[postaction={decorate}] (11.5,9) -- (11.6,9);
	\draw[thick] (11,9) -- (12,9);
	\draw[thick, dashed] (12,9) -- (13,9);
	\draw[postaction={decorate}] (14,9) -- (13,9);
%	\node at (10.5, 8.5) {$\alpha_{13,5}$};
%	\node at (11.5, 8.5) {$\alpha_{23,5}$};
%	\node at (12.5, 8.5) {$\alpha_{3,5}$};
	\node at (13.5, 8.5) {{\tiny$\alpha_{55}$}};
	
	% 4
	\draw[postaction={decorate}] (15,8) -- (16,8);
	\draw[postaction={decorate}] (16,8) -- (17,8);
	\draw[postaction={decorate}] (17,8) -- (18,8);
	\draw[postaction={decorate}] (18,8) -- (19,8);
	\draw[postaction={decorate}] (15,8) -- (15,9);
	\draw[postaction={decorate}] (16,8) -- (16,9);
	\draw[postaction={decorate}] (17,8) -- (17,9);
	\draw[postaction={decorate}] (18,8) -- (18,9);
	\draw[postaction={decorate}] (19,8) -- (19,9);
	\draw[postaction={decorate}] (15,9) -- (16,9);
	\draw[postaction={decorate}] (16,9) -- (17,9);
	%					\draw[postaction={decorate}] (17,9) -- (18,9);
	\draw[thick] (17,9) -- (18,9);
	\draw[postaction={decorate}] (17.4,9) -- (17.5,9);
	\draw[postaction={decorate}] (19,9) -- (18,9);
%	\node at (15.5, 8.5) {$\alpha_{13,5}$};
	\node at (16.5, 8.5) {{\tiny$\alpha_{23,55}$}};
%	\node at (17.5, 8.5) {$\alpha_{3,5}$};
%	\node at (18.5, 8.5) {$\alpha_{5}$};

	% Top
	\draw[postaction={decorate}] (0,11) -- (0,12);
	\draw[postaction={decorate}] (0,12) -- (0,13);
	\draw[postaction={decorate}] (0,13) -- (0,14);
	\draw[postaction={decorate}] (0,11) -- (1,11);
	\draw[postaction={decorate}] (0,12) -- (1,12);
	\draw[postaction={decorate}] (0,13) -- (1,13);
	\draw[postaction={decorate}] (0,14) -- (1,14);
	\draw[postaction={decorate}] (1,11) -- (1,12);
	\draw[postaction={decorate}] (1,12) -- (1,13);
	\draw[postaction={decorate}] (1,13) -- (1,14);
	\draw[postaction={decorate}] (1,11) -- (2,11);
	\draw[postaction={decorate}] (1,12) -- (2,12);
	\draw[postaction={decorate}] (2,13) -- (1,13);
	\draw[postaction={decorate}] (2,11) -- (2,12);
	\draw[postaction={decorate}] (2,12) -- (2,13);
	\draw[postaction={decorate}] (2,11) -- (3,11);
	\draw[postaction={decorate}] (3,12) -- (2,12);
	\draw[postaction={decorate}] (3,11) -- (3,12);
%	\node at (0.5, 11.5) {$\alpha_{15}$};
%	\node at (0.5, 12.5) {$\alpha_{15, 3}$};
%	\node at (0.5, 13.5) {$\alpha_{15, 23} $};
%	\node at (1.5, 11.5) {$\alpha_{25}$};
%	\node at (1.5, 12.5) {$\alpha_{25, 3} $};
%	\node at (2.5, 11.5) {$\alpha_{35}$};
	\draw[fill=black] (1, 14) circle[radius=2pt, fill=black] node[above right] {$F$};

		\draw[red,line width=1.5pt,opacity=0.7] (0,0)-- (0,5)--(4,5)--(4,6)--(3,6)--(15,8)--(17,8)--(17,9)--(18,9)--(3,11)--(3,12)--(2,12)--(2,13)--(1,13)--(1,14);	
    	\draw[blue,line width=1.5pt,opacity=0.6] (-0.15,0)-- (-0.15,5.1)--(2,5.1)--(2,6)--(10,8)--(14,8)--(14,9)--(13.1,9)--(2.85,11)--(2.85,11.85)--(1.85,11.85)--(1.85,12.85)--(0.85,12.85)--(0.85,13.85);
\end{scope}
\end{scope}
\end{tikzpicture}

%% file: figures/Gamma2i.tex
\begin{tikzpicture}[thick,decoration={markings,mark=at position 0.9 with {\arrow{>}}},xscale=1.35]    
	% Bottom
	\node at (-0.5,0.5) {$ \Gamma_{D_n}^{2,i}: $};
	\draw[postaction={decorate}] (4,0) -- (5,0);
	\draw[postaction={decorate}] (5,0) -- (6,0);
	\draw[postaction={decorate}] (6,0) -- (6.8,0);
	\draw[postaction={decorate}] (7.2,0) -- (8,0);
	\draw[postaction={decorate}] (8,0) -- (9,0);
	\draw[thick,dashed] (5,1) -- (6,1);
	\draw[thick,dashed] (6,1) -- (6.8,1);
	\draw[thick,dashed] (7.2,1) -- (9,1);
	\draw[postaction={decorate}] (10,1) -- (9,1);
	\draw[postaction={decorate}] (9,0) -- (10,0);
	\draw[postaction={decorate}] (10,0) -- (10,1);

	\draw[thick,dashed] (0,1) -- (2.8,1);
	\draw[thick,dashed] (3.1,1) -- (5,1);
	\draw[thick,dashed] (0,0) -- (2.8,0);
	\draw[thick,dashed] (3.2,0) -- (4,0);
	\draw[thick,dashed] (0,0) -- (0,1);
	\draw[thick,dashed] (1,0) -- (1,1);
	\draw[thick,dashed] (2,0) -- (2,1);
	\draw[postaction={decorate}] (4,0) -- (4,1);
	\draw[postaction={decorate}] (5,0) -- (5,1);
	\draw[postaction={decorate}] (6,0) -- (6,1);
%	\draw[postaction={decorate}] (7,0) -- (7,1);
	\draw[postaction={decorate}] (8,0) -- (8,1);
	\draw[postaction={decorate}] (9,0) -- (9,1);
	\node at (0.5, 0.5) {\scalebox{0.85}{$\alpha_{1n-1}$}};
	\node at (1.5, 0.5) {\scalebox{0.8}{$\alpha_{2n-1}$}};
	\node at (3, 0.5) {$\cdots$};
	\node at (4.5, 0.5) {\scalebox{0.85}{$\alpha_{in-1}$}};
	\node at (5.5, 0.5) {\scalebox{0.85}{$\alpha_{i+\!1n-1}$}};
	\node at (7, 0.5) {$\cdots$};
	\node at (8.5, 0.5) {\scalebox{0.85}{$\alpha_{n\!-\!2n\!-\!1}$}};
	\node at (9.5, 0.5){\scalebox{0.85}{$\alpha_{n\!-\!1n\!-\!1}$}};
	
\end{tikzpicture}

%% file: figures/Gamma3i.tex
\begin{tikzpicture}[thick,decoration={markings,mark=at position 0.9 with {\arrow{>}}},xscale=1.35]    
	% Bottom
	\node at (-0.5,0.5) {$ \Gamma_{D_n}^{3,i}: $};
	\draw[postaction={decorate}] (4,0) -- (5,0);
	\draw[postaction={decorate}] (5,0) -- (6,0);
	\draw[postaction={decorate}] (6,0) -- (6.8,0);
	\draw[postaction={decorate}] (7.2,0) -- (8,0);
	\draw[postaction={decorate}] (8,0) -- (9,0);
	\draw[postaction={decorate}] (0,1) -- (1,1);
	\draw[postaction={decorate}] (1,1) -- (2,1);
	\draw[postaction={decorate}] (2,1) -- (2.8,1);
	\draw[postaction={decorate}] (3.2,1) -- (4,1);
	\draw[postaction={decorate}] (4,1) -- (5,1);
	\draw[thick,dashed] (5,1) -- (6,1);
	\draw[thick,dashed] (6,1) -- (6.8,1);
	\draw[thick,dashed] (7.2,1) -- (9,1);
	\draw[postaction={decorate}] (10,1) -- (9,1);
	\draw[postaction={decorate}] (9,0) -- (10,0);
	\draw[postaction={decorate}] (10,0) -- (10,1);

	\draw[postaction={decorate}] (0,0) -- (1,0);
	\draw[postaction={decorate}] (1,0) -- (2,0);
	\draw[postaction={decorate}] (2,0) -- (2.8,0);
	\draw[postaction={decorate}] (3.2,0) -- (4,0);
%	\draw[thick,dashed] (0,1) -- (2.8,1);
%	\draw[thick,dashed] (3.1,1) -- (5,1);
%	\draw[thick,dashed] (0,0) -- (2.8,0);
%	\draw[thick,dashed] (3.2,0) -- (4,0);
	\draw[postaction={decorate}] (0,0) -- (0,1);
	\draw[postaction={decorate}] (1,0) -- (1,1);
	\draw[postaction={decorate}] (2,0) -- (2,1);
	\draw[postaction={decorate}] (4,0) -- (4,1);
	\draw[postaction={decorate}] (5,0) -- (5,1);
	\draw[postaction={decorate}] (6,0) -- (6,1);
%	\draw[postaction={decorate}] (7,0) -- (7,1);
	\draw[postaction={decorate}] (8,0) -- (8,1);
	\draw[postaction={decorate}] (9,0) -- (9,1);
	\node at (0.5, 0.5) {\scalebox{0.85}{$\alpha_{1n\!-\!2,nn}$}};
	\node at (1.5, 0.5) {\scalebox{0.85}{$\alpha_{2n\!-\!2,nn}$}};
	\node at (3, 0.5) {$\cdots$};
	\node at (4.5, 0.5)  {\scalebox{0.85}{$\alpha_{in\!-\!2,nn}$}};
	\node at (5.5, 0.5) {\scalebox{0.85}{$\alpha_{i+1n\!-\!2,nn}$}};
	\node at (7, 0.5) {$\cdots$};
	\node at (8.5, 0.5){\scalebox{0.85}{$\alpha_{n\!-\!2n\!-\!2,nn}$}};
	\node at (9.5, 0.5){\scalebox{0.85}{$\alpha_{nn}$}};
	
\end{tikzpicture}

%% file: figures/D5.tex
            \begin{tikzpicture}[scale=0.8]
      	
% \draw[decorate,decoration={markings,mark=at position 0.7 with {\arrow{>}}}]          	
% \draw[red,decorate,decoration={markings,mark=at position 0.7 with {\arrow{>}}}]  (0,0)--(1,1);          	
					\begin{scope}[shift={(0,0)},thick,decoration={markings,mark=at position 0.7 with {\arrow{>}}}]    
		
		       	\draw [decorate,decoration={brace,amplitude=10pt},gray]
		       	(-0.5,11) -- (-0.5,14) node [black,midway,xshift=-0.8cm]
		       	{\footnotesize 				\color{gray}	 $\Gamma^4_{D_n}$};		
		       	
		       	\draw [decorate,decoration={brace,amplitude=5pt},gray]
		       	(-0.5,7.9) -- (-0.5,9.1) node [black,midway,xshift=-0.8cm]
		       	{\footnotesize 				\color{gray}	 $\Gamma^3_{D_n}$};	
		       	
		       	\draw [decorate,decoration={brace,amplitude=5pt},gray]
		       	(-0.5,4.9) -- (-0.5,6.1) node [black,midway,xshift=-0.8cm]
		       	{\footnotesize 				\color{gray}	 $\Gamma^2_{D_n}$};		
		       	
		       	\draw [decorate,decoration={brace,amplitude=10pt},gray]
		       	(-0.5,0) -- (-0.5,3) node [black,midway,xshift=-0.8cm]
		       	{\footnotesize 				\color{gray}	$\Gamma^1_{D_n}$};
		
					  % Between 1, 2
					  \draw[postaction={decorate}] (0,3) -- (0,5);
					  \draw[postaction={decorate}] (1,3) -- (6,5);
					  \draw[postaction={decorate}] (2,3) -- (12,5);
					  \draw[postaction={decorate}] (3,3) -- (18,5);
					
					  % Between 2, 3
					  \draw[postaction={decorate}] (0,6) -- (0,8);
					  \draw[postaction={decorate}] (1,6) -- (5,8);
					  \draw[postaction={decorate}] (2,6) -- (10,8);
					  \draw[postaction={decorate}] (3,6) -- (15,8);
					  \draw[postaction={decorate}] (6,6) -- (6,8);
					  \draw[postaction={decorate}] (7,6) -- (11,8);
					  \draw[postaction={decorate}] (8,6) -- (16,8);
					  \draw[postaction={decorate}] (12,6) -- (12,8);
					  \draw[postaction={decorate}] (13,6) -- (17,8);
					  \draw[postaction={decorate}] (18,6) -- (18,8);
					
					  % Between
					  \draw[postaction={decorate}] (0,9) -- (0,11);
					  \draw[postaction={decorate}] (1,9) -- (1,11);
					  \draw[postaction={decorate}] (2,9) -- (2,11);
					  \draw[postaction={decorate}] (3,9) -- (3,11);
					  \draw[postaction={decorate}] (6,9) -- (1,11);
					  \draw[postaction={decorate}] (7,9) -- (2,11);
					  \draw[postaction={decorate}] (8,9) -- (3,11);
					  \draw[postaction={decorate}] (12,9) -- (2,11);
					  \draw[postaction={decorate}] (13,9) -- (3,11);
					  \draw[postaction={decorate}] (18,9) -- (3,11);
					\end{scope}
					
					\begin{scope}[shift={(0,0)},thick,decoration={markings,mark=at position 0.9 with {\arrow{>}}}]    
					  % Bottom
					  
					  \draw[postaction={decorate}] (0,0) -- (0,1);
					  \draw[postaction={decorate}] (0,1) -- (0,2);
					  \draw[postaction={decorate}] (0,2) -- (0,3);
					  %        \end{scope}
					\draw[postaction={decorate}] (0,0) -- (0,1);
					\draw[postaction={decorate}] (0,1) -- (0,2);
					\draw[postaction={decorate}] (0,2) -- (0,3);
					\draw[postaction={decorate}] (0,0) -- (1,0);
					\draw[postaction={decorate}] (0,1) -- (1,1);
					\draw[postaction={decorate}] (0,2) -- (1,2);
					\draw[thick, dashed] (0,3) -- (1,3);
					\draw[postaction={decorate}] (1,0) -- (1,1);
					\draw[postaction={decorate}] (1,1) -- (1,2);
					\draw[postaction={decorate}] (1,2) -- (1,3);
					\draw[postaction={decorate}] (1,1) -- (2,1);
					\draw[postaction={decorate}] (1,2) -- (2,2);
					\draw[thick, dashed] (1,3) -- (2,3);
					\draw[postaction={decorate}] (2,1) -- (2,2);
					\draw[postaction={decorate}] (2,2) -- (2,3);
					\draw[postaction={decorate}] (2,2) -- (3,2);
					\draw[thick, dashed] (2,3) -- (3,3);
					\draw[postaction={decorate}] (3,2) -- (3,3);
					\draw[fill=black] (0, 0) circle[radius=2pt, fill=black] node[below left] {$I$};
                    \node at (0.5, 0.5) {\tiny$\alpha_{11}$};
					\node at (0.5, 1.5) {\tiny$\alpha_{12}$};
					\node at (0.5, 2.5) {\tiny$\alpha_{13}$};
					\node at (1.5, 1.5) {\tiny$\alpha_{22}$};
					\node at (1.5, 2.5) {\tiny$\alpha_{23}$};
					\node at (2.5, 2.5) {\tiny$\alpha_{33}$};

					% Mid Bot
					% 1
					\draw[postaction={decorate}] (0,5) -- (1,5);
					\draw[postaction={decorate}] (1,5) -- (2,5);
					\draw[postaction={decorate}] (2,5) -- (3,5);
					\draw[postaction={decorate}] (3,5) -- (4,5);
					\draw[postaction={decorate}] (0,5) -- (0,6);
					\draw[postaction={decorate}] (1,5) -- (1,6);
					\draw[postaction={decorate}] (2,5) -- (2,6);
					\draw[postaction={decorate}] (3,5) -- (3,6);
					\draw[postaction={decorate}] (4,5) -- (4,6);
					\draw[thick, dashed] (0,6) -- (3,6);

					\draw[postaction={decorate}] (3.6,6) -- (3.5,6);
					\draw[thick] (4,6) -- (3,6);

					\node at (0.5, 5.5) {\tiny$\alpha_{14}$};
					\node at (1.5, 5.5) {\tiny$\alpha_{24}$};
					\node at (2.5, 5.5) {\tiny$\alpha_{34}$};
					\node at (3.5, 5.5) {\tiny$\alpha_{44}$};
					% 2
					\draw[postaction={decorate}] (6,5) -- (7,5);
					\draw[postaction={decorate}] (7,5) -- (8,5);
					\draw[postaction={decorate}] (8,5) -- (9,5);
					\draw[postaction={decorate}] (6,5) -- (6,6);
					\draw[postaction={decorate}] (7,5) -- (7,6);
					\draw[postaction={decorate}] (8,5) -- (8,6);
					\draw[postaction={decorate}] (9,5) -- (9,6);
					\draw[thick, dashed] (5,5) -- (5,6);
					\draw[thick, dashed] (5,5) -- (6,5);
					\draw[thick, dashed] (5,6) -- (8,6);
%					\draw[postaction={decorate}] (9,6) -- (8,6);
					\draw[thick] (9,6) -- (8,6);
					\draw[postaction={decorate}] (8.6,6) -- (8.5,6);
					\node at (5.5, 5.5) {\tiny$\alpha_{14}$};
					\node at (6.5, 5.5) {\tiny$\alpha_{24}$};
					\node at (7.5, 5.5) {\tiny$\alpha_{34}$};
					\node at (8.5, 5.5) {\tiny$\alpha_{44}$};
					% 3
					\draw[postaction={decorate}] (12,5) -- (13,5);
					\draw[postaction={decorate}] (13,5) -- (14,5);
					\draw[postaction={decorate}] (12,5) -- (12,6);
					\draw[postaction={decorate}] (13,5) -- (13,6);
					\draw[postaction={decorate}] (14,5) -- (14,6);
					\draw[thick, dashed] (10,5) -- (10,6);
					\draw[thick, dashed] (11,5) -- (11,6);
					\draw[thick, dashed] (10,5) -- (12,5);
					\draw[thick, dashed] (10,6) -- (13,6);
%					\draw[postaction={decorate}] (14,6) -- (13,6);
					\draw[thick] (14,6) -- (13,6);
					\draw[postaction={decorate}] (13.5,6) -- (13.4,6);
					\node at (10.5, 5.5) {\tiny$\alpha_{14}$};
					\node at (11.5, 5.5) {\tiny$\alpha_{24}$};
					\node at (12.5, 5.5) {\tiny$\alpha_{34}$};
					\node at (13.5, 5.5) {\tiny$\alpha_{44}$};
					% 4
					\draw[postaction={decorate}] (18,5) -- (19,5);
					\draw[postaction={decorate}] (18,5) -- (18,6);
					\draw[postaction={decorate}] (19,5) -- (19,6);
					\draw[thick, dashed] (15,5) -- (15,6);
					\draw[thick, dashed] (16,5) -- (16,6);
					\draw[thick, dashed] (17,5) -- (17,6);
					\draw[thick, dashed] (15,5) -- (18,5);
					\draw[thick, dashed] (15,6) -- (18,6);
					\draw[postaction={decorate}] (19,6) -- (18,6);
					\node at (15.5, 5.5) {\tiny$\alpha_{14}$};
					\node at (16.5, 5.5) {\tiny$\alpha_{24}$};
					\node at (17.5, 5.5) {\tiny$\alpha_{34}$};
					\node at (18.5, 5.5) {\tiny$\alpha_{44}$};

					% Mid Top
					% 1
					\draw[postaction={decorate}] (0,8) -- (1,8);
					\draw[postaction={decorate}] (1,8) -- (2,8);
					\draw[postaction={decorate}] (2,8) -- (3,8);
					\draw[postaction={decorate}] (3,8) -- (4,8);
					\draw[postaction={decorate}] (0,8) -- (0,9);
					\draw[postaction={decorate}] (1,8) -- (1,9);
					\draw[postaction={decorate}] (2,8) -- (2,9);
					\draw[postaction={decorate}] (3,8) -- (3,9);
					\draw[postaction={decorate}] (4,8) -- (4,9);
					\draw[thick, dashed] (0,9) -- (1,9);
					\draw[thick, dashed] (1,9) -- (2,9);
					\draw[thick, dashed] (2,9) -- (3,9);
					\draw[postaction={decorate}] (4,9) -- (3,9);
					\node at (0.5, 8.5) {\tiny$\alpha_{13,55}$};
					\node at (1.5, 8.5) {\tiny$\alpha_{23,55}$};
					\node at (2.5, 8.5) {\tiny$\alpha_{33,55}$};
					\node at (3.5, 8.5) {\tiny$\alpha_{55}$};
					% 2
					\draw[postaction={decorate}] (5,8) -- (6,8);
				  \draw[postaction={decorate}] (6,8) -- (7,8);
					\draw[postaction={decorate}] (7,8) -- (8,8);
					\draw[postaction={decorate}] (8,8) -- (9,8);
					\draw[postaction={decorate}] (5,8) -- (5,9);
					\draw[postaction={decorate}] (6,8) -- (6,9);
					\draw[postaction={decorate}] (7,8) -- (7,9);
					\draw[postaction={decorate}] (8,8) -- (8,9);
					\draw[postaction={decorate}] (9,8) -- (9,9);
%					\draw[postaction={decorate}] (5,9) -- (6,9);
					\draw[postaction={decorate}] (5.6,9) -- (5.7,9);
					\draw[thick] (5,9) -- (6,9);
					\draw[thick, dashed] (6,9) -- (7,9);
					\draw[thick, dashed] (7,9) -- (8,9);
					\draw[postaction={decorate}] (9,9) -- (8,9);

					\node at (5.5, 8.5) {\tiny$\alpha_{13,55}$};
					\node at (6.5, 8.5) {\tiny$\alpha_{23,55}$};
					\node at (7.5, 8.5) {\tiny$\alpha_{33,55}$};
					\node at (8.5, 8.5) {\tiny$\alpha_{55}$};
					% 3
					\draw[thick, dashed] (10,8) -- (11,8);
					\draw[postaction={decorate}] (10,8) -- (10.8,8);
					\draw[postaction={decorate}] (11,8) -- (12,8);
					\draw[postaction={decorate}] (12,8) -- (13,8);
					\draw[postaction={decorate}] (13,8) -- (14,8);
					\draw[postaction={decorate}] (10,8) -- (10,9);
					\draw[postaction={decorate}] (11,8) -- (11,9);
					\draw[postaction={decorate}] (12,8) -- (12,9);
					\draw[postaction={decorate}] (13,8) -- (13,9);
					\draw[postaction={decorate}] (14,8) -- (14,9);
					\draw[postaction={decorate}] (10,9) -- (11,9);
%					\draw[postaction={decorate}] (11,9) -- (12,9);
					\draw[postaction={decorate}] (11.5,9) -- (11.6,9);
					\draw[thick] (11,9) -- (12,9);
					\draw[thick, dashed] (12,9) -- (13,9);
					\draw[postaction={decorate}] (14,9) -- (13,9);
					\node at (10.5, 8.5) {\tiny$\alpha_{13,55}$};
					\node at (11.5, 8.5) {\tiny$\alpha_{23,55}$};
					\node at (12.5, 8.5) {\tiny$\alpha_{33,55}$};
					\node at (13.5, 8.5) {\tiny$\alpha_{55}$};
					% 4
					\draw[postaction={decorate}] (15,8) -- (16,8);
					\draw[postaction={decorate}] (16,8) -- (17,8);
					\draw[postaction={decorate}] (17,8) -- (18,8);
					\draw[postaction={decorate}] (18,8) -- (19,8);
					\draw[postaction={decorate}] (15,8) -- (15,9);
					\draw[postaction={decorate}] (16,8) -- (16,9);
					\draw[postaction={decorate}] (17,8) -- (17,9);
					\draw[postaction={decorate}] (18,8) -- (18,9);
					\draw[postaction={decorate}] (19,8) -- (19,9);
					\draw[postaction={decorate}] (15,9) -- (16,9);
					\draw[postaction={decorate}] (16,9) -- (17,9);
%					\draw[postaction={decorate}] (17,9) -- (18,9);
					\draw[thick] (17,9) -- (18,9);
					\draw[postaction={decorate}] (17.4,9) -- (17.5,9);
					\draw[postaction={decorate}] (19,9) -- (18,9);
					\node at (15.5, 8.5) {\tiny$\alpha_{13,55}$};
					\node at (16.5, 8.5) {\tiny$\alpha_{23,55}$};
					\node at (17.5, 8.5) {\tiny$\alpha_{33,55}$};
					\node at (18.5, 8.5) {\tiny$\alpha_{55}$};

					% Top
					\draw[postaction={decorate}] (0,11) -- (0,12);
					\draw[postaction={decorate}] (0,12) -- (0,13);
					\draw[postaction={decorate}] (0,13) -- (0,14);
					\draw[postaction={decorate}] (0,11) -- (1,11);
					\draw[postaction={decorate}] (0,12) -- (1,12);
					\draw[postaction={decorate}] (0,13) -- (1,13);
					\draw[postaction={decorate}] (0,14) -- (1,14);
					\draw[postaction={decorate}] (1,11) -- (1,12);
					\draw[postaction={decorate}] (1,12) -- (1,13);
					\draw[postaction={decorate}] (1,13) -- (1,14);
					\draw[postaction={decorate}] (1,11) -- (2,11);
					\draw[postaction={decorate}] (1,12) -- (2,12);
					\draw[postaction={decorate}] (2,13) -- (1,13);
					\draw[postaction={decorate}] (2,11) -- (2,12);
					\draw[postaction={decorate}] (2,12) -- (2,13);
					\draw[postaction={decorate}] (2,11) -- (3,11);
					\draw[postaction={decorate}] (3,12) -- (2,12);
					\draw[postaction={decorate}] (3,11) -- (3,12);
					\node at (0.5, 11.5) {\tiny$\alpha_{15}$};
					\node at (0.5, 12.5) {\tiny$\alpha_{15, 33}$};
					\node at (0.5, 13.5) {\tiny$\alpha_{15, 23} $};
					\node at (1.5, 11.5) {\tiny$\alpha_{25}$};
					\node at (1.5, 12.5) {\tiny$\alpha_{25, 33} $};
					\node at (2.5, 11.5) {\tiny$\alpha_{35}$};
					\draw[fill=black] (1, 14) circle[radius=2pt, fill=black] node[above right] {$F$};
					\end{scope}
					\end{tikzpicture}
    

%% file: figures/diagonal_sum.tex
		\begin{tikzpicture}
                [thick,decoration={markings,mark=at position 0.9 with {\arrow{>}}}]
					\begin{scope}[shift={(0,0)}, xscale=1]
				\draw[postaction={decorate}] (0,0) -- (0,1);
				\draw[postaction={decorate}] (0,1) -- (0,2);
				\draw[postaction={decorate}] (0,2) -- (0,3);
				\draw[postaction={decorate}] (0,3) -- (0,4);
				\draw[postaction={decorate}] (0,4) -- (0,5);
				\draw[postaction={decorate}] (0,5) -- (0,6);
				\draw[postaction={decorate}] (0,6) -- (0,7);
				\draw[postaction={decorate}] (0,0) -- (1,0);
				\draw[postaction={decorate}] (0,1) -- (1,1);
				\draw[postaction={decorate}] (0,2) -- (1,2);
				\draw[postaction={decorate}] (0,3) -- (1,3);
				\draw[postaction={decorate}] (0,4) -- (1,4);
				\draw[postaction={decorate}] (0,5) -- (1,5);
				\draw[postaction={decorate}] (0,6) -- (1,6);
				\draw[postaction={decorate}] (0,7) -- (1,7);
				\draw[postaction={decorate}] (1,0) -- (1,1);
				\draw[postaction={decorate}] (1,1) -- (1,2);
				\draw[postaction={decorate}] (1,2) -- (1,3);
				\draw[postaction={decorate}] (1,3) -- (1,4);
				\draw[postaction={decorate}] (1,4) -- (1,5);
				\draw[postaction={decorate}] (1,5) -- (1,6);
				\draw[postaction={decorate}] (1,6) -- (1,7);
				\draw[postaction={decorate}] (1,1) -- (2,1);
				\draw[postaction={decorate}] (1,2) -- (2,2);
				\draw[postaction={decorate}] (1,3) -- (2,3);
				\draw[postaction={decorate}] (1,4) -- (2,4);
				\draw[postaction={decorate}] (1,5) -- (2,5);
				\draw[postaction={decorate}] (2,6) -- (1,6);
				\draw[postaction={decorate}] (2,1) -- (2,2);
				\draw[postaction={decorate}] (2,2) -- (2,3);
				\draw[postaction={decorate}] (2,3) -- (2,4);
				\draw[postaction={decorate}] (2,4) -- (2,5);
				\draw[postaction={decorate}] (2,5) -- (2,6);
				\draw[postaction={decorate}] (2,2) -- (3,2);
				\draw[postaction={decorate}] (2,3) -- (3,3);
				\draw[postaction={decorate}] (2,4) -- (3,4);
				\draw[postaction={decorate}] (3,5) -- (2,5);
				\draw[postaction={decorate}] (3,2) -- (3,3);
				\draw[postaction={decorate}] (3,3) -- (3,4);
				\draw[postaction={decorate}] (3,4) -- (3,5);
				\draw[postaction={decorate}] (3,3) -- (4,3);
				\draw[postaction={decorate}] (4,4) -- (3,4);
				\draw[postaction={decorate}] (4,3) -- (4,4);
				\node at (0.5, 0.5) {$\alpha_{11}$};
				\node at (0.5, 1.5) {$\alpha_{12}$};
				\node at (0.5, 2.5) {$\alpha_{13}$};
				\node at (0.5, 3.5) {$\alpha_{14}$};
				\node at (1.5, 1.5) {$\alpha_{22}$};
				\node at (1.5, 2.5) {$\alpha_{23}$};
				\node at (1.5, 3.5) {$\alpha_{24}$};
				\node at (2.5, 2.5) {$\alpha_{33}$};
				\node at (2.5, 3.5) {$\alpha_{34}$};
				\node at (3.5, 3.5) {$\alpha_{44}$};
				\node at (0.5, 4.5) {$\alpha_{14,4}$};
				\node at (0.5, 5.5) {$\alpha_{14,34}$};
				\node at (0.5, 6.5) {$\alpha_{14,24}$};
				\node at (1.5, 4.5) {$\alpha_{24,4}$};
				\node at (1.5, 5.5) {$\alpha_{24,34}$};
				\node at (2.5, 4.5) {$\alpha_{34,4}$};
					\draw[fill=Black] (0,0) circle[radius=2pt, fill=Black] node[below left] {$I$};
			\draw[fill=Black] (1,7) circle[radius=2pt, fill=Black] node[above right] {$F$};
			\draw[gray] (0,-1)--(4.1,3.1);
			\draw[gray] (5.5,2.5)--(1,7);
			\draw[gray] (5,2)--(0,7);
			\node at (2.4,1){ \rotatebox{45}{\color{gray} $x-y=0$}};
			\node at (4.7,2){ \rotatebox{-45}{\color{gray} $x-y=2n$}};
			\node at (5.7,2){ \rotatebox{-45}{\color{gray} $x-y=2n+1$}};
			\end{scope}
	\begin{scope}[shift={(8,0)}, xscale=1]
	\draw[postaction={decorate}] (0,0) -- (0,1);
	\draw[postaction={decorate}] (0,1) -- (0,2);
	\draw[postaction={decorate}] (0,2) -- (0,3);
	\draw[postaction={decorate}] (0,3) -- (0,4);
	\draw[postaction={decorate}] (0,4) -- (0,5);
	\draw[postaction={decorate}] (0,5) -- (0,6);
	\draw[postaction={decorate}] (0,6) -- (0,7);
	\draw[postaction={decorate}] (0,0) -- (1,0);
	\draw[postaction={decorate}] (0,1) -- (1,1);
	\draw[postaction={decorate}] (0,2) -- (1,2);
	\draw[postaction={decorate}] (0,3) -- (1,3);
	\draw[postaction={decorate}] (0,4) -- (1,4);
	\draw[postaction={decorate}] (0,5) -- (1,5);
	\draw[postaction={decorate}] (0,6) -- (1,6);
	\draw[postaction={decorate}] (0,7) -- (1,7);
	\draw[postaction={decorate}] (1,0) -- (1,1);
	\draw[postaction={decorate}] (1,1) -- (1,2);
	\draw[postaction={decorate}] (1,2) -- (1,3);
	\draw[postaction={decorate}] (1,3) -- (1,4);
	\draw[postaction={decorate}] (1,4) -- (1,5);
	\draw[postaction={decorate}] (1,5) -- (1,6);
	\draw[postaction={decorate}] (1,6) -- (1,7);
	\draw[postaction={decorate}] (1,1) -- (2,1);
	\draw[postaction={decorate}] (1,2) -- (2,2);
	\draw[postaction={decorate}] (1,3) -- (2,3);
	\draw[postaction={decorate}] (1,4) -- (2,4);
	\draw[postaction={decorate}] (1,5) -- (2,5);
	\draw[postaction={decorate}] (2,6) -- (1,6);
	\draw[postaction={decorate}] (2,1) -- (2,2);
	\draw[postaction={decorate}] (2,2) -- (2,3);
	\draw[postaction={decorate}] (2,3) -- (2,4);
	\draw[postaction={decorate}] (2,4) -- (2,5);
	\draw[postaction={decorate}] (2,5) -- (2,6);
	\draw[postaction={decorate}] (2,2) -- (3,2);
	\draw[postaction={decorate}] (2,3) -- (3,3);
	\draw[postaction={decorate}] (2,4) -- (3,4);
	\draw[postaction={decorate}] (3,5) -- (2,5);
	\draw[postaction={decorate}] (3,2) -- (3,3);
	\draw[postaction={decorate}] (3,3) -- (3,4);
	\draw[postaction={decorate}] (3,4) -- (3,5);
	\draw[postaction={decorate}] (3,3) -- (4,3);
	\draw[postaction={decorate}] (4,4) -- (3,4);
	\draw[postaction={decorate}] (4,3) -- (4,4);
	\node at (0.5, 0.5) {$\alpha_{11}$};
	\node at (0.5, 1.5) {$\alpha_{12}$};
	\node at (0.5, 2.5) {$\alpha_{13}$};
	\node at (0.5, 3.5) {$\alpha_{14}$};
	\node at (1.5, 1.5) {$\alpha_{22}$};
	\node at (1.5, 2.5) {$\alpha_{23}$};
	\node at (1.5, 3.5) {$\alpha_{24}$};
	\node at (2.5, 2.5) {$\alpha_{33}$};
	\node at (2.5, 3.5) {$\alpha_{34}$};
	\node at (3.5, 3.5) {$\alpha_{44}$};
	\node at (0.5, 4.5) {$\alpha_{14,4}$};
	\node at (0.5, 5.5) {$\alpha_{14,34}$};
	\node at (0.5, 6.5) {$\alpha_{14,24}$};
	\node at (1.5, 4.5) {$\alpha_{24,4}$};
	\node at (1.5, 5.5) {$\alpha_{24,34}$};
	\node at (2.5, 4.5) {$\alpha_{34,4}$};
	\draw[fill=Black] (0,0) circle[radius=2pt, fill=Black] node[below left] {$I$};
	\draw[fill=Black] (1,7) circle[radius=2pt, fill=Black] node[above right] {$F$};
		\draw[fill=gray] (0,7) circle[radius=2pt,gray] node[above right] {\color{gray} $F^{D_0}$};
	\draw[fill=gray] (1,6) circle[radius=2pt,gray] node[above right] {\color{gray} $F^{D_1}$};
	\draw[fill=gray] (2,5) circle[radius=2pt,gray] node[above right] {\color{gray} $F^{D_2}$};
	\draw[fill=gray] (3,4) circle[radius=2pt,gray] node[above right] {\color{gray} $F^{D_3}$};
	\draw[fill=gray] (4,3) circle[radius=2pt,gray] node[above right] {\color{gray} $F^{D_4}$};
\end{scope}
		\end{tikzpicture}
	

%% file: figures/D4_1.tex
\begin{tikzpicture}[thick,decoration={markings,mark=at position 0.9 with {\arrow{>}}}]
    \node at (-3, 1) {$\Gamma_{D_4}^1$:};
    \draw[postaction={decorate}] (0,0) -- (0,1);
    \draw[postaction={decorate}] (0,1) -- (0,2);
    \draw[postaction={decorate}] (0,0) -- (1,0);
    \draw[postaction={decorate}] (0,1) -- (1,1);
    \draw[dashed] (0,2) -- (1,2);
    \draw[postaction={decorate}] (1,0) -- (1,1);
    \draw[postaction={decorate}] (1,1) -- (1,2);
    \draw[postaction={decorate}] (1,1) -- (2,1);
    \draw[dashed] (1,2) -- (2,2);
    \draw[postaction={decorate}] (2,1) -- (2,2);
    \node at (0.5, 0.5) {$\alpha_{11}$};
    \node at (0.5, 1.5) {$\alpha_{12}$};
    \node at (1.5, 1.5) {$\alpha_{22}$};
    \draw[fill=black] (0, 0) circle[radius=2pt, fill=black] node[below left] {$I = (0, 1)_{(1, 0)}$};
    \draw[fill=black] (1, 0) circle[radius=2pt, fill=black] node[below right] {$(1, 1)_{(1, 0)}$};
    \draw[fill=black] (0, 1) circle[radius=2pt, fill=black] node[left] {$(0, 2)_{(1, 0)}$};
    \draw[fill=black] (0, 2) circle[radius=2pt, fill=black] node[above left] {$(0, 3)_{(1, 0)}$};
    \draw[fill=black] (1, 1) circle[radius=2pt, fill=black] node[below right] {$(1, 2)_{(1, 0)}$};
    \draw[fill=black] (1, 2) circle[radius=2pt, fill=black] node[above] {$(1, 3)_{(1, 0)}$};
    \draw[fill=black] (2, 1) circle[radius=2pt, fill=black] node[right] {$(2, 2)_{(1, 0)}$};
    \draw[fill=black] (2, 2) circle[radius=2pt, fill=black] node[above right] {$(2, 3)_{(1, 0)}$};
\end{tikzpicture}

%% file: figures/D4_2.tex
\begin{tikzpicture}[thick,decoration={markings,mark=at position 0.9 with {\arrow{>}}}]    
    % Mid Bot
    % 2
    \node at (-2, 0.5) {$\Gamma_{D_4}^{2,2}$:};
    \draw[postaction={decorate}] (1,0) -- (2,0);
    \draw[postaction={decorate}] (2,0) -- (3,0);
    \draw[postaction={decorate}] (1,0) -- (1,1);
    \draw[postaction={decorate}] (2,0) -- (2,1);
    \draw[postaction={decorate}] (3,0) -- (3,1);
    \draw[thick, dashed] (0,0) -- (0,1);
    \draw[thick, dashed] (0,0) -- (1,0);
    \draw[thick, dashed] (0,1) -- (2,1);
    \draw[postaction={decorate}] (3,1) -- (2,1);
    \node at (0.5, 0.5) {$\alpha_{13}$};
    \node at (1.5, 0.5) {$\alpha_{23}$};
    \node at (2.5, 0.5) {$\alpha_{33}$};
    \draw[fill=black] (1, 0) circle[radius=2pt, fill=black] node[below left] {$(1, 0)_{(2, 2)}$};
    \draw[fill=black] (2, 0) circle[radius=2pt, fill=black] node[below] {$(2, 0)_{(2, 2)}$};
    \draw[fill=black] (3, 0) circle[radius=2pt, fill=black] node[below right] {$(3, 0)_{(2, 2)}$};
    \draw[fill=black] (1, 1) circle[radius=2pt, fill=black] node[above left] {$(1, 1)_{(2, 2)}$};
    \draw[fill=black] (2, 1) circle[radius=2pt, fill=black] node[above] {$(2, 1)_{(2, 2)}$};
    \draw[fill=black] (3, 1) circle[radius=2pt, fill=black] node[above right] {$(3, 1)_{(2, 2)}$};
\end{tikzpicture}

%% file: figures/D4_3.tex
\begin{tikzpicture}[thick,decoration={markings,mark=at position 0.9 with {\arrow{>}}}]
    % Mid Top
    % 2
    \node at (-2, 0.5) {$\Gamma_{D_4}^{3,2}$:};
    \draw[postaction={decorate}] (0,0) -- (1,0);
    \draw[postaction={decorate}] (1,0) -- (3,0);
    \draw[postaction={decorate}] (3,0) -- (4,0);
    \draw[postaction={decorate}] (0,0) -- (0,1);
    \draw[postaction={decorate}] (1,0) -- (1,1);
    \draw[postaction={decorate}] (3,0) -- (3,1);
    \draw[postaction={decorate}] (4,0) -- (4,1);
    \draw[] (0,1) -- (1,1);
    \draw[dashed] (1,1) -- (3,1);
    \draw[postaction={decorate}] (4,1) -- (3,1);
    
    \node at (0.5, 0.5) {$\alpha_{12,44}$};
    \node at (2, 0.5) {$\alpha_{22,44}$};
    \node at (3.5, 0.5) {$\alpha_{44}$};
    \draw[fill=black] (0, 0) circle[radius=2pt, fill=black] node[below left] {$(0, 0)_{(3, 2)}$};
    \draw[fill=black] (1, 0) circle[radius=2pt, fill=black] node[below] {$(1, 0)_{(3, 2)}$};
    \draw[fill=black] (3, 0) circle[radius=2pt, fill=black] node[below] {$(2, 0)_{(3, 2)}$};
    \draw[fill=black] (4, 0) circle[radius=2pt, fill=black] node[below right] {$(3, 0)_{(3, 2)}$};
    \draw[fill=black] (0, 1) circle[radius=2pt, fill=black] node[above left] {$(0, 1)_{(3, 2)}$};
    \draw[fill=black] (1, 1) circle[radius=2pt, fill=black] node[above] {$(1, 1)_{(3, 2)}$};
    \draw[fill=black] (3, 1) circle[radius=2pt, fill=black] node[above] {$(2, 1)_{(3, 2)}$};
    \draw[fill=black] (4, 1) circle[radius=2pt, fill=black] node[above right] {$(3, 1)_{(3, 2)}$};
\end{tikzpicture}

%% file: figures/D4_4.tex
\begin{tikzpicture}[thick,decoration={markings,mark=at position 0.9 with {\arrow{>}}}]    
    % Top
    \node at (-3, 1) {$\Gamma_{D_4}^4$:};
    \draw[postaction={decorate}] (0,0) -- (0,1);
    \draw[postaction={decorate}] (0,1) -- (0,2);
    \draw[postaction={decorate}] (0,0) -- (1,0);
    \draw[postaction={decorate}] (0,1) -- (1,1);
    \draw[postaction={decorate}] (0,2) -- (1,2);
    \draw[postaction={decorate}] (1,0) -- (1,1);
    \draw[postaction={decorate}] (1,1) -- (1,2);
    \draw[postaction={decorate}] (1,0) -- (2,0);
    \draw[postaction={decorate}] (2,1) -- (1,1);
    \draw[postaction={decorate}] (2,0) -- (2,1);
    \node at (0.5, 0.5) {$\alpha_{14}$};
    \node at (0.5, 1.5) {$\alpha_{14, 22}$};
    \node at (1.5, 0.5) {$\alpha_{24}$};
    \draw[fill=black] (0, 0) circle[radius=2pt, fill=black] node[below left] {$(0, 6)_{(4, 0)}$};
    \draw[fill=black] (1, 0) circle[radius=2pt, fill=black] node[below] {$(1, 6)_{(4, 0)}$};
    \draw[fill=black] (2, 0) circle[radius=2pt, fill=black] node[below right] {$(2, 6)_{(4, 0)}$};
    \draw[fill=black] (0, 1) circle[radius=2pt, fill=black] node[left] {$(0, 7)_{(4, 0)}$};
    \draw[fill=black] (1, 1) circle[radius=2pt, fill=black] node[above right] {$(1, 7)_{(4, 0)}$};
    \draw[fill=black] (2, 1) circle[radius=2pt, fill=black] node[right] {$(2, 7)_{(4, 0)}$};
    \draw[fill=black] (0, 2) circle[radius=2pt, fill=black] node[above left] {$(0, 8)_{(4, 0)}$};
    \draw[fill=black] (1, 2) circle[radius=2pt, fill=black] node[above right] {$(1, 8)_{(4, 0)} = F$};
\end{tikzpicture}

%% file: figures/F4.tex
{\small
    \begin{tikzpicture}[yscale=0.9, xscale=1.5]
	\begin{scope}[shift={(0,0)},thick,decoration={markings,mark=at position 0.7 with {\arrow{>}}}]    

        % Between 1->2
        \draw[postaction={decorate}] (0,3) -- (0,5);
        \draw[postaction={decorate}] (1,3) -- (5.5,5);
        \draw[postaction={decorate}] (2,3) -- (11,5);
        \draw[postaction={decorate}] (3,3) -- (12,5);
        % Between 3->4
        \draw[postaction={decorate}] (0,6) -- (0,8);
        \draw[postaction={decorate}] (1,6) -- (4,8);
        \draw[postaction={decorate}] (2,6) -- (8,8);
        \draw[postaction={decorate}] (5.5,6) -- (6,8);
        \draw[postaction={decorate}] (6.5,6) -- (10,8);
        \draw[postaction={decorate}] (11,6) -- (11,8);

        % Between 3->4
        \draw[postaction={decorate}] (0,9) -- (0,11);
        \draw[postaction={decorate}] (1,9) -- (1,11);
        \draw[postaction={decorate}] (2,9) -- (2,11);
        \draw[postaction={decorate}] (3,9) -- (3,11);
        \draw[postaction={decorate}] (6,9) -- (2,11);
        \draw[postaction={decorate}] (7,9) -- (3,11);
        \draw[postaction={decorate}] (11,9) -- (3,11);
    \end{scope}
	\begin{scope}[shift={(0,0)},thick,decoration={markings,mark=at position 0.9 with {\arrow{>}}}]
        % Bottom
        \draw[postaction={decorate}] (0,0) -- (0,1);
        \draw[postaction={decorate}] (0,1) -- (0,2);
        \draw[postaction={decorate}] (0,2) -- (0,3);
        \draw[postaction={decorate}] (0,0) -- (1,0);
        \draw[postaction={decorate}] (0,1) -- (1,1);
        \draw[postaction={decorate}] (0,2) -- (1,2);
        \draw[thick, dashed] (0,3) -- (1,3);
        \draw[postaction={decorate}] (1,0) -- (1,1);
        \draw[postaction={decorate}] (1,1) -- (1,2);
        \draw[postaction={decorate}] (1,2) -- (1,3);
        \draw[postaction={decorate}] (1,1) -- (2,1);
        \draw[postaction={decorate}] (1,2) -- (2,2);
        \draw[thick, dashed] (1,3) -- (2,3);
        \draw[postaction={decorate}] (2,1) -- (2,2);
        \draw[postaction={decorate}] (2,2) -- (2,3);
        \draw[postaction={decorate}] (2,2) -- (3,2);
        \draw[thick, dashed] (2,3) -- (3,3);
        \draw[postaction={decorate}] (3,2) -- (3,3);
        \draw[fill=Black] (0, 0) circle[radius=2pt, fill=Black] node[below left] {$I$};
        \node at (0.5, 0.5) {\tiny$\alpha_{11}$};
        \node at (0.5, 1.5) {\tiny$\alpha_{12}$};
        \node at (0.5, 2.5) {\tiny$\alpha_{13}$};
        \node at (1.5, 1.5) {\tiny$\alpha_{22}$};
        \node at (1.5, 2.5) {\tiny$\alpha_{23}$};
        \node at (2.5, 2.5) {\tiny$\alpha_{33}$};

        % Mid Bot
        % 1
        \draw[postaction={decorate}] (0,5) -- (1,5);
        \draw[postaction={decorate}] (1,5) -- (2,5);
        \draw[postaction={decorate}] (2,5) -- (3,5);
        \draw[postaction={decorate}] (3,5) -- (4,5);
        \draw[postaction={decorate}] (0,5) -- (0,6);
        \draw[postaction={decorate}] (1,5) -- (1,6);
        \draw[postaction={decorate}] (2,5) -- (2,6);
        \draw[postaction={decorate}] (3,5) -- (3,6);
        \draw[postaction={decorate}] (4,5) -- (4,6);
        \draw[thick, dashed] (0,6) -- (2,6);
        \draw[postaction={decorate}] (4,6) -- (3,6);
        \draw[postaction={decorate}] (3,6) -- (2,6);
        \node at (0.5, 5.5) {\tiny$\alpha_{14}$};
        \node at (1.5, 5.5) {\tiny$\alpha_{24}$};
        \node at (2.5, 5.5) {\tiny$\alpha_{34}$};
        \node at (3.5, 5.5) {\tiny$\alpha_{44}$};
        % 2
        \draw[postaction={decorate}] (5.5,5) -- (6.5,5);
        \draw[postaction={decorate}] (6.5,5) -- (7.5,5);
        \draw[postaction={decorate}] (7.5,5) -- (8.5,5);
        \draw[postaction={decorate}] (5.5,5) -- (5.5,6);
        \draw[postaction={decorate}] (6.5,5) -- (6.5,6);
        \draw[postaction={decorate}] (7.5,5) -- (7.5,6);
        \draw[postaction={decorate}] (8.5,5) -- (8.5,6);
        \draw[thick, dashed] (4.5,5) -- (4.5,6);
        \draw[thick, dashed] (4.5,5) -- (5.5,5);
        \draw[thick, dashed] (4.5,6) -- (6.5,6);
        \draw[postaction={decorate}] (8.5,6) -- (7.5,6);
        \draw[postaction={decorate}] (7.5,6) -- (6.5,6);
        \node at (5, 5.5) {\tiny$\alpha_{14}$};
        \node at (6, 5.5) {\tiny$\alpha_{24}$};
        \node at (7, 5.5) {\tiny$\alpha_{34}$};
        \node at (8, 5.5) {\tiny$\alpha_{44}$};
        % 3
        \draw[postaction={decorate}] (11,5) -- (12,5);
        \draw[postaction={decorate}] (12,5) -- (13,5);
        \draw[postaction={decorate}] (11,5) -- (11,6);
        \draw[postaction={decorate}] (12,5) -- (12,6);
        \draw[postaction={decorate}] (13,5) -- (13,6);
        \draw[thick, dashed] (9,5) -- (9,6);
        \draw[thick, dashed] (10,5) -- (10,6);
        \draw[thick, dashed] (9,5) -- (11,5);
        \draw[thick, dashed] (9,6) -- (11,6);
        \draw[postaction={decorate}] (13,6) -- (12,6);
        \draw[postaction={decorate}] (12,6) -- (11,6);
        \node at (9.5, 5.5) {\tiny$\alpha_{14}$};
        \node at (10.5, 5.5) {\tiny$\alpha_{24}$};
        \node at (11.5, 5.5) {\tiny$\alpha_{34}$};
        \node at (12.5, 5.5) {\tiny$\alpha_{44}$};

        % Mid Top
        % 1
        \draw[postaction={decorate}] (0,8) -- (1,8);
        \draw[postaction={decorate}] (1,8) -- (2,8);
        \draw[postaction={decorate}] (2,8) -- (3,8);
        \draw[postaction={decorate}] (0,8) -- (0,9);
        \draw[postaction={decorate}] (1,8) -- (1,9);
        \draw[postaction={decorate}] (2,8) -- (2,9);
        \draw[postaction={decorate}] (3,8) -- (3,9);
        \draw[thick, dashed] (0,9) -- (1,9);
        \draw[thick, dashed] (1,9) -- (2,9);
        \draw[thick, dashed] (2,9) -- (3,9);
        \node at (0.5, 8.5) {\tiny$\alpha_{13,12}$};
        \node at (1.5, 8.5) {\tiny$\alpha_{13,22}$};
        \node at (2.5, 8.5) {\tiny$\alpha_{23,22}$};
        % 2
        \draw[postaction={decorate}] (4,8) -- (5,8);
        \draw[postaction={decorate}] (5,8) -- (6,8);
        \draw[postaction={decorate}] (6,8) -- (7,8);
        \draw[postaction={decorate}] (4,8) -- (4,9);
        \draw[postaction={decorate}] (5,8) -- (5,9);
        \draw[postaction={decorate}] (6,8) -- (6,9);
        \draw[postaction={decorate}] (7,8) -- (7,9);
        \draw[postaction={decorate}] (4,9) -- (5,9);
        \draw[postaction={decorate}] (5,9) -- (6,9);
        \draw[thick, dashed] (6,9) -- (7,9);
        \node at (4.5, 8.5) {\tiny$\alpha_{13,12}$};
        \node at (5.5, 8.5) {\tiny$\alpha_{13,22}$};
        \node at (6.5, 8.5) {\tiny$\alpha_{23,22}$};
        % 3
        \draw[postaction={decorate}] (8,8) -- (9,8);
        \draw[postaction={decorate}] (9,8) -- (10,8);
        \draw[postaction={decorate}] (10,8) -- (11,8);
        \draw[postaction={decorate}] (8,8) -- (8,9);
        \draw[postaction={decorate}] (9,8) -- (9,9);
        \draw[postaction={decorate}] (10,8) -- (10,9);
        \draw[postaction={decorate}] (11,8) -- (11,9);
        \draw[postaction={decorate}] (8,9) -- (9,9);
        \draw[postaction={decorate}] (9,9) -- (10,9);
        \draw[postaction={decorate}] (10,9) -- (11,9);
        \node at (8.5, 8.5) {\tiny$\alpha_{13,12}$};
        \node at (9.5, 8.5) {\tiny$\alpha_{13,22}$};
        \node at (10.5, 8.5) {\tiny$\alpha_{23,22}$};

        % Top
        \draw[postaction={decorate}] (0,11) -- (0,12); % first col
        \draw[postaction={decorate}] (0,12) -- (0,13);
        \draw[postaction={decorate}] (0,13) -- (0,14);
        \draw[postaction={decorate}] (0,14) -- (0,15);
        \draw[postaction={decorate}] (0,15) -- (0,16);
        \draw[postaction={decorate}] (0,16) -- (0,17);

        \draw[postaction={decorate}] (0,11) -- (1,11); % first horiz
        \draw[postaction={decorate}] (0,12) -- (1,12);
        \draw[postaction={decorate}] (0,13) -- (1,13);
        \draw[postaction={decorate}] (0,14) -- (1,14);
        \draw[postaction={decorate}] (0,15) -- (1,15);
        \draw[postaction={decorate}] (0,16) -- (1,16);
        \draw[postaction={decorate}] (0,17) -- (1,17);

        \draw[postaction={decorate}] (1,11) -- (1,12); % second col
        \draw[postaction={decorate}] (1,12) -- (1,13);
        \draw[postaction={decorate}] (1,13) -- (1,14);
        \draw[postaction={decorate}] (1,14) -- (1,15);
        \draw[postaction={decorate}] (1,15) -- (1,16);
        \draw[postaction={decorate}] (1,16) -- (1,17);

        \draw[postaction={decorate}] (1,11) -- (2,11); % Second rows
        \draw[postaction={decorate}] (1,12) -- (2,12);
        \draw[postaction={decorate}] (1,13) -- (2,13);
        \draw[postaction={decorate}] (2,14) -- (1,14);

        \draw[postaction={decorate}] (2,11) -- (2,12); % Third Col
        \draw[postaction={decorate}] (2,12) -- (2,13);
        \draw[postaction={decorate}] (2,13) -- (2,14);

        \draw[postaction={decorate}] (2,11) -- (3,11); % Third Rows
        \draw[postaction={decorate}] (2,12) -- (3,12);
        \draw[postaction={decorate}] (3,13) -- (2,13);

        \draw[postaction={decorate}] (3,11) -- (3,12); % Final Col
        \draw[postaction={decorate}] (3,12) -- (3,13);
        \node at (0.5, 11.5) {\tiny$\alpha_{14, 12}$};
        \node at (0.5, 12.5) {\tiny$\alpha_{14, 13}$};
        \node at (0.5, 13.5) {\tiny$\alpha_{14, 13, 22} $};
        \node at (0.5, 14.5) {\tiny$\alpha_{14, 13, 22, 22} $};
        \node at (0.5, 15.5) {\tiny$\alpha_{14, 13, 23, 22} $};
        \node at (0.5, 16.5) {\tiny$\alpha_{14, 14, 23, 22} $};

        \node at (1.5, 11.5) {\tiny$\alpha_{14, 22}$};
        \node at (1.5, 12.5) {\tiny$\alpha_{14, 23} $};
        \node at (1.5, 13.5) {\tiny$\alpha_{14, 23, 22} $};

        \node at (2.5, 11.5) {\tiny$\alpha_{24, 22}$};
        \node at (2.5, 12.5) {\tiny$\alpha_{24, 23}$};
        \draw[fill=Black] (1, 17) circle[radius=2pt, fill=Black] node[above right] {$F$};
    \end{scope}
\end{tikzpicture}
}

%% file: figures/G2.tex
{\small
    \begin{tikzpicture}[xscale=1.2,thick,decoration={markings,mark=at position 0.9 with {\arrow{>}}}]
        \draw[postaction={decorate}] (0, 1) -- (1, 1);
        \draw[postaction={decorate}] (0, 1) -- (0, 2);
        \draw[postaction={decorate}] (1, 1) -- (1, 2);
        \draw[postaction={decorate}] (0, 2) -- (1, 2);
        \draw[postaction={decorate}] (1, 2) -- (2, 2);
        \draw[postaction={decorate}] (0, 2) -- (0, 3);
        \draw[postaction={decorate}] (1, 2) -- (1, 3);
        \draw[postaction={decorate}] (2, 2) -- (2, 3);
        \draw[postaction={decorate}] (0, 3) -- (1, 3);
        \draw[postaction={decorate}] (2, 3) -- (1, 3);
        \draw[postaction={decorate}] (0, 3) -- (0, 4);
        \draw[postaction={decorate}] (1, 3) -- (1, 4);
        \draw[postaction={decorate}] (0, 4) -- (1, 4);
        \draw[postaction={decorate}] (0, 4) -- (0, 5);
        \draw[postaction={decorate}] (1, 4) -- (1, 5);
        \draw[postaction={decorate}] (0, 5) -- (1, 5);
        \draw[postaction={decorate}] (0, 5) -- (0, 6);
        \draw[postaction={decorate}] (1, 5) -- (1, 6);
        \draw[postaction={decorate}] (0, 6) -- (1, 6);

        \draw[fill=Black] (0, 1) circle[radius=2pt, fill=Black] node[below left] {$I$};
        \draw[fill=Black] (1, 6) circle[radius=2pt, fill=Black] node[above right] {$F$};
        \node at (0.5, 1.5) {$\alpha_{11}$};
        \node at (0.5, 2.5) {$\alpha_{12}$};
        \node at (1.5, 2.5) {$\alpha_{22}$};
        \node at (0.5, 3.5) {$\alpha_{12,11}$};
        \node at (0.5, 4.5) {$\alpha_{12,11,11}$};
        \node at (0.5, 5.5) {$\alpha_{12,12,11}$};
    \end{tikzpicture}
}

%% file: figures/E6.tex
\begin{tikzpicture}[scale=0.36,thick,decoration={markings,mark=at position 0.9 with {\arrow[scale=0.6]{>}}}]
    \draw[postaction={decorate}] (0, 0) -- (0, 1);
    \draw[postaction={decorate}] (0, 0) -- (1, 0);
    \draw[postaction={decorate}] (0, 1) -- (0, 2);
    \draw[postaction={decorate}] (0, 1) -- (1, 1);
    \draw[postaction={decorate}] (0, 10) -- (0, 11);
    \draw[postaction={decorate}] (0, 10) -- (1, 10);
    \draw[postaction={decorate}] (0, 11) -- (0, 19);
    \draw[postaction={decorate}] (0, 19) -- (0, 20);
    \draw[postaction={decorate}] (0, 19) -- (1, 19);
    \draw[postaction={decorate}] (0, 20) -- (0, 24);
    \draw[postaction={decorate}] (0, 2) -- (0, 3);
    \draw[postaction={decorate}] (0, 2) -- (1, 2);
    \draw[postaction={decorate}] (0, 24) -- (0, 25);
    \draw[postaction={decorate}] (0, 24) -- (1, 24);
    \draw[postaction={decorate}] (0, 25) -- (0, 35);
    \draw[postaction={decorate}] (0, 3) -- (0, 5);
    \draw[postaction={decorate}] (0, 35) -- (0, 36);
    \draw[postaction={decorate}] (0, 35) -- (1, 35);
    \draw[postaction={decorate}] (0, 36) -- (0, 39);
    \draw[postaction={decorate}] (0, 39) -- (0, 40);
    \draw[postaction={decorate}] (0, 39) -- (1, 39);
    \draw[postaction={decorate}] (0, 40) -- (0, 42);
    \draw[postaction={decorate}] (0, 42) -- (0, 43);
    \draw[postaction={decorate}] (0, 42) -- (1, 42);
    \draw[postaction={decorate}] (0, 43) -- (0, 44);
    \draw[postaction={decorate}] (0, 43) -- (1, 43);
    \draw[postaction={decorate}] (0, 44) -- (0, 45);
    \draw[postaction={decorate}] (0, 44) -- (1, 44);
    \draw[postaction={decorate}] (0, 45) -- (0, 46);
    \draw[postaction={decorate}] (0, 45) -- (1, 45);
    \draw[postaction={decorate}] (0, 46) -- (1, 46);
    \draw[postaction={decorate}] (0, 5) -- (0, 6);
    \draw[postaction={decorate}] (0, 5) -- (1, 5);
    \draw[postaction={decorate}] (0, 6) -- (0, 10);
    \draw[postaction={decorate}] (1, 0) -- (1, 1);
    \draw[postaction={decorate}] (1, 1) -- (1, 2);
    \draw[postaction={decorate}] (1, 1) -- (2, 1);
    \draw[postaction={decorate}] (1, 10) -- (1, 11);
    \draw[postaction={decorate}] (1, 10) -- (2, 10);
    \draw[postaction={decorate}] (1, 11) -- (30, 19);
    \draw[postaction={decorate}] (1, 19) -- (1, 20);
    \draw[postaction={decorate}] (1, 19) -- (2, 19);
    \draw[postaction={decorate}] (1, 20) -- (1, 24);
    \draw[postaction={decorate}] (1, 2) -- (1, 3);
    \draw[postaction={decorate}] (1, 2) -- (2, 2);
    \draw[postaction={decorate}] (1, 24) -- (1, 25);
    \draw[postaction={decorate}] (1, 24) -- (2, 24);
    \draw[postaction={decorate}] (1, 25) -- (12, 35);
    \draw[postaction={decorate}] (1, 3) -- (6, 5);
    \draw[postaction={decorate}] (1, 35) -- (1, 36);
    \draw[postaction={decorate}] (1, 35) -- (2, 35);
    \draw[postaction={decorate}] (1, 36) -- (5, 39);
    \draw[postaction={decorate}] (1, 39) -- (1, 40);
    \draw[postaction={decorate}] (1, 39) -- (2, 39);
    \draw[postaction={decorate}] (1, 40) -- (1, 42);
    \draw[postaction={decorate}] (1, 42) -- (1, 43);
    \draw[postaction={decorate}] (1, 42) -- (2, 42);
    \draw[postaction={decorate}] (1, 43) -- (1, 44);
    \draw[postaction={decorate}] (1, 43) -- (2, 43);
    \draw[postaction={decorate}] (1, 44) -- (1, 45);
    \draw[postaction={decorate}] (1, 45) -- (1, 46);
    \draw[postaction={decorate}] (1, 5) -- (1, 6);
    \draw[postaction={decorate}] (1, 5) -- (2, 5);
    \draw[postaction={decorate}] (1, 6) -- (6, 10);
    \draw[postaction={decorate}] (10, 10) -- (10, 11);
    \draw[postaction={decorate}] (10, 11) -- (9, 11);
    \draw[postaction={decorate}] (10, 19) -- (10, 20);
    \draw[postaction={decorate}] (10, 20) -- (9, 20);
    \draw[postaction={decorate}] (10, 39) -- (10, 40);
    \draw[postaction={decorate}] (10, 39) -- (11, 39);
    \draw[postaction={decorate}] (10, 40) -- (11, 40);
    \draw[postaction={decorate}] (11, 39) -- (11, 40);
    \draw[postaction={decorate}] (11, 39) -- (12, 39);
    \draw[postaction={decorate}] (11, 40) -- (12, 40);
    \draw[postaction={decorate}] (12, 10) -- (12, 11);
    \draw[postaction={decorate}] (12, 10) -- (13, 10);
    \draw[postaction={decorate}] (12, 11) -- (12, 19);
    \draw[postaction={decorate}] (12, 19) -- (12, 20);
    \draw[postaction={decorate}] (12, 19) -- (13, 19);
    \draw[postaction={decorate}] (12, 20) -- (13, 20);
    \draw[postaction={decorate}] (12, 24) -- (12, 25);
    \draw[postaction={decorate}] (12, 24) -- (13, 24);
    \draw[postaction={decorate}] (12, 25) -- (24, 35);
    \draw[postaction={decorate}] (24, 35) -- (24, 36);
    \draw[postaction={decorate}] (24, 35) -- (25, 35);
    \draw[postaction={decorate}] (24, 36) -- (13, 39);
    \draw[postaction={decorate}] (12, 39) -- (12, 40);
    \draw[postaction={decorate}] (12, 39) -- (13, 39);
    \draw[postaction={decorate}] (12, 40) -- (13, 40);
    \draw[postaction={decorate}] (12, 5) -- (12, 6);
    \draw[postaction={decorate}] (12, 5) -- (13, 5);
    \draw[postaction={decorate}] (12, 6) -- (47, 10);
    \draw[postaction={decorate}] (13, 10) -- (13, 11);
    \draw[postaction={decorate}] (13, 10) -- (14, 10);
    \draw[postaction={decorate}] (13, 11) -- (42, 19);
    \draw[postaction={decorate}] (13, 19) -- (13, 20);
    \draw[postaction={decorate}] (13, 19) -- (14, 19);
    \draw[postaction={decorate}] (13, 20) -- (14, 20);
    \draw[postaction={decorate}] (13, 24) -- (13, 25);
    \draw[postaction={decorate}] (13, 25) -- (25, 35);
    \draw[postaction={decorate}] (25, 35) -- (25, 36);
    \draw[postaction={decorate}] (25, 36) -- (24, 36);
    \draw[postaction={decorate}] (13, 39) -- (13, 40);
    \draw[postaction={decorate}] (13, 40) -- (3, 42);
    \draw[postaction={decorate}] (13, 5) -- (13, 6);
    \draw[postaction={decorate}] (13, 5) -- (14, 5);
    \draw[postaction={decorate}] (13, 6) -- (53, 10);
    \draw[postaction={decorate}] (14, 10) -- (14, 11);
    \draw[postaction={decorate}] (14, 10) -- (15, 10);
    \draw[postaction={decorate}] (14, 11) -- (13, 11);
    \draw[postaction={decorate}] (14, 19) -- (14, 20);
    \draw[postaction={decorate}] (14, 19) -- (15, 19);
    \draw[postaction={decorate}] (14, 20) -- (12, 24);
    \draw[postaction={decorate}] (14, 5) -- (14, 6);
    \draw[postaction={decorate}] (14, 6) -- (59, 10);
    \draw[postaction={decorate}] (15, 10) -- (15, 11);
    \draw[postaction={decorate}] (15, 11) -- (14, 11);
    \draw[postaction={decorate}] (15, 19) -- (15, 20);
    \draw[postaction={decorate}] (15, 19) -- (16, 19);
    \draw[postaction={decorate}] (15, 20) -- (13, 24);
    \draw[postaction={decorate}] (15, 24) -- (15, 25);
    \draw[postaction={decorate}] (15, 24) -- (16, 24);
    \draw[postaction={decorate}] (15, 25) -- (1, 35);
    \draw[postaction={decorate}] (16, 19) -- (16, 20);
    \draw[postaction={decorate}] (16, 20) -- (15, 20);
    \draw[postaction={decorate}] (16, 24) -- (16, 25);
    \draw[postaction={decorate}] (16, 24) -- (17, 24);
    \draw[postaction={decorate}] (16, 25) -- (12, 35);
    \draw[postaction={decorate}] (17, 24) -- (17, 25);
    \draw[postaction={decorate}] (17, 24) -- (18, 24);
    \draw[postaction={decorate}] (17, 25) -- (24, 35);
    \draw[postaction={decorate}] (18, 10) -- (18, 11);
    \draw[postaction={decorate}] (18, 10) -- (19, 10);
    \draw[postaction={decorate}] (18, 11) -- (18, 19);
    \draw[postaction={decorate}] (18, 19) -- (18, 20);
    \draw[postaction={decorate}] (18, 19) -- (19, 19);
    \draw[postaction={decorate}] (18, 20) -- (19, 20);
    \draw[postaction={decorate}] (18, 24) -- (18, 25);
    \draw[postaction={decorate}] (18, 25) -- (25, 35);
    \draw[postaction={decorate}] (18, 5) -- (18, 6);
    \draw[postaction={decorate}] (18, 5) -- (19, 5);
    \draw[postaction={decorate}] (18, 6) -- (63, 10);
    \draw[postaction={decorate}] (19, 10) -- (19, 11);
    \draw[postaction={decorate}] (19, 10) -- (20, 10);
    \draw[postaction={decorate}] (19, 11) -- (18, 11);
    \draw[postaction={decorate}] (19, 19) -- (19, 20);
    \draw[postaction={decorate}] (19, 19) -- (20, 19);
    \draw[postaction={decorate}] (19, 20) -- (20, 20);
    \draw[postaction={decorate}] (19, 5) -- (19, 6);
    \draw[postaction={decorate}] (19, 6) -- (69, 10);
    \draw[postaction={decorate}] (2, 1) -- (2, 2);
    \draw[postaction={decorate}] (2, 10) -- (2, 11);
    \draw[postaction={decorate}] (2, 10) -- (3, 10);
    \draw[postaction={decorate}] (2, 11) -- (48, 19);
    \draw[postaction={decorate}] (2, 19) -- (2, 20);
    \draw[postaction={decorate}] (2, 19) -- (3, 19);
    \draw[postaction={decorate}] (2, 20) -- (2, 24);
    \draw[postaction={decorate}] (2, 2) -- (2, 3);
    \draw[postaction={decorate}] (2, 2) -- (3, 2);
    \draw[postaction={decorate}] (2, 24) -- (2, 25);
    \draw[postaction={decorate}] (2, 24) -- (3, 24);
    \draw[postaction={decorate}] (2, 25) -- (24, 35);
    \draw[postaction={decorate}] (2, 3) -- (12, 5);
    \draw[postaction={decorate}] (2, 35) -- (2, 36);
    \draw[postaction={decorate}] (2, 35) -- (3, 35);
    \draw[postaction={decorate}] (2, 36) -- (10, 39);
    \draw[postaction={decorate}] (2, 39) -- (2, 40);
    \draw[postaction={decorate}] (2, 39) -- (3, 39);
    \draw[postaction={decorate}] (2, 40) -- (2, 42);
    \draw[postaction={decorate}] (2, 42) -- (2, 43);
    \draw[postaction={decorate}] (2, 42) -- (3, 42);
    \draw[postaction={decorate}] (2, 43) -- (2, 44);
    \draw[postaction={decorate}] (2, 43) -- (3, 43);
    \draw[postaction={decorate}] (2, 44) -- (1, 44);
    \draw[postaction={decorate}] (2, 5) -- (2, 6);
    \draw[postaction={decorate}] (2, 5) -- (3, 5);
    \draw[postaction={decorate}] (2, 6) -- (12, 10);
    \draw[postaction={decorate}] (20, 10) -- (20, 11);
    \draw[postaction={decorate}] (20, 11) -- (19, 11);
    \draw[postaction={decorate}] (20, 19) -- (20, 20);
    \draw[postaction={decorate}] (20, 19) -- (21, 19);
    \draw[postaction={decorate}] (20, 20) -- (21, 20);
    \draw[postaction={decorate}] (21, 19) -- (21, 20);
    \draw[postaction={decorate}] (21, 19) -- (22, 19);
    \draw[postaction={decorate}] (21, 20) -- (13, 24);
    \draw[postaction={decorate}] (21, 24) -- (21, 25);
    \draw[postaction={decorate}] (21, 24) -- (22, 24);
    \draw[postaction={decorate}] (21, 25) -- (13, 35);
    \draw[postaction={decorate}] (22, 19) -- (22, 20);
    \draw[postaction={decorate}] (22, 20) -- (21, 20);
    \draw[postaction={decorate}] (22, 24) -- (22, 25);
    \draw[postaction={decorate}] (22, 24) -- (23, 24);
    \draw[postaction={decorate}] (22, 25) -- (24, 35);
    \draw[postaction={decorate}] (23, 11) -- (24, 19);
    \draw[postaction={decorate}] (23, 24) -- (23, 25);
    \draw[postaction={decorate}] (23, 25) -- (25, 35);
    \draw[postaction={decorate}] (24, 10) -- (24, 11);
    \draw[postaction={decorate}] (24, 10) -- (25, 10);
    \draw[postaction={decorate}] (24, 11) -- (23, 11);
    \draw[postaction={decorate}] (24, 19) -- (24, 20);
    \draw[postaction={decorate}] (24, 19) -- (25, 19);
    \draw[postaction={decorate}] (24, 20) -- (25, 20);
    \draw[postaction={decorate}] (25, 10) -- (25, 11);
    \draw[postaction={decorate}] (25, 11) -- (24, 11);
    \draw[postaction={decorate}] (25, 19) -- (25, 20);
    \draw[postaction={decorate}] (25, 19) -- (26, 19);
    \draw[postaction={decorate}] (25, 20) -- (26, 20);
    \draw[postaction={decorate}] (26, 10) -- (26, 11);
    \draw[postaction={decorate}] (26, 10) -- (27, 10);
    \draw[postaction={decorate}] (26, 11) -- (7, 19);
    \draw[postaction={decorate}] (26, 19) -- (26, 20);
    \draw[postaction={decorate}] (26, 19) -- (27, 19);
    \draw[postaction={decorate}] (26, 20) -- (27, 20);
    \draw[postaction={decorate}] (27, 10) -- (27, 11);
    \draw[postaction={decorate}] (27, 10) -- (28, 10);
    \draw[postaction={decorate}] (27, 11) -- (37, 19);
    \draw[postaction={decorate}] (27, 19) -- (27, 20);
    \draw[postaction={decorate}] (27, 19) -- (28, 19);
    \draw[postaction={decorate}] (27, 20) -- (13, 24);
    \draw[postaction={decorate}] (27, 24) -- (27, 25);
    \draw[postaction={decorate}] (27, 24) -- (28, 24);
    \draw[postaction={decorate}] (27, 25) -- (28, 25);
    \draw[postaction={decorate}] (28, 10) -- (28, 11);
    \draw[postaction={decorate}] (28, 10) -- (29, 10);
    \draw[postaction={decorate}] (28, 11) -- (55, 19);
    \draw[postaction={decorate}] (28, 19) -- (28, 20);
    \draw[postaction={decorate}] (28, 20) -- (27, 20);
    \draw[postaction={decorate}] (28, 24) -- (28, 25);
    \draw[postaction={decorate}] (28, 25) -- (25, 35);
    \draw[postaction={decorate}] (29, 10) -- (29, 11);
    \draw[postaction={decorate}] (29, 10) -- (30, 10);
    \draw[postaction={decorate}] (29, 11) -- (28, 11);
    \draw[postaction={decorate}] (3, 10) -- (3, 11);
    \draw[postaction={decorate}] (3, 10) -- (4, 10);
    \draw[postaction={decorate}] (3, 11) -- (60, 19);
    \draw[postaction={decorate}] (3, 19) -- (3, 20);
    \draw[postaction={decorate}] (3, 19) -- (4, 19);
    \draw[postaction={decorate}] (3, 20) -- (3, 24);
    \draw[postaction={decorate}] (3, 2) -- (3, 3);
    \draw[postaction={decorate}] (3, 24) -- (3, 25);
    \draw[postaction={decorate}] (3, 25) -- (25, 35);
    \draw[postaction={decorate}] (3, 3) -- (18, 5);
    \draw[postaction={decorate}] (3, 35) -- (3, 36);
    \draw[postaction={decorate}] (3, 36) -- (2, 36);
    \draw[postaction={decorate}] (3, 39) -- (3, 40);
    \draw[postaction={decorate}] (3, 40) -- (3, 42);
    \draw[postaction={decorate}] (3, 42) -- (3, 43);
    \draw[postaction={decorate}] (3, 43) -- (3, 44);
    \draw[postaction={decorate}] (3, 44) -- (2, 44);
    \draw[postaction={decorate}] (3, 5) -- (3, 6);
    \draw[postaction={decorate}] (3, 5) -- (4, 5);
    \draw[postaction={decorate}] (3, 6) -- (18, 10);
    \draw[postaction={decorate}] (30, 10) -- (30, 11);
    \draw[postaction={decorate}] (30, 11) -- (29, 11);
    \draw[postaction={decorate}] (30, 19) -- (30, 20);
    \draw[postaction={decorate}] (30, 19) -- (31, 19);
    \draw[postaction={decorate}] (30, 20) -- (15, 24);
    \draw[postaction={decorate}] (30, 24) -- (30, 25);
    \draw[postaction={decorate}] (30, 24) -- (31, 24);
    \draw[postaction={decorate}] (30, 25) -- (2, 35);
    \draw[postaction={decorate}] (31, 19) -- (31, 20);
    \draw[postaction={decorate}] (31, 19) -- (32, 19);
    \draw[postaction={decorate}] (31, 20) -- (16, 24);
    \draw[postaction={decorate}] (31, 24) -- (31, 25);
    \draw[postaction={decorate}] (31, 24) -- (32, 24);
    \draw[postaction={decorate}] (31, 25) -- (13, 35);
    \draw[postaction={decorate}] (32, 10) -- (32, 11);
    \draw[postaction={decorate}] (32, 10) -- (33, 10);
    \draw[postaction={decorate}] (32, 11) -- (13, 19);
    \draw[postaction={decorate}] (32, 19) -- (32, 20);
    \draw[postaction={decorate}] (32, 19) -- (33, 19);
    \draw[postaction={decorate}] (32, 20) -- (17, 24);
    \draw[postaction={decorate}] (32, 24) -- (32, 25);
    \draw[postaction={decorate}] (32, 24) -- (33, 24);
    \draw[postaction={decorate}] (32, 25) -- (24, 35);
    \draw[postaction={decorate}] (33, 10) -- (33, 11);
    \draw[postaction={decorate}] (33, 10) -- (34, 10);
    \draw[postaction={decorate}] (33, 11) -- (43, 19);
    \draw[postaction={decorate}] (33, 19) -- (33, 20);
    \draw[postaction={decorate}] (33, 19) -- (34, 19);
    \draw[postaction={decorate}] (33, 20) -- (18, 24);
    \draw[postaction={decorate}] (33, 24) -- (33, 25);
    \draw[postaction={decorate}] (33, 25) -- (25, 35);
    \draw[postaction={decorate}] (34, 10) -- (34, 11);
    \draw[postaction={decorate}] (34, 10) -- (35, 10);
    \draw[postaction={decorate}] (34, 11) -- (33, 11);
    \draw[postaction={decorate}] (34, 19) -- (34, 20);
    \draw[postaction={decorate}] (34, 20) -- (33, 20);
    \draw[postaction={decorate}] (35, 10) -- (35, 11);
    \draw[postaction={decorate}] (35, 11) -- (34, 11);
    \draw[postaction={decorate}] (36, 19) -- (36, 20);
    \draw[postaction={decorate}] (36, 19) -- (37, 19);
    \draw[postaction={decorate}] (36, 20) -- (37, 20);
    \draw[postaction={decorate}] (36, 24) -- (36, 25);
    \draw[postaction={decorate}] (36, 24) -- (37, 24);
    \draw[postaction={decorate}] (36, 25) -- (14, 35);
    \draw[postaction={decorate}] (37, 19) -- (37, 20);
    \draw[postaction={decorate}] (37, 19) -- (38, 19);
    \draw[postaction={decorate}] (37, 20) -- (21, 24);
    \draw[postaction={decorate}] (37, 24) -- (37, 25);
    \draw[postaction={decorate}] (37, 24) -- (38, 24);
    \draw[postaction={decorate}] (37, 25) -- (38, 25);
    \draw[postaction={decorate}] (38, 10) -- (38, 11);
    \draw[postaction={decorate}] (38, 10) -- (39, 10);
    \draw[postaction={decorate}] (38, 11) -- (19, 19);
    \draw[postaction={decorate}] (38, 19) -- (38, 20);
    \draw[postaction={decorate}] (38, 19) -- (39, 19);
    \draw[postaction={decorate}] (38, 20) -- (22, 24);
    \draw[postaction={decorate}] (38, 24) -- (38, 25);
    \draw[postaction={decorate}] (38, 25) -- (25, 35);
    \draw[postaction={decorate}] (39, 10) -- (39, 11);
    \draw[postaction={decorate}] (39, 10) -- (40, 10);
    \draw[postaction={decorate}] (39, 11) -- (38, 11);
    \draw[postaction={decorate}] (39, 19) -- (39, 20);
    \draw[postaction={decorate}] (39, 19) -- (40, 19);
    \draw[postaction={decorate}] (39, 20) -- (23, 24);
    \draw[postaction={decorate}] (4, 10) -- (4, 11);
    \draw[postaction={decorate}] (4, 10) -- (5, 10);
    \draw[postaction={decorate}] (4, 11) -- (3, 11);
    \draw[postaction={decorate}] (4, 19) -- (4, 20);
    \draw[postaction={decorate}] (4, 20) -- (3, 20);
    \draw[postaction={decorate}] (4, 5) -- (4, 6);
    \draw[postaction={decorate}] (4, 6) -- (24, 10);
    \draw[postaction={decorate}] (40, 10) -- (40, 11);
    \draw[postaction={decorate}] (40, 11) -- (39, 11);
    \draw[postaction={decorate}] (40, 19) -- (40, 20);
    \draw[postaction={decorate}] (40, 20) -- (39, 20);
    \draw[postaction={decorate}] (40, 24) -- (40, 25);
    \draw[postaction={decorate}] (40, 24) -- (41, 24);
    \draw[postaction={decorate}] (40, 25) -- (3, 35);
    \draw[postaction={decorate}] (41, 24) -- (41, 25);
    \draw[postaction={decorate}] (41, 24) -- (42, 24);
    \draw[postaction={decorate}] (41, 25) -- (14, 35);
    \draw[postaction={decorate}] (42, 19) -- (42, 20);
    \draw[postaction={decorate}] (42, 19) -- (43, 19);
    \draw[postaction={decorate}] (42, 20) -- (43, 20);
    \draw[postaction={decorate}] (42, 24) -- (42, 25);
    \draw[postaction={decorate}] (42, 24) -- (43, 24);
    \draw[postaction={decorate}] (42, 25) -- (43, 25);
    \draw[postaction={decorate}] (43, 11) -- (25, 19);
    \draw[postaction={decorate}] (43, 19) -- (43, 20);
    \draw[postaction={decorate}] (43, 19) -- (44, 19);
    \draw[postaction={decorate}] (43, 20) -- (44, 20);
    \draw[postaction={decorate}] (43, 24) -- (43, 25);
    \draw[postaction={decorate}] (43, 25) -- (25, 35);
    \draw[postaction={decorate}] (44, 10) -- (44, 11);
    \draw[postaction={decorate}] (44, 10) -- (45, 10);
    \draw[postaction={decorate}] (44, 11) -- (43, 11);
    \draw[postaction={decorate}] (44, 19) -- (44, 20);
    \draw[postaction={decorate}] (44, 19) -- (45, 19);
    \draw[postaction={decorate}] (44, 20) -- (27, 24);
    \draw[postaction={decorate}] (45, 10) -- (45, 11);
    \draw[postaction={decorate}] (45, 11) -- (44, 11);
    \draw[postaction={decorate}] (45, 19) -- (45, 20);
    \draw[postaction={decorate}] (45, 19) -- (46, 19);
    \draw[postaction={decorate}] (45, 20) -- (28, 24);
    \draw[postaction={decorate}] (46, 19) -- (46, 20);
    \draw[postaction={decorate}] (46, 20) -- (45, 20);
    \draw[postaction={decorate}] (47, 10) -- (47, 11);
    \draw[postaction={decorate}] (47, 10) -- (48, 10);
    \draw[postaction={decorate}] (47, 11) -- (14, 19);
    \draw[postaction={decorate}] (48, 10) -- (48, 11);
    \draw[postaction={decorate}] (48, 10) -- (49, 10);
    \draw[postaction={decorate}] (48, 11) -- (44, 19);
    \draw[postaction={decorate}] (48, 19) -- (48, 20);
    \draw[postaction={decorate}] (48, 19) -- (49, 19);
    \draw[postaction={decorate}] (48, 20) -- (30, 24);
    \draw[postaction={decorate}] (49, 10) -- (49, 11);
    \draw[postaction={decorate}] (49, 10) -- (50, 10);
    \draw[postaction={decorate}] (49, 11) -- (48, 11);
    \draw[postaction={decorate}] (49, 19) -- (49, 20);
    \draw[postaction={decorate}] (49, 19) -- (50, 19);
    \draw[postaction={decorate}] (49, 20) -- (31, 24);
    \draw[postaction={decorate}] (5, 10) -- (5, 11);
    \draw[postaction={decorate}] (5, 11) -- (4, 11);
    \draw[postaction={decorate}] (5, 39) -- (5, 40);
    \draw[postaction={decorate}] (5, 39) -- (6, 39);
    \draw[postaction={decorate}] (5, 40) -- (6, 40);
    \draw[postaction={decorate}] (50, 10) -- (50, 11);
    \draw[postaction={decorate}] (50, 11) -- (49, 11);
    \draw[postaction={decorate}] (50, 19) -- (50, 20);
    \draw[postaction={decorate}] (50, 19) -- (51, 19);
    \draw[postaction={decorate}] (50, 20) -- (32, 24);
    \draw[postaction={decorate}] (51, 19) -- (51, 20);
    \draw[postaction={decorate}] (51, 19) -- (52, 19);
    \draw[postaction={decorate}] (51, 20) -- (33, 24);
    \draw[postaction={decorate}] (52, 19) -- (52, 20);
    \draw[postaction={decorate}] (52, 20) -- (51, 20);
    \draw[postaction={decorate}] (53, 10) -- (53, 11);
    \draw[postaction={decorate}] (53, 10) -- (54, 10);
    \draw[postaction={decorate}] (53, 11) -- (20, 19);
    \draw[postaction={decorate}] (54, 10) -- (54, 11);
    \draw[postaction={decorate}] (54, 10) -- (55, 10);
    \draw[postaction={decorate}] (54, 11) -- (53, 11);
    \draw[postaction={decorate}] (54, 19) -- (54, 20);
    \draw[postaction={decorate}] (54, 19) -- (55, 19);
    \draw[postaction={decorate}] (54, 20) -- (55, 20);
    \draw[postaction={decorate}] (55, 10) -- (55, 11);
    \draw[postaction={decorate}] (55, 11) -- (54, 11);
    \draw[postaction={decorate}] (55, 19) -- (55, 20);
    \draw[postaction={decorate}] (55, 19) -- (56, 19);
    \draw[postaction={decorate}] (55, 20) -- (36, 24);
    \draw[postaction={decorate}] (56, 19) -- (56, 20);
    \draw[postaction={decorate}] (56, 19) -- (57, 19);
    \draw[postaction={decorate}] (56, 20) -- (37, 24);
    \draw[postaction={decorate}] (57, 19) -- (57, 20);
    \draw[postaction={decorate}] (57, 19) -- (58, 19);
    \draw[postaction={decorate}] (57, 20) -- (38, 24);
    \draw[postaction={decorate}] (58, 11) -- (26, 19);
    \draw[postaction={decorate}] (58, 19) -- (58, 20);
    \draw[postaction={decorate}] (58, 20) -- (57, 20);
    \draw[postaction={decorate}] (59, 10) -- (59, 11);
    \draw[postaction={decorate}] (59, 10) -- (60, 10);
    \draw[postaction={decorate}] (59, 11) -- (58, 11);
    \draw[postaction={decorate}] (6, 10) -- (6, 11);
    \draw[postaction={decorate}] (6, 10) -- (7, 10);
    \draw[postaction={decorate}] (6, 11) -- (6, 19);
    \draw[postaction={decorate}] (6, 19) -- (6, 20);
    \draw[postaction={decorate}] (6, 19) -- (7, 19);
    \draw[postaction={decorate}] (6, 20) -- (7, 20);
    \draw[postaction={decorate}] (6, 24) -- (6, 25);
    \draw[postaction={decorate}] (6, 24) -- (7, 24);
    \draw[postaction={decorate}] (6, 25) -- (12, 35);
    \draw[postaction={decorate}] (12, 35) -- (12, 36);
    \draw[postaction={decorate}] (12, 35) -- (13, 35);
    \draw[postaction={decorate}] (12, 36) -- (7, 39);
    \draw[postaction={decorate}] (6, 39) -- (6, 40);
    \draw[postaction={decorate}] (6, 39) -- (7, 39);
    \draw[postaction={decorate}] (6, 40) -- (7, 40);
    \draw[postaction={decorate}] (6, 5) -- (6, 6);
    \draw[postaction={decorate}] (6, 5) -- (7, 5);
    \draw[postaction={decorate}] (6, 6) -- (26, 10);
    \draw[postaction={decorate}] (60, 10) -- (60, 11);
    \draw[postaction={decorate}] (60, 11) -- (59, 11);
    \draw[postaction={decorate}] (60, 19) -- (60, 20);
    \draw[postaction={decorate}] (60, 19) -- (61, 19);
    \draw[postaction={decorate}] (60, 20) -- (40, 24);
    \draw[postaction={decorate}] (61, 19) -- (61, 20);
    \draw[postaction={decorate}] (61, 19) -- (62, 19);
    \draw[postaction={decorate}] (61, 20) -- (41, 24);
    \draw[postaction={decorate}] (62, 19) -- (62, 20);
    \draw[postaction={decorate}] (62, 19) -- (63, 19);
    \draw[postaction={decorate}] (62, 20) -- (42, 24);
    \draw[postaction={decorate}] (63, 10) -- (63, 11);
    \draw[postaction={decorate}] (63, 10) -- (64, 10);
    \draw[postaction={decorate}] (63, 11) -- (21, 19);
    \draw[postaction={decorate}] (63, 19) -- (63, 20);
    \draw[postaction={decorate}] (63, 19) -- (64, 19);
    \draw[postaction={decorate}] (63, 20) -- (43, 24);
    \draw[postaction={decorate}] (64, 10) -- (64, 11);
    \draw[postaction={decorate}] (64, 10) -- (65, 10);
    \draw[postaction={decorate}] (64, 11) -- (63, 11);
    \draw[postaction={decorate}] (64, 19) -- (64, 20);
    \draw[postaction={decorate}] (64, 20) -- (63, 20);
    \draw[postaction={decorate}] (65, 10) -- (65, 11);
    \draw[postaction={decorate}] (65, 11) -- (64, 11);
    \draw[postaction={decorate}] (68, 11) -- (27, 19);
    \draw[postaction={decorate}] (69, 10) -- (69, 11);
    \draw[postaction={decorate}] (69, 10) -- (70, 10);
    \draw[postaction={decorate}] (69, 11) -- (68, 11);
    \draw[postaction={decorate}] (7, 10) -- (7, 11);
    \draw[postaction={decorate}] (7, 10) -- (8, 10);
    \draw[postaction={decorate}] (7, 11) -- (36, 19);
    \draw[postaction={decorate}] (7, 19) -- (7, 20);
    \draw[postaction={decorate}] (7, 19) -- (8, 19);
    \draw[postaction={decorate}] (7, 20) -- (6, 24);
    \draw[postaction={decorate}] (7, 24) -- (7, 25);
    \draw[postaction={decorate}] (7, 24) -- (8, 24);
    \draw[postaction={decorate}] (7, 25) -- (24, 35);
    \draw[postaction={decorate}] (13, 35) -- (13, 36);
    \draw[postaction={decorate}] (13, 35) -- (14, 35);
    \draw[postaction={decorate}] (13, 36) -- (12, 39);
    \draw[postaction={decorate}] (7, 39) -- (7, 40);
    \draw[postaction={decorate}] (7, 39) -- (8, 39);
    \draw[postaction={decorate}] (7, 40) -- (2, 42);
    \draw[postaction={decorate}] (7, 5) -- (7, 6);
    \draw[postaction={decorate}] (7, 5) -- (8, 5);
    \draw[postaction={decorate}] (7, 6) -- (32, 10);
    \draw[postaction={decorate}] (70, 10) -- (70, 11);
    \draw[postaction={decorate}] (70, 11) -- (69, 11);
    \draw[postaction={decorate}] (8, 10) -- (8, 11);
    \draw[postaction={decorate}] (8, 10) -- (9, 10);
    \draw[postaction={decorate}] (8, 11) -- (54, 19);
    \draw[postaction={decorate}] (8, 19) -- (8, 20);
    \draw[postaction={decorate}] (8, 19) -- (9, 19);
    \draw[postaction={decorate}] (8, 20) -- (7, 24);
    \draw[postaction={decorate}] (8, 24) -- (8, 25);
    \draw[postaction={decorate}] (8, 25) -- (25, 35);
    \draw[postaction={decorate}] (14, 35) -- (14, 36);
    \draw[postaction={decorate}] (14, 36) -- (13, 36);
    \draw[postaction={decorate}] (8, 39) -- (8, 40);
    \draw[postaction={decorate}] (8, 40) -- (3, 42);
    \draw[postaction={decorate}] (8, 5) -- (8, 6);
    \draw[postaction={decorate}] (8, 5) -- (9, 5);
    \draw[postaction={decorate}] (8, 6) -- (38, 10);
    \draw[postaction={decorate}] (9, 10) -- (10, 10);
    \draw[postaction={decorate}] (9, 10) -- (9, 11);
    \draw[postaction={decorate}] (9, 11) -- (8, 11);
    \draw[postaction={decorate}] (9, 19) -- (10, 19);
    \draw[postaction={decorate}] (9, 19) -- (9, 20);
    \draw[postaction={decorate}] (9, 20) -- (8, 24);
    \draw[postaction={decorate}] (9, 5) -- (9, 6);
    \draw[postaction={decorate}] (9, 6) -- (44, 10);

    \draw[dashed] (0, 3) -- (1, 3);
    \draw[dashed] (1, 3) -- (2, 3);
    \draw[dashed] (2, 3) -- (3, 3);

    \draw[dashed] (0, 6) -- (1, 6);
    \draw[dashed] (1, 6) -- (2, 6);
    \draw[dashed] (2, 6) -- (3, 6);
    \draw[dashed] (3, 6) -- (4, 6);
    \draw[dashed] (6, 6) -- (7, 6);
    \draw[dashed] (7, 6) -- (8, 6);
    \draw[dashed] (8, 6) -- (9, 6);
    \draw[dashed] (12, 6) -- (13, 6);
    \draw[dashed] (13, 6) -- (14, 6);
    \draw[dashed] (18, 6) -- (19, 6);

    \draw[dashed] (0, 11) -- (1, 11);
    \draw[dashed] (1, 11) -- (2, 11);
    \draw[dashed] (2, 11) -- (3, 11);
    \draw[dashed] (6, 11) -- (7, 11);
    \draw[dashed] (7, 11) -- (8, 11);
    \draw[dashed] (12, 11) -- (13, 11); 
    \draw[dashed] (26, 11) -- (27, 11);
    \draw[dashed] (27, 11) -- (28, 11);
    \draw[dashed] (32, 11) -- (33, 11);
    \draw[dashed] (47, 11) -- (48, 11);

    \draw[dashed] (0, 20) -- (1, 20);
    \draw[dashed] (1, 20) -- (2, 20);
    \draw[dashed] (2, 20) -- (3, 20);
    \draw[dashed] (7, 20) -- (8, 20);
    \draw[dashed] (8, 20) -- (9, 20);
    \draw[dashed] (14, 20) -- (15, 20);
    \draw[dashed] (30, 20) -- (31, 20);
    \draw[dashed] (31, 20) -- (32, 20);
    \draw[dashed] (32, 20) -- (33, 20);
    \draw[dashed] (37, 20) -- (38, 20);
    \draw[dashed] (38, 20) -- (39, 20);
    \draw[dashed] (44, 20) -- (45, 20);
    \draw[dashed] (48, 20) -- (49, 20);
    \draw[dashed] (49, 20) -- (50, 20);
    \draw[dashed] (50, 20) -- (51, 20);
    \draw[dashed] (55, 20) -- (56, 20);
    \draw[dashed] (56, 20) -- (57, 20);
    \draw[dashed] (60, 20) -- (61, 20);
    \draw[dashed] (61, 20) -- (62, 20);
    \draw[dashed] (62, 20) -- (63, 20);

    \draw[dashed] (0, 25) -- (1, 25);
    \draw[dashed] (1, 25) -- (2, 25);
    \draw[dashed] (2, 25) -- (3, 25);
    \draw[dashed] (6, 25) -- (7, 25);
    \draw[dashed] (7, 25) -- (8, 25);
    \draw[dashed] (12, 25) -- (13, 25);
    \draw[dashed] (15, 25) -- (16, 25);
    \draw[dashed] (16, 25) -- (17, 25);
    \draw[dashed] (17, 25) -- (18, 25);
    \draw[dashed] (21, 25) -- (22, 25);
    \draw[dashed] (22, 25) -- (23, 25);
    \draw[dashed] (30, 25) -- (31, 25);
    \draw[dashed] (31, 25) -- (32, 25);
    \draw[dashed] (32, 25) -- (33, 25);
    \draw[dashed] (36, 25) -- (37, 25);
    \draw[dashed] (40, 25) -- (41, 25);
    \draw[dashed] (41, 25) -- (42, 25);

    \draw[dashed] (0, 36) -- (1, 36);
    \draw[dashed] (1, 36) -- (2, 36);
    \draw[dashed] (12, 36) -- (13, 36);

    \draw[dashed] (0, 40) -- (1, 40);
    \draw[dashed] (1, 40) -- (2, 40);
    \draw[dashed] (2, 40) -- (3, 40);
    \draw[dashed] (7, 40) -- (8, 40);

    % -0.5 +0.5
    %      + 6
    % + 6
    \node at (0.5, 0.5) {{\tiny$1$}};
    \node at (1.5, 1.5) {{\tiny$2$}};
    \node at (3.5, 19.5) {{\tiny$3$}};
    \node at (9.5, 19.5) {{\tiny$3$}};
    \node at (15.5, 19.5) {{\tiny$3$}};
    \node at (21.5, 19.5) {{\tiny$3$}};
    \node at (27.5, 19.5) {{\tiny$3$}};
    \node at (33.5, 19.5) {{\tiny$3$}};
    \node at (39.5, 19.5) {{\tiny$3$}};
    \node at (45.5, 19.5) {{\tiny$3$}};
    \node at (51.5, 19.5) {{\tiny$3$}};
    \node at (57.5, 19.5) {{\tiny$3$}};
    \node at (63.5, 19.5) {{\tiny$3$}};
    \node at (2.5, 2.5) {{\tiny$4$}};
    \node at (3.5, 5.5) {{\tiny$5$}};
    \node at (8.5, 5.5) {{\tiny$5$}};
    \node at (13.5, 5.5) {{\tiny$5$}};
    \node at (18.5, 5.5) {{\tiny$5$}};
    \node at (4.5, 10.5) {{\tiny$6$}};
    \node at (9.5, 10.5) {{\tiny$6$}};
    \node at (14.5, 10.5) {{\tiny$6$}};
    \node at (19.5, 10.5) {{\tiny$6$}};
    \node at (24.5, 10.5) {{\tiny$6$}};
    \node at (29.5, 10.5) {{\tiny$6$}};
    \node at (34.5, 10.5) {{\tiny$6$}};
    \node at (39.5, 10.5) {{\tiny$6$}};
    \node at (44.5, 10.5) {{\tiny$6$}};
    \node at (49.5, 10.5) {{\tiny$6$}};
    \node at (54.5, 10.5) {{\tiny$6$}};
    \node at (59.5, 10.5) {{\tiny$6$}};
    \node at (64.5, 10.5) {{\tiny$6$}};
    \node at (69.5, 10.5) {{\tiny$6$}};
    \node at (0.5, 1.5) {{\tiny$7$}};
    \node at (1.5, 2.5) {{\tiny$8$}};
    \node at (2.5, 19.5) {{\tiny$9$}};
    \node at (8.5, 19.5) {{\tiny$9$}};
    \node at (14.5, 19.5) {{\tiny$9$}};
    \node at (20.5, 19.5) {{\tiny$9$}};
    \node at (26.5, 19.5) {{\tiny$9$}};
    \node at (32.5, 19.5) {{\tiny$9$}};
    \node at (38.5, 19.5) {{\tiny$9$}};
    \node at (44.5, 19.5) {{\tiny$9$}};
    \node at (50.5, 19.5) {{\tiny$9$}};
    \node at (56.5, 19.5) {{\tiny$9$}};
    \node at (62.5, 19.5) {{\tiny$9$}};
    \node at (2.5, 5.5) {{\tiny$10$}};
    \node at (7.5, 5.5) {{\tiny$10$}};
    \node at (12.5, 5.5) {{\tiny$10$}};
    \node at (3.5, 10.5) {{\tiny$11$}};
    \node at (8.5, 10.5) {{\tiny$11$}};
    \node at (13.5, 10.5) {{\tiny$11$}};
    \node at (18.5, 10.5) {{\tiny$11$}};
    \node at (28.5, 10.5) {{\tiny$11$}};
    \node at (33.5, 10.5) {{\tiny$11$}};
    \node at (38.5, 10.5) {{\tiny$11$}};
    \node at (48.5, 10.5) {{\tiny$11$}};
    \node at (53.5, 10.5) {{\tiny$11$}};
    \node at (63.5, 10.5) {{\tiny$11$}};
    \node at (0.5, 2.5) {{\tiny$12$}};
    \node at (1.5, 19.5) {{\tiny$13$}};
    \node at (7.5, 19.5) {{\tiny$13$}};
    \node at (13.5, 19.5) {{\tiny$13$}};
    \node at (19.5, 19.5) {{\tiny$13$}};
    \node at (25.5, 19.5) {{\tiny$13$}};
    \node at (31.5, 19.5) {{\tiny$13$}};
    \node at (37.5, 19.5) {{\tiny$13$}};
    \node at (43.5, 19.5) {{\tiny$13$}};
    \node at (49.5, 19.5) {{\tiny$13$}};
    \node at (55.5, 19.5) {{\tiny$13$}};
    \node at (61.5, 19.5) {{\tiny$13$}};
    \node at (1.5, 5.5) {{\tiny$14$}};
    \node at (6.5, 5.5) {{\tiny$14$}};
    \node at (2.5, 24.5) {{\tiny$15$}};
    \node at (7.5, 24.5) {{\tiny$15$}};
    \node at (12.5, 24.5) {{\tiny$15$}};
    \node at (17.5, 24.5) {{\tiny$15$}};
    \node at (22.5, 24.5) {{\tiny$15$}};
    \node at (27.5, 24.5) {{\tiny$15$}};
    \node at (32.5, 24.5) {{\tiny$15$}};
    \node at (37.5, 24.5) {{\tiny$15$}};
    \node at (42.5, 24.5) {{\tiny$15$}};
    \node at (2.5, 10.5) {{\tiny$16$}};
    \node at (7.5, 10.5) {{\tiny$16$}};
    \node at (12.5, 10.5) {{\tiny$16$}};
    \node at (27.5, 10.5) {{\tiny$16$}};
    \node at (32.5, 10.5) {{\tiny$16$}};
    \node at (47.5, 10.5) {{\tiny$16$}};
    \node at (0.5, 19.5) {{\tiny$17$}};
    \node at (6.5, 19.5) {{\tiny$17$}};
    \node at (12.5, 19.5) {{\tiny$17$}};
    \node at (18.5, 19.5) {{\tiny$17$}};
    \node at (24.5, 19.5) {{\tiny$17$}};
    \node at (30.5, 19.5) {{\tiny$17$}};
    \node at (36.5, 19.5) {{\tiny$17$}};
    \node at (42.5, 19.5) {{\tiny$17$}};
    \node at (48.5, 19.5) {{\tiny$17$}};
    \node at (54.5, 19.5) {{\tiny$17$}};
    \node at (60.5, 19.5) {{\tiny$17$}};
    \node at (0.5, 5.5) {{\tiny$18$}};
    \node at (1.5, 24.5) {{\tiny$19$}};
    \node at (6.5, 24.5) {{\tiny$19$}};
    \node at (16.5, 24.5) {{\tiny$19$}};
    \node at (21.5, 24.5) {{\tiny$19$}};
    \node at (31.5, 24.5) {{\tiny$19$}};
    \node at (36.5, 24.5) {{\tiny$19$}};
    \node at (41.5, 24.5) {{\tiny$19$}};
    \node at (1.5, 10.5) {{\tiny$20$}};
    \node at (6.5, 10.5) {{\tiny$20$}};
    \node at (26.5, 10.5) {{\tiny$20$}};
    \node at (2.5, 35.5) {{\tiny$21$}};
    \node at (13.5, 35.5) {{\tiny$21$}};
    \node at (24.5, 35.5) {{\tiny$21$}};
    \node at (0.5, 24.5) {{\tiny$22$}};
    \node at (15.5, 24.5) {{\tiny$22$}};
    \node at (30.5, 24.5) {{\tiny$22$}};
    \node at (40.5, 24.5) {{\tiny$22$}};
    \node at (0.5, 10.5) {{\tiny$24$}};
    \node at (1.5, 35.5) {{\tiny$25$}};
    \node at (12.5, 35.5) {{\tiny$25$}};
    \node at (1.5, 39.5) {{\tiny$26$}};
    \node at (6.5, 39.5) {{\tiny$26$}};
    \node at (11.5, 39.5) {{\tiny$26$}};
    \node at (0.5, 35.5) {{\tiny$27$}};
    \node at (2.5, 42.5) {{\tiny$28$}};
    \node at (0.5, 39.5) {{\tiny$29$}};
    \node at (5.5, 39.5) {{\tiny$29$}};
    \node at (10.5, 39.5) {{\tiny$29$}};
    \node at (1.5, 42.5) {{\tiny$30$}};
    \node at (2.5, 43.5) {{\tiny$31$}};
    \node at (0.5, 42.5) {{\tiny$32$}};
    \node at (1.5, 43.5) {{\tiny$33$}};
    \node at (0.5, 43.5) {{\tiny$34$}};
    \node at (0.5, 44.5) {{\tiny$35$}};
    \node at (0.5, 45.5) {{\tiny$36$}};
    \node at (2.5, 39.5) {{\tiny$23$}};
    \node at (7.5, 39.5) {{\tiny$23$}};
    \node at (12.5, 39.5) {{\tiny$23$}};

\end{tikzpicture}

%% file: figures/E6root.tex
\begin{tikzpicture}[yscale=1.5, thick,decoration={markings,mark=at position 0.9 with {\arrow{>}}}]
    \begin{scope}[shift={(0,0)}, xscale=1.6]
    \node (a1) at (0, 0) {{\footnotesize$\alpha_{1}$}};
    \node (a2) at (2, 0) {{\footnotesize$\alpha_{3}$}};
    \node (a3) at (4, 0) {{\footnotesize$\alpha_{2}$}};
    \node (a4) at (5, 0) {{\footnotesize$\alpha_{4}$}};
    \node (a5) at (7, 0) {{\footnotesize$\alpha_{5}$}};
    \node (a6) at (9, 0) {{\footnotesize$\alpha_{6}$}};

    \node (a7) at (1, 1) {{\footnotesize$\alpha_{1} + \alpha_{3}$}};
    \node (a8) at (3, 1) {{\footnotesize$\alpha_{3} + \alpha_{4}$}};
    \node (a9) at (5, 1) {{\footnotesize$\alpha_{2} + \alpha_{4}$}};
    \node (a10) at (7, 1) {{\footnotesize$\alpha_{4} + \alpha_{5}$}};
    \node (a11) at (9, 1) {{\footnotesize$\alpha_{5} + \alpha_{6}$}};

    \node (a12) at (1, 2) {{\footnotesize$\alpha_{1} + \alpha_{3} + \alpha_{4}$}};
    \node (a13) at (3, 2) {{\footnotesize$\alpha_{2} + \alpha_{3} + \alpha_{4}$}};
    \node (a14) at (5, 2) {{\footnotesize$\alpha_{3} + \alpha_{4} + \alpha_{5}$}};
    \node (a15) at (7, 2) {{\footnotesize$\alpha_{2} + \alpha_{4} + \alpha_{5}$}};
    \node (a16) at (9, 2) {{\footnotesize$\alpha_{4} + \alpha_{5} + \alpha_{6}$}};

    \node (a17) at (1, 3) {{\footnotesize$\alpha_{1} + \alpha_{2} + \alpha_{3} + \alpha_{4}$}};
    \node (a18) at (3, 3) {{\footnotesize$\alpha_{1} + \alpha_{3} + \alpha_{4} + \alpha_{5}$}};
    \node (a19) at (5, 3) {{\footnotesize$\alpha_{2} + \alpha_{3} + \alpha_{4} + \alpha_{5}$}};
    \node (a20) at (7, 3) {{\footnotesize$\alpha_{3} + \alpha_{4} + \alpha_{5} + \alpha_{6}$}};
    \node (a21) at (9, 3) {{\footnotesize$\alpha_{2} + \alpha_{4} + \alpha_{5} + \alpha_{6}$}};

    \node (a36) at (5, 3.7) {{\footnotesize$\alpha_{2} + \alpha_{3} + 2 \alpha_{4} + \alpha_{5}$}};
    \node (a22) at (2, 4) {{\footnotesize$\alpha_{1} + \alpha_{2} + \alpha_{3} + \alpha_{4} + \alpha_{5}$}};
    \node (a23) at (5, 4.3) {{\footnotesize$\alpha_{1} + \alpha_{3} + \alpha_{4} + \alpha_{5} + \alpha_{6}$}};
    \node (a24) at (8, 4) {{\footnotesize$\alpha_{2} + \alpha_{3} + \alpha_{4} + \alpha_{5} + \alpha_{6}$}};

    \node (a25) at (2, 5) {{\footnotesize$\alpha_{1} + \alpha_{2} + \alpha_{3} + 2 \alpha_{4} + \alpha_{5}$}};
    \node (a26) at (5, 5) {{\footnotesize$\alpha_{1} + \alpha_{2} + \alpha_{3} + \alpha_{4} + \alpha_{5} + \alpha_{6}$}};
    \node (a27) at (8, 5) {{\footnotesize$\alpha_{2} + \alpha_{3} + 2 \alpha_{4} + \alpha_{5} + \alpha_{6}$}};

    \node (a28) at (2, 6) {{\footnotesize$\alpha_{1} + \alpha_{2} + 2 \alpha_{3} + 2 \alpha_{4} + \alpha_{5}$}};
    \node (a29) at (5, 6) {{\footnotesize$\alpha_{1} + \alpha_{2} + \alpha_{3} + 2 \alpha_{4} + \alpha_{5} + \alpha_{6}$}};
    \node (a30) at (8, 6) {{\footnotesize$\alpha_{2} + \alpha_{3} + 2 \alpha_{4} + 2 \alpha_{5} + \alpha_{6}$}};

    \node (a31) at (3.5, 7) {{\footnotesize$\alpha_{1} + \alpha_{2} + 2 \alpha_{3} + 2 \alpha_{4} + \alpha_{5} + \alpha_{6}$}};
    \node (a32) at (6.5, 7) {{\footnotesize$\alpha_{1} + \alpha_{2} + \alpha_{3} + 2 \alpha_{4} + 2 \alpha_{5} + \alpha_{6}$}};

    \node (a33) at (5, 8) {{\footnotesize$\alpha_{1} + \alpha_{2} + 2 \alpha_{3} + 2 \alpha_{4} + 2 \alpha_{5} + \alpha_{6}$}};
    \node (a34) at (5, 9) {{\footnotesize$\alpha_{1} + \alpha_{2} + 2 \alpha_{3} + 3 \alpha_{4} + 2 \alpha_{5} + \alpha_{6}$}};
    \node (a35) at (5, 10) {{\footnotesize$\alpha_{1} + 2\alpha_{2} + 2 \alpha_{3} + 3 \alpha_{4} + 2 \alpha_{5} + \alpha_{6}$}};

    \draw[postaction={decorate}] (a6) -- (a11);
    \draw[postaction={decorate}] (a5) -- (a11);
    \draw[postaction={decorate}] (a5) -- (a10);
    \draw[postaction={decorate}] (a11) -- (a16);
    \draw[postaction={decorate}] (a4) -- (a10);
    \draw[postaction={decorate}] (a4) -- (a8);
    \draw[postaction={decorate}] (a4) -- (a9);
    \draw[postaction={decorate}] (a10) -- (a16);
    \draw[postaction={decorate}, dashed] (a10) -- (a14);
    \draw[postaction={decorate}] (a10) -- (a15);
    \draw[postaction={decorate}] (a16) -- (a21);
    \draw[postaction={decorate}, dashed] (a16) -- (a20);
    \draw[postaction={decorate}] (a2) -- (a7);
    \draw[postaction={decorate}] (a2) -- (a8);
    \draw[postaction={decorate}] (a8) -- (a12);
    \draw[postaction={decorate}, dashed] (a8) -- (a14);
    \draw[postaction={decorate}] (a8) -- (a13);
    \draw[postaction={decorate}, dashed] (a14) -- (a18);
    \draw[postaction={decorate}, dashed] (a14) -- (a20);
    \draw[postaction={decorate}, dashed] (a14) -- (a19);
    \draw[postaction={decorate}, dashed] (a20) -- (a24);
    \draw[postaction={decorate}, dashed] (a20) -- (a23);
    \draw[postaction={decorate}] (a3) -- (a9);
    \draw[postaction={decorate}] (a9) -- (a13);
    \draw[postaction={decorate}] (a9) -- (a15);
    \draw[postaction={decorate}] (a13) -- (a17);
    \draw[postaction={decorate}] (a13) -- (a19);
    \draw[postaction={decorate}] (a15) -- (a21);
    \draw[postaction={decorate}] (a15) -- (a19);
    \draw[postaction={decorate}] (a19) -- (a24);
    \draw[postaction={decorate}] (a19) -- (a22);
    \draw[postaction={decorate}] (a19) -- (a36);
    \draw[postaction={decorate}] (a36) -- (a27);
    \draw[postaction={decorate}] (a36) -- (a25);
    \draw[postaction={decorate}] (a21) -- (a24);
    \draw[postaction={decorate}] (a24) -- (a27);
    \draw[postaction={decorate}, dashed] (a24) -- (a26);
    \draw[postaction={decorate}] (a27) -- (a30);
    \draw[postaction={decorate}] (a27) -- (a29);
    \draw[postaction={decorate}] (a30) -- (a32);
    \draw[postaction={decorate}] (a1) -- (a7);
    \draw[postaction={decorate}] (a7) -- (a12);
    \draw[postaction={decorate}] (a12) -- (a17);
    \draw[postaction={decorate}, dashed] (a12) -- (a18);
    \draw[postaction={decorate}] (a17) -- (a22);
    \draw[postaction={decorate}, dashed] (a18) -- (a22);
    \draw[postaction={decorate}, dashed] (a18) -- (a23);
    \draw[postaction={decorate}] (a22) -- (a25);
    \draw[postaction={decorate}, dashed] (a22) -- (a26);
    \draw[postaction={decorate}] (a25) -- (a28);
    \draw[postaction={decorate}] (a25) -- (a29);
    \draw[postaction={decorate}] (a28) -- (a31);
    \draw[postaction={decorate}, dashed] (a23) -- (a26);
    \draw[postaction={decorate}] (a26) -- (a29);
    \draw[postaction={decorate}] (a29) -- (a31);
    \draw[postaction={decorate}] (a29) -- (a32);
    \draw[postaction={decorate}] (a32) -- (a33);
    \draw[postaction={decorate}] (a31) -- (a33);
    \draw[postaction={decorate}] (a33) -- (a34);
    \draw[postaction={decorate}] (a34) -- (a35);
    \end{scope}
    \begin{scope}[shift={(16,0)}, xscale=1]
    \node (a1) at (0, 0) {{\footnotesize$1$}};
    \node (a2) at (2, 0) {{\footnotesize$2$}};
    \node (a3) at (4, 0) {{\footnotesize$3$}};
    \node (a4) at (5, 0) {{\footnotesize$4$}};
    \node (a5) at (7, 0) {{\footnotesize$5$}};
    \node (a6) at (9, 0) {{\footnotesize$6$}};

    \node (a7) at (1, 1) {{\footnotesize$7$}};
    \node (a8) at (3, 1) {{\footnotesize$8$}};
    \node (a9) at (5, 1) {{\footnotesize$9$}};
    \node (a10) at (7, 1) {{\footnotesize$10$}};
    \node (a11) at (9, 1) {{\footnotesize$11$}};

    \node (a12) at (1, 2) {{\footnotesize$12$}};
    \node (a13) at (3, 2) {{\footnotesize$13$}};
    \node (a14) at (5, 2) {{\footnotesize$14$}};
    \node (a15) at (7, 2) {{\footnotesize$15$}};
    \node (a16) at (9, 2) {{\footnotesize$16$}};

    \node (a17) at (1, 3) {{\footnotesize$17$}};
    \node (a18) at (3, 3) {{\footnotesize$18$}};
    \node (a19) at (5, 3) {{\footnotesize$19$}};
    \node (a20) at (7, 3) {{\footnotesize$20$}};
    \node (a21) at (9, 3) {{\footnotesize$21$}};

    \node (a36) at (5, 3.7) {{\footnotesize$23$}};
    \node (a22) at (2, 4) {{\footnotesize$22$}};
    \node (a23) at (5, 4.3) {{\footnotesize$24$}};
    \node (a24) at (8, 4) {{\footnotesize$25$}};

    \node (a25) at (2, 5) {{\footnotesize$26$}};
    \node (a26) at (5, 5) {{\footnotesize$27$}};
    \node (a27) at (8, 5) {{\footnotesize$28$}};

    \node (a28) at (2, 6) {{\footnotesize$29$}};
    \node (a29) at (5, 6) {{\footnotesize$30$}};
    \node (a30) at (8, 6) {{\footnotesize$31$}};

    \node (a31) at (3.5, 7) {{\footnotesize$32$}};
    \node (a32) at (6.5, 7) {{\footnotesize$33$}};

    \node (a33) at (5, 8) {{\footnotesize$34$}};
    \node (a34) at (5, 9) {{\footnotesize$35$}};
    \node (a35) at (5, 10) {{\footnotesize$36$}};

    \draw[postaction={decorate}] (a6) -- (a11);
    \draw[postaction={decorate}] (a5) -- (a11);
    \draw[postaction={decorate}] (a5) -- (a10);
    \draw[postaction={decorate}] (a11) -- (a16);
    \draw[postaction={decorate}] (a4) -- (a10);
    \draw[postaction={decorate}] (a4) -- (a8);
    \draw[postaction={decorate}] (a4) -- (a9);
    \draw[postaction={decorate}] (a10) -- (a16);
    \draw[postaction={decorate}, dashed] (a10) -- (a14);
    \draw[postaction={decorate}] (a10) -- (a15);
    \draw[postaction={decorate}] (a16) -- (a21);
    \draw[postaction={decorate}, dashed] (a16) -- (a20);
    \draw[postaction={decorate}] (a2) -- (a7);
    \draw[postaction={decorate}] (a2) -- (a8);
    \draw[postaction={decorate}] (a8) -- (a12);
    \draw[postaction={decorate}, dashed] (a8) -- (a14);
    \draw[postaction={decorate}] (a8) -- (a13);
    \draw[postaction={decorate}, dashed] (a14) -- (a18);
    \draw[postaction={decorate}, dashed] (a14) -- (a20);
    \draw[postaction={decorate}, dashed] (a14) -- (a19);
    \draw[postaction={decorate}, dashed] (a20) -- (a24);
    \draw[postaction={decorate}, dashed] (a20) -- (a23);
    \draw[postaction={decorate}] (a3) -- (a9);
    \draw[postaction={decorate}] (a9) -- (a13);
    \draw[postaction={decorate}] (a9) -- (a15);
    \draw[postaction={decorate}] (a13) -- (a17);
    \draw[postaction={decorate}] (a13) -- (a19);
    \draw[postaction={decorate}] (a15) -- (a21);
    \draw[postaction={decorate}] (a15) -- (a19);
    \draw[postaction={decorate}] (a19) -- (a24);
    \draw[postaction={decorate}] (a19) -- (a22);
    \draw[postaction={decorate}] (a19) -- (a36);
    \draw[postaction={decorate}] (a36) -- (a27);
    \draw[postaction={decorate}] (a36) -- (a25);
    \draw[postaction={decorate}] (a21) -- (a24);
    \draw[postaction={decorate}] (a24) -- (a27);
    \draw[postaction={decorate}, dashed] (a24) -- (a26);
    \draw[postaction={decorate}] (a27) -- (a30);
    \draw[postaction={decorate}] (a27) -- (a29);
    \draw[postaction={decorate}] (a30) -- (a32);
    \draw[postaction={decorate}] (a1) -- (a7);
    \draw[postaction={decorate}] (a7) -- (a12);
    \draw[postaction={decorate}] (a12) -- (a17);
    \draw[postaction={decorate}, dashed] (a12) -- (a18);
    \draw[postaction={decorate}] (a17) -- (a22);
    \draw[postaction={decorate}, dashed] (a18) -- (a22);
    \draw[postaction={decorate}, dashed] (a18) -- (a23);
    \draw[postaction={decorate}] (a22) -- (a25);
    \draw[postaction={decorate}, dashed] (a22) -- (a26);
    \draw[postaction={decorate}] (a25) -- (a28);
    \draw[postaction={decorate}] (a25) -- (a29);
    \draw[postaction={decorate}] (a28) -- (a31);
    \draw[postaction={decorate}, dashed] (a23) -- (a26);
    \draw[postaction={decorate}] (a26) -- (a29);
    \draw[postaction={decorate}] (a29) -- (a31);
    \draw[postaction={decorate}] (a29) -- (a32);
    \draw[postaction={decorate}] (a32) -- (a33);
    \draw[postaction={decorate}] (a31) -- (a33);
    \draw[postaction={decorate}] (a33) -- (a34);
    \draw[postaction={decorate}] (a34) -- (a35);
    \end{scope}
\end{tikzpicture}